\documentclass{ws-aa}

\usepackage{amsmath}
\usepackage{amsbsy}  %% for bold face Greek Letters
\usepackage{array}
\usepackage{graphicx}
\usepackage{tkz-euclide}
\usepackage{mathtools}
\usepackage{amsbsy}  %% for bold face Greek Letters
\usepackage[subnum]{cases}

\newcommand{\ds}{\displaystyle}
\newcommand{\e}{{\rm e}}
\newcommand{\dd}{{\rm d}}

\newcommand{\OPS}{orthogonal polynomial sequence}
\newcommand{\Lf}{\mathcal{L}}

\newcommand{\pfq}[3]{{}_{2}F_{1} \left( {#1 \atop #2 }; #3 \right) }
\newcommand{\pFq}[5]{{}_{#1}F_{#2} \left( {#3 \atop #4 }; #5 \right) }
\begin{document}

\markboth{A. F. Loureiro and W. Van Assche}{Threefold symmetric Hahn-classical multiple orthogonal polynomials }

\catchline{}{}{}{}{}
%
%%%%%%%%%%%%%%%%%%%%%%%%%%%%%%%%%%%%%%%%%%%%%%%%%%%%%%%%%%%%%%%%%%%%

\title{THREEFOLD SYMMETRIC HAHN-CLASSICAL MULTIPLE ORTHOGONAL POLYNOMIALS }

\author{ANA F. LOUREIRO%\footnote{Typeset names in 8~pt roman,  uppercase. Use the footnote to indicate thepresent or permanent address of the author.}
}

\address{School of Mathematics, Statistics and Actuarial Science (SMSAS), Sibson Building, 
University of Kent, Canterbury, Kent CT2 7FS (U.K.)\\
a.loureiro@kent.ac.uk (corresponding author)}

\author{WALTER VAN ASSCHE}

\address{Department of Mathematics, KU Leuven, Celestijnenlaan 200B box 2400, BE-3001 Leuven, Belgium\\
walter.vanassche@kuleuven.be}

\maketitle
\date{}

%\begin{history}
%\received{(Day Month Year)}
%\revised{(Day Month Year)}
%%\accepted{(Day Month Year)}
%%\comby{(xxxxxxxxxx)}
%\end{history}

\begin{abstract}
We characterize all the multiple orthogonal threefold symmetric polynomial sequences whose sequence of derivatives is also multiple orthogonal. Such a property is commonly called the Hahn property and it is an extension of the concept of classical polynomials to the context of multiple orthogonality. The emphasis is on the polynomials whose indices lie on the step line, also known as $2$-orthogonal polynomials. We explain the relation of the asymptotic behavior of the recurrence coefficients to that of the largest zero (in absolute value) of the polynomial set. We provide a full characterization of the Hahn-classical orthogonality measures supported on a $3$-star in the complex plane containing all the zeros of the polynomials. There are essentially three distinct families, one of them $2$-orthogonal with respect to two confluent functions of the second kind. 
This paper complements earlier research of Douak and Maroni.  
\end{abstract}



\section{Introduction and motivation}\label{sec:intro}

In this paper we investigate and characterize all the multiple orthogonal polynomials of type II that are threefold symmetric and are such that the polynomial sequence of its derivatives is also a multiple orthogonal sequence of type II. 
Within the standard orthogonality context there are only four families satisfying this property, commonly referred to as the {\it Hahn property}, and they are the Hermite, Laguerre, Jacobi and Bessel,  collectively known as the {\it classical orthogonal polynomials}. These four families of polynomials also share a number of analytic and algebraic properties. Several studies are dedicated to extensions of those properties to the context of multiple orthogonality. However, those extensions give rise to completely different sequences of multiple orthogonal polynomials. For the usual orthogonal polynomials the Hahn property is equivalent with the Bochner characterization (polynomials satisfy a second order differential equation of Sturm-Liouville type) and the existence of a Rodrigues formula. This is no longer true for multiple orthogonal polynomials since there
are families of multiple orthogonal polynomials with a Rodrigues type formula \cite{ABV} that do not satisfy the Hahn property, and various examples of multiple orthogonal polynomials satisfy a higher-order linear recurrence
relation, which is not of Sturm-Liouville type \cite{CouWVA}. Hence characterizations for multiple orthogonal polynomials based on the Hahn property, the Bochner property or the Rodrigues formula give different families of
polynomials. 

A sequence of monic polynomials $\{ P_n\}_{n\geqslant 0}$ with $\deg P_n =n$ is orthogonal with respect to a Borel measure $\mu$ whenever 
$$
	\int P_n(x) x^k \dd\mu(x) = 0 \ if\ k=0,1,\ldots ,n-1, \ n\geq 0
$$
and 
$$
	\int P_n(x) x^n \dd \mu(x) \neq 0 \ for \ n\geqslant 0. 	
$$
Without loss of generality, often we normalize $\mu$ so that it is a probability measure. Obviously, an \OPS\ forms a basis of the vector space of polynomials $\mathcal{P}$. The measures described above can be represented via a linear functional $\Lf$, defined on $\mathcal{P}^{\prime}$, the dual space of $\mathcal{P}$, and it is understood that the action of $\Lf$ over a polynomial $f$ corresponds to $\int f(x) \dd \mu(x)$. Throughout, we denote this action as 
$$
	\langle\Lf,f(x)\rangle  \coloneqq \int f(x) \dd \mu(x). 
$$
The derivative of a function $f$ is denoted by $f'$ or $\frac{\dd }{\dd x}f(x)$. 
Properties on $\mathcal{P}^{\prime}$, such as differentiation or multiplication by a polynomial, can be defined by duality. A detailed explanation can be found in \cite{MarAlg} \cite{MarVar}. In particular, given $g\in\mathcal{P}$ and a functional $\Lf\in\mathcal{P}^\prime$, we define  
\begin{equation} \label{properties functionals}
	\langle g(x)\Lf , f(x) \rangle \coloneqq  \langle \Lf , g(x)f(x) \rangle
	\quad \text{and}\quad 
	\langle \Lf ', f(x) \rangle \coloneqq - \langle \Lf , f' (x)\rangle
\end{equation}
for any polynomial $f$. 

From the definition it is straightforward that an \OPS\ satisfies a second order recurrence relation 
\begin{equation}\label{rec rel OPS}
	P_{n+1}(x) = (x-\beta_n) P_n(x) - \gamma_{n} P_{n-1}(x), \ n\geqslant 1, 
\end{equation}
with $P_0(x) =1$, $P_1(x)=x-\beta_0$ and $\gamma_{n}\neq 0$ for all integers $n\geqslant 1$. This relation is often called the three-term recurrence relation, but we will avoid this terminology as we will be dealing with three-term recurrence relations which are of higher order. There is an important converse of this connection between orthogonal polynomials and second order recurrence relations, known as the Shohat-Favard theorem or spectral theorem for orthogonal polynomials. It states that any sequence of monic polynomials $\{ P_n\}_{n\geqslant 0}$ with $\deg P_n =n$, satisfying the recurrence relation \eqref{rec rel OPS} with $\gamma_n\neq 0$ for all $n\geqslant1$ and initial conditions $P_{-1}(x) =0$ and $P_0(x)=1$, is always an \OPS\ with respect to some measure $\mu$ and, if, in addition, $\beta_n\in\mathbb{R}$ and $\gamma_{n+1}>0$ for all $n\geqslant0$, then $\mu$ is a positive measure on the real line. So, basically, the orthogonality conditions and the second order recurrence relations are two equivalent ways to characterize an \OPS. 

Multiple orthogonal polynomials of type II correspond to a sequence of polynomials of a single variable that satisfy multiple orthogonality conditions with respect to $r>1$ measures. With the multi-index $\vec{n}=(n_1,\ldots,n_r) \in \mathbb{N}^r$, type II multiple orthogonal polynomials correspond to a (multi-index) sequence of monic polynomials $P_{\vec{n}}(x)$ of degree $|\vec{n}| = n_1+\ldots+n_r $ for which there is a vector of measures 
$(\mu_0, \ldots, \mu_{r-1})$ such that   
\begin{equation}\label{multi OPS}
	\int x^k P_{\vec{n}}(x) \dd \mu_j(x) = 0 , \quad 0\leqslant k \leqslant n_j-1
\end{equation}
holds for every $j=0,1,\ldots, r-1$.  By setting $r=1$, we recover the usual orthogonal polynomials. Obviously, \eqref{multi OPS} amounts to the same as saying there exists a vector of $r$ linear functionals $(\Lf_0,\ldots, \Lf_{r-1})$ such that 
$$ 
	\langle \Lf_j, x^k P_{\vec{n}}(x) \rangle = 0 , \quad 0\leqslant k \leqslant n_j-1. 
$$ 
This polynomial $P_{\vec{n}}$ set may not exist or is not unique unless we impose some extra conditions on the $r$ measures $\mu_0,\ldots,\mu_{r-1}$, but
under appropriate conditions the polynomial set satisfies a system of $r$ recurrence relations, relating $P_{\vec{n}}$ with its nearest neighbors 
$P_{\vec{n}+\vec{e}_k}$ and $P_{\vec{n}-\vec{e}_j}$, where $\vec{e}_k$ consists of the $r$-dimensional unit vector with zero entries except for the $k$th component which is $1$ (see \cite{VAneigh11} for further details). The focus of the present work is  on the polynomials whose multi-indices lie on the step-line near the diagonal $\vec{n}=(n,n,\ldots,n)$, and are defined by  
$$ 
	P_{rn+j}(x) = P_{\vec{n}+\vec{s}_j}(x) , \quad \text{with }\quad \vec{s}_j= \sum_{\ell=0}^j \vec{e}_\ell, \qquad j=0,1,\ldots, r-1,
$$ 
that is, 
$$ 
	P_{rn+j}(x) = P_{(\scriptstyle\underbrace{\scriptstyle{ n+1,\ldots,n+1}}_{j},\underbrace{\scriptstyle n,\ldots, n}_{r-j})}(x)
	, \qquad j=0,1,\ldots, r-1. 
$$ 
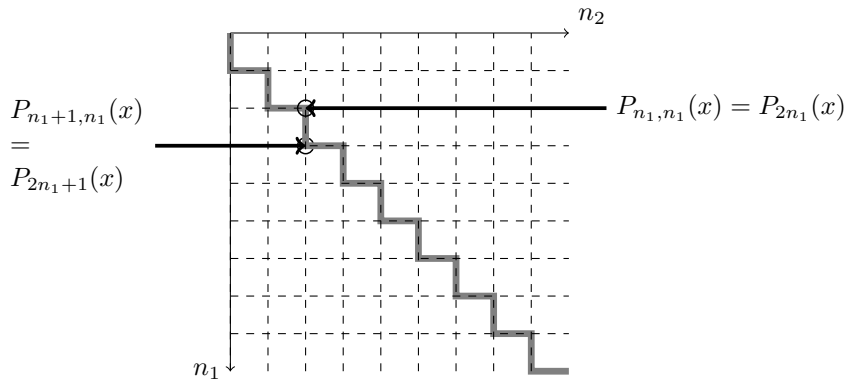
\begin{figure}[ht]
\centering
\begin{tikzpicture}[domain=0:2]
\draw[->] (0,4.5) -- (4.5,4.5)  node[above right] {$n_2$};
\draw[->] (0,4.5)  --  (0,0)  node[left] {$n_1$};
\draw[line width=0.2mm] (1,3) circle (0.7ex) ;
%---
\draw[line width=0.9mm, color=gray] (0,4.5) -- (0,4) -- (0.5,4) -- (0.5, 3.5) -- (1,3.5) -- (1,3) -- (1.5, 3) -- (1.5,2.5) -- (2,2.5) -- (2,2) -- (2.5 , 2) -- (2.5,1.5) -- (3,1.5) -- (3,1) -- (3.5, 1) -- (3.5 , 0.5) -- (4, 0.5) -- (4 , 0)-- (4.5, 0);
\draw[->, line width=0.5mm] (-1,3) -- (1,3) ;
\draw[line width=0.2mm] (-1,3)  node[left]{\parbox{1.8cm}{$P_{n_1+1,n_1}(x)\\=P_{2n_1+1}(x)$}};
\draw[->, line width=0.5mm] (5,3.5) -- (1,3.5) ;
\draw[line width=0.2mm] (5,3.5)  node[right]{$P_{n_1,n_1}(x)=P_{2n_1}(x)$};
\draw[line width=0.2mm] (1,3.5) circle (0.7ex) ;
%---
\draw[-, color=black, dashed] (4,0) -- (4,4.5);
\draw[-, color=black, dashed] (3.5,0) -- (3.5,4.5);
\draw[-, color=black, dashed] (3,0) -- (3,4.5);
\draw[-, color=black, dashed] (2.5,0) -- (2.5,4.5);
\draw[-, color=black, dashed] (2,0) -- (2,4.5);
\draw[-, color=black, dashed] (1.5,0) -- (1.5,4.5);
\draw[-, color=black, dashed] (1,0) -- (1,4.5);
\draw[-, color=black, dashed] (0.5,0) -- (0.5,4.5);
\draw[-, color=black, dashed] (0,0) -- (0,4.5);

%---
\draw[-, color=black, dashed] (0,4) -- (4.5,4);
\draw[-, color=black, dashed] (0,3.5) -- (4.5,3.5);
\draw[-, color=black, dashed] (0,3) -- (4.5,3);
\draw[-, color=black, dashed] (0,2.5) -- (4.5,2.5);
\draw[-, color=black, dashed] (0,2) -- (4.5,2);
\draw[-, color=black, dashed] (0,1.5) -- (4.5,1.5);
\draw[-, color=black, dashed] (0,1) -- (4.5,1);
\draw[-, color=black, dashed] (0,0.5) -- (4.5,0.5);
\end{tikzpicture}
\caption{Step-line for the multi-index  $(n_1,n_2)$  when $r=2$.}
\label{fig:step-line}
\end{figure}
This {\it step-line} polynomial sequence $\{{P}_n\}_{n\geqslant0}$ satisfies a $(r+1)$th order recurrence relation (involving up to $(r+2)$ consecutive terms) and it can be written as 
\begin{equation}\label{recrel dOPS}
\begin{array}{l}
	\ds xP_{n}(x) = P_{n+1}(x) + \sum_{j=0}^{r} \zeta_{n,j} P_{n-j}(x) , \qquad n\geqslant r, \\
	\ds P_j(x)=x^j, \qquad  j=0,\ldots, r ,
\end{array}
\end{equation}
where $\zeta_{n,r}\neq 0$ for all $n\geqslant r$. Such a step-line polynomial sequence $\{{P}_n\}_{n\geqslant0}$ corresponds to so-called {\it $d$-orthogonal polynomials} (with $d=r$), see \cite{Mar89,DM95,Mar99} and \linebreak Figure \ref{fig:step-line} for the case of $r=2$.  So, basically, if there exists a vector of linear functionals $(\Lf_0,\ldots, \Lf_{d-1})$ and a polynomial sequence $\{{P}_n\}_{n\geqslant0}$ satisfying 
$$
	\langle\Lf_j, x^k P_n(x) \rangle = 0  \quad \text{and} \quad \langle\Lf_j, x^{n} P_{dn+j}(x) \rangle \neq 0  , 
	\quad 0 \leqslant k \leqslant \left\lfloor \frac{n-j}{d} \right\rfloor, \ n\geqslant 0,
$$
then the polynomials $\{{P}_n\}_{n\geqslant0}$ are related by \eqref{recrel dOPS}. There is a converse result, which is a natural generalization of the Shohat-Favard theorem, in the sense that if a polynomial sequence $\{{P}_n\}_{n\geqslant0}$ satisfies \eqref{recrel dOPS} with $\zeta_{n,d}\neq 0$ for all $n\geqslant d$, then there is a vector of $r$ linear functionals $(\Lf_0, \ldots, \Lf_{d-1})$ with respect to which $\{{P}_n\}_{n\geqslant0}$ is $d$-orthogonal. Such a vector of linear functionals is not unique, but we can consider its components to be the first $d$ elements of the dual sequence $\{{u}_n\}_{n\geqslant0}$ associated to $\{{P}_n\}_{n\geqslant0}$, which always exist and which are defined by 
$$
	\langle u_n,P_m \rangle \coloneqq \delta_{n,m}	, \quad m,n\geqslant 0, 
$$
where $\delta_{n,m}$ represents the {\it Kronecker symbol}. The remaining elements of the dual sequence of a $d$-\OPS\ can be generated from the first ones: for each pair of integers $(n,j)$ there exist polynomials $q_{n,j,\nu}(x)$ such that \cite{Mar89} 
$$
	u_{dn+j} \ = \ \sum_{\nu=0}^{d-1} q_{n,j,\nu}(x) u_\nu, \quad \text{for} \quad 0\leqslant j\leqslant d-1, 
$$
where $\deg q_{n,j,j}(x) =n$, $\deg q_{n,j,\nu}(x) \leqslant n $ for $0\leqslant \nu \leqslant j-1$ if $1\leqslant j \leqslant d-1$ and $\deg q_{n,j,\nu}(x) \leqslant n -1 $ for $j+1\leqslant \nu\leqslant d-1$ if $0\leqslant j\leqslant d-2$. Unlike the standard orthogonality (case $d=1$), this result only provides structural properties for the dual sequence and in practical terms it can be used in that sense. 

In this paper we mainly deal with $2$-\OPS s, but our results and ideas can be extended to the $d$-orthogonal case. Here, we provide a characterization of all the $2$-\OPS s  that are {\it threefold symmetric} and possess the {\it Hahn property} in the context of $2$-orthogonality, i.e., the sequence of monic derivatives is also $2$-orthogonal. The notion and properties of {\it threefold symmetric} $2$-\OPS s are revised and discussed in Section \ref{sec:2}, while we also bring a new result (Theorem \ref{thm: asym zero}) relating the asymptotic behavior of the recurrence coefficient to the asymptotic behavior of the absolute value of the largest zero. Section \ref{sec:Hahn 2OPS} is dedicated to a detailed characterization of all the threefold symmetric $2$-orthogonal polynomials satisfying the Hahn property. We bring together old and new results in a self-contained and complete analysis to fully characterize this type of sequences in terms of their explicit recurrence relations, weight functions and a differential equation of third order. In that regard, we complete the study initiated by Douak and Maroni in \cite{DM92, DM95} by providing sufficient conditions and taking the study based only on the properties of the zeros in combination with the weights. 
We end up this coherent description with a detailed analysis of the four distinct families of threefold symmetric $2$-orthogonal polynomials satisfying the Hahn property. The first case, treated in Section \ref{subs: caseA}  corresponds to polynomials with the Appell property whose $2$-orthogonality weights are supported on the three-starlike set (as in Fig.\ref{fig: 3 rays}) represented via the Airy function and its derivative. These polynomials have been studied in 
\cite{Douak96}, but the explicit representation of the weights supported on a set containing all the zeros has not been considered there. In Sections \ref{subs: caseB1} and \ref{subs: caseB2} we study the second and third cases where the weights are represented via the confluent hypergeometric Kummer function of 2nd kind. These are, to the best of our knowledge, new. The fourth case is treated in Section \ref{subs: caseC}, where the orthogonality weights are expressed via hypergeometric functions and depend on two parameters. For special choices of those parameters we recover some known polynomial families. However, the existent literature on the subject has not taken into account the explicit representation of the weights supported on a set containing all the zeros and this has been accomplished here for all the threefold symmetric polynomials satisfying the {\it Hahn property} within the context of $2$-orthogonality. 
%  inserted by Walter on request of a referee  (17/05/2019) 
The orthogonality \eqref{multi OPS} of these polynomials corresponds to an non-hermitian inner product (no conjugation is used) on three-starlike sets. The Airy function, the confluent hypergeometric function and the hypergeometric function are analytic functions, so the paths of integration can
be deformed away from the three-starlike sets without loosing orthogonality. However, the choice of the three-starlike sets is convenient because
of the positivity of the measures and the location of the zeros.
It turns out that the orthogonality measures under analysis are solutions to a second order differential equation and, as such, the recurrence coefficients for the whole set of the corresponding multiple orthogonal polynomials of type II can only be obtained algorithmically via the nearest-neighbor algorithm explained in \cite{VAneigh11} \cite{FHVA}. 
Some of the $d$-orthogonal polynomials that we found appear in the theory of random matrices, in particular in the investigation of singular values
of products of Ginibre matrices, which uses multiple orthogonal polynomials with weight functions expressed in terms of Meijer G-functions
\cite{KuijlSti}. For $r=2$ these
Meijer G-functions are hypergeometric or confluent hypergeometric functions.

\section{Threefold 2-\OPS}\label{sec:2}

A sequence of monic polynomials $\{P_n\}_{n\geqslant 0}$ (with $\deg{P_n}=n$) is {\it symmetric} whenever $P_n(-x) = (-1)^n P_n(x)$ for all $n\geqslant0$. This means that all even degree polynomials are even functions while odd degree polynomials are odd functions. Hermite and Gegenbauer polynomials are examples of symmetric polynomial sets, which also happen to be the only classical orthogonal polynomials that are symmetric. Many other examples of symmetric orthogonal polynomial sequences are around in the literature. 

The notion of symmetry of polynomial sequences has been extended and commonly referred to as $d$-symmetric in works by Maroni \cite{Mar89,DM95}  and followed by Ben Cheikh and his collaborators \cite {BCBR11,BRMG16,BR08,LOHerm}. The case $d=1$ would correspond to the usual symmetric case. We believe the name is misleading, as will soon become apparent, and therefore we call it differently as it pictures better the nature of the problem. 
\begin{definition}\label{def:3-fold}
A polynomial sequence $\{P_n\}_{n\geqslant 0}$ is {\bf $\mathbf{m}$-fold symmetric} if 
\begin{equation}\label{sym}
	P_n(\omega x) = \omega^{n} \, P_n(x), \ \text{ for any }\  n\geqslant 0. 
\end{equation}
where 
$
	\omega= \e^{ \frac{2 i  \pi}{m}}. 
$
\end{definition}
By induction, property \eqref{sym} corresponds to 
$$ 
	P_n(\omega^k x) = \omega^{nk} \, P_n(x), \ \text{ for any }\  n\geqslant 0 \ \text{ and }\ k=1,2,\ldots, m-1. 
$$ 
So, an $m$-fold symmetric sequence is a $(m-1)$-symmetric sequence in  \cite{ BRMG16,BR08, BCBR11, DM95,Mar89, LOHerm}. 
(The symmetric sequences of Hermite and Gegenbauer polynomials are examples of {\it twofold symmetric} polynomials.) In particular, a {\it threefold symmetric} sequence $\{P_n\}_{n\geqslant 0}$, is such that 
$$
	P_n(\omega x) =  \omega^n P_n(x) \ \text{ and } \ 
	P_n(\omega^2 x) = \omega^{2n} P_n(x)  \ \text{ with } \ \omega =\e^{ \frac{2 i  \pi}{3}},\ n\geqslant0, 
$$
which corresponds to say that there exist three sequences $\{P_n^{[j]}\}_{n\geqslant 0}$ with $j\in\{0,1,2\}$ such that 
\begin{equation}\label{CD}
	P_{3n+j}(x) = x^j P_n^{[j]}(x^3), \quad j=0,1,2. 
\end{equation}
Throughout, we will refer to $\{P_n^{[j]}\}_{n\geqslant 0}$ as the {\it diagonal components of the cubic decomposition} of the threefold symmetric sequence $\{P_n\}_{n\geqslant 0}$, which is line with the terminology adopted in a more general cubic decomposition framework in \cite{MarMes}. 

In the case of a $2$-\OPS\ $\{ P_n\}_{n\geqslant 0}$ with respect to a vector linear functional ${\bf u} = (u_0,u_1)$, we have, as discussed in Section \ref{sec:intro}, 
\begin{eqnarray} \label{2OPS u0}
	&&  \langle u_0 , x^m P_n \rangle =
		\left\{\begin{array}{lcl}
				0 & \text{for} & n\geqslant 2m+1\\
				N_0(n)\neq 0 & \text{for} & n= 2m\\
		\end{array}\right. \\
	&&  \label{2OPS u1}
	 \langle u_1 , x^m P_n \rangle  =
		\left\{\begin{array}{lcl}
				0 & \text{for} & n\geqslant 2m+2\\
				N_1(n)\neq 0 & \text{for} & n= 2m+1, \\
		\end{array}\right. 
\end{eqnarray}
and there exists a set of coefficients $\{(\beta_n, \alpha_n,\gamma_n)\}_{n\geqslant0}$ such that $\{ P_n\}_{n\geqslant 0}$ satisfies a third order recurrence relation (see \cite{Mar89, VIseg}) 
\begin{equation} \label{rec rel 2OPS}
	P_{n+1}(x) = (x-\beta_n) P_n(x) - \alpha_n P_{n-1}(x) - \gamma_{n-1} P_{n-2}(x) 
\end{equation}
with $P_{-2}(x) = P_{-1}(x) =0 $ and $P_0(x)=1$. Straightforwardly from the definition, one has 
\begin{equation} \label{gammas}
	 \gamma_{2n+1} =\frac{\langle u_0, x^{n+1} P_{2n+2}\rangle}{\langle u_0, x^{n} P_{2n}\rangle} \neq 0,
	\quad  \gamma_{2n+2} = \frac{\langle u_1, x^{n+1} P_{2n+3}\rangle}{\langle u_1, x^{n} P_{2n+1} \rangle} \neq 0 , \ \ n\geqslant 0, 
\end{equation} 
or, equivalently, 
$$
	\langle u_0, x^{n+1} P_{2n+2} \rangle = \prod_{k=0}^n \gamma_{2k+1}
	\quad \text{and }\quad 
	\langle u_1, x^{n+1} P_{2n+3} \rangle = \prod_{k=0}^n \gamma_{2k+2} ,\quad \text{for }\quad n\geqslant 0.
$$
Whenever a $2$-\OPS\ is threefold symmetric, the recurrence relation \eqref{rec rel 2OPS} reduces to a three-term relation, where the $\beta$- and $\alpha$-coefficients all vanish. For this type of $2$-\OPS s  more can be said.  

\begin{proposition}\label{prop:3fold sym}\cite{DM92} Let $\{P_n\}_{n\geqslant 0}$  be a 2-orthogonal polynomial sequence with respect to the linear functional $U=(u_0, u_1)$ satisfying \eqref{2OPS u0}-\eqref{2OPS u1}. The following statements are equivalent: 
\begin{enumerate}
\item[(a)] The sequence $\{P_n\}_{n\geqslant 0}$  is threefold symmetric. 
\item[(b)] The linear functional is threefold symmetric, that is, 
\begin{equation}\label{sym u0 u1}
	(u_\nu)_{3n+j} =0 , \ \text{ for } \ \nu =0,1 \ \text{ and }\  j=1,2 \ \text{ with  }\ j\neq \nu, 
\end{equation}
where $(u)_m \coloneqq \langle u, x^m\rangle$ for $m\geqslant 0$. 
\item[(c)] The sequence $\{P_n\}_{n\geqslant 0}$  satisfies the third order recurrence relation 
\begin{equation}\label{rec rel 3sym}
	P_{n+1}(x) = x P_n(x)  - \gamma_{n-1} P_{n-2}(x) , \qquad n\geqslant 2, 
\end{equation}
with $P_0(x) = 1, \ P_1(x)= x$ and $P_2(x) = x^2$. 
\end{enumerate}
\end{proposition}

Each of the components of the cubic decomposition are also $2$-\OPS s. 

\begin{lemma}\cite{Mar89}\label{lem: 2sym components rec rel} Let $\{P_n\}_{n\geqslant 0}$ be a threefold symmetric $2$-OPS. The three polynomial sequences $\{P_n^{[j]}\}_{n\geqslant 0}$ (with $j=0,1,2$) in the cubic decomposition of $\{P_n\}_{n\geqslant 0}$ described in  \eqref{CD} are 2-orthogonal polynomial sequences satisfying: 
\begin{equation}\label{CD rec rel}
		P_{n+1}^{[j]}(x) = (x-\beta_{n}^{[j]}) P_n^{[j]}(x)  -\alpha_{n}^{[j]} P_{n-1}^{[j]}(x) - \gamma_{n-1}^{[j]} P_{n-2}^{[j]}(x) ,  
\end{equation}
where 
$$\begin{array}{l}
	\beta_n^{[j]} = \gamma_{3n-1+j}+\gamma_{3n+j}+\gamma_{3n+1+j}, \qquad n\geqslant 0,\\
	\alpha_n^{[j]} = \gamma_{3n-2+j}\gamma_{3n+j}+\gamma_{3n-1+j}\gamma_{3n-3+j}+\gamma_{3n-2+j}\gamma_{3n-1 +j}
			, \qquad n\geqslant 1, \\
	\gamma_n^{[j]} = \gamma_{3n-2+j}\gamma_{3n+j}\gamma_{3n+2+j}\neq 0, \qquad n\geqslant 2.\\
	\end{array}
$$
Moreover, $\{P_n^{[j]}\}_{n\geqslant 0}$ is 2-orthogonal with respect to the vector functional $U^{[j]}=(u_0^{[j]},u_1^{[j]})$ with
$$
	\ u_\nu^{[j]} = \sigma_3 (x^j u_{3\nu+j}) \ \text{ for each } \ j=0,1,2 \ \text{ and } \ \nu=0,1. 
$$ 
where $\{u_n\}_{n\geqslant0}$ represents the dual sequence of $\{P_n\}_{n\geqslant 0}$ and $\sigma_3:\mathcal{P}' \longrightarrow \mathcal{P}' $ is the operator defined by $\langle \sigma_3 (v), f(x) \rangle \coloneqq \langle v, f(x^3) \rangle $ for any $v\in \mathcal{P}' $ and $f\in\mathcal{P} $.
\end{lemma}

The orthogonality measures are supported on a starlike set with three rays.

\begin{theorem}\cite{AKVI}\label{thm:3sym 2OPS} If $\gamma_n>0$ for $n\geqslant 1$ in \eqref{rec rel 3sym}, then there exists a vector of linear functionals  $(u_0,u_1)$ such that the polynomials $P_n$ defined by \eqref{rec rel 3sym} satisfy the 2-orthogonal relations \eqref{2OPS u0}-\eqref{2OPS u1}. Moreover, the vector of linear functionals  $(u_0,u_1)$ satisfies \eqref{sym u0 u1} and there exist a vector of two measures $(\mu_0,\mu_1)$ such that 
\begin{eqnarray*}
	&& \langle u_0,f(x) \rangle = \int_S f(x) \text{d} \mu_0(x) \\
	&& \langle u_1,f(x) \rangle = \int_S f(x) \text{d} \mu_1(x) 
\end{eqnarray*}
where $S$ represents the starlike set 
$$
	S\coloneqq \bigcup_{k=0}^2 \Gamma_k\qquad \text{with } \qquad 
	\Gamma_k = [0,\e^{2\pi i k/3} \infty) , 
$$
and the measures have a common support which is a subset of $S$ and are invariant under rotations of $2\pi /3$. 
\end{theorem}

\begin{figure}[h]
\begin{center}
\noindent\fbox{\ 
\begin{tikzpicture}
  \tkzInit[xmin=-10,ymin=-10,xmax=10,ymax=10]
\draw [->,very thick] (-1,1.73) -- (-0.5,0.86)  ;
\draw [-, very thick]  (-0.5,0.86)  -- (0,0)  ;
\node [above] at (-0.9,0.9) {\scriptsize${\Gamma}_1$};
\draw [->, very thick] (-1,-1.73) -- (-0.5,-0.86) ;
\draw [-, very thick] (-0.5,-0.86) -- (0,0)  ;
\node [below] at (-0.9,-0.7) {\scriptsize${\Gamma}_2$};
\draw [->, very thick] (0,0) -- (0.5,0)  ;
\draw [-, very thick] (0.5,0) -- (1,0)  ;
%\draw [->,] (-2,0) -- (2.4,0)  ;
%\draw [->,] (0,-2) -- (0,2)  ;
\node [below] at (1,0) {\scriptsize${\Gamma}_0$};
\end{tikzpicture}\ }\end{center}
\caption{The three rays ${\Gamma}_0, \ {\Gamma}_1$ and ${\Gamma}_2$.}\label{fig: 3 rays}
\end{figure}
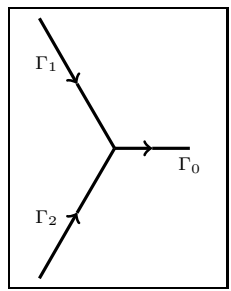

Regarding the behavior of the zeros of any threefold symmetric 2-\OPS $\{P_n\}_{n\geqslant0}$, it was shown in \cite{BR08} that between two nonzero consecutive roots of $P_{n+2}$ there is exactly one root of $P_n$ and one root of $P_{n+1}$ and all those roots lie on the starlike set $S$. 

\begin{proposition}\label{prop:interlace}
Let $\{P_n\}_{n\geqslant0}$ be a 2-OPS satisfying \eqref{rec rel 3sym}. If $\gamma_n >0$, then the following statements hold: 
\begin{enumerate}
\item[(a)] If $x$ is a zero of $P_{3n+j}$, then $\omega^j x$ are also zeros of $P_{3n+j}$ with $\omega=\e^{2\pi i/3}$.
\item[(b)] $0$ is a zero of $P_{3n+j}$ of multiplicity $j$ when $j=1,2$.
\item[(c)] $P_{3n+j}$ has $n$ distinct positive real zeros $\{x_{n,k}^{(j)}\}_{k=1}^n$ with $0<x_{n,1}^{(j)}<\ldots <x_{n,n}^{(j)}$.  
\item[(d)] Between two real zeros of $P_{3n+j+3}$ there exist only one zero of  $P_{3n+j+2}$ and only one zero of $P_{3n+j+1}$, that is, 
$x_{n,k}^{(j+2)}<x_{n,k+1}^{(j)}<x_{n,k+1}^{(j+1)}<x_{n,k+1}^{(j+2)}$. 
\end{enumerate}
\end{proposition}
\begin{proof} The result is a consequence of \cite[Theorem 2.2]{BR08} for the case $d=2$. 
\end{proof}

So, $P_{3n+j}$ with $j\in\{0,1,2\}$ has $n$ zeros $(x_{n,1}^{(j)}, \ldots ,x_{n,n}^{(j)})$  on the positive real line and all the other zeros are obtained by rotations of $2\pi/3$ of $(x_{n,1}^{(j)}, \ldots ,x_{n,n}^{(j)})$ and $0$ is a single zero for $P_{3n+1}$ and a double zero for $P_{3n+2}$. The connection between the asymptotic behavior of the $\gamma$-recurrence coefficients in \eqref{rec rel 3sym}  and the 
upper bound for the largest zero $x_{n,n}^{(j)}$ is discussed in \cite{AKS}, but for bounded recurrence coefficients. Here, we extend that discussion, by embracing the cases where the recurrence coefficients are unbounded with different asymptotic behavior for even and odd order indices, which will be instrumental in Section \ref{sec:Hahn 2OPS}.

\begin{theorem}\label{thm: asym zero} If $\gamma_n $ in \eqref{rec rel 3sym} are positive and, additionally, $\gamma_{2n} = c_0 n^{\alpha} +{o}(n^{\alpha})$ and $\gamma_{2n+1} = c_1 n^{\alpha} +{o}(n^{\alpha})$ for large $n$, with $\min\{c_0,c_1\}\geqslant 0$, $\ c=\max \{c_0,c_1\}>0$ and $\alpha\geqslant 0$, then largest zero in absolute value $|x_{n,n}|$ behaves as 
\begin{equation}\label{larger zero}
	|x_{n,n}| \leqslant \frac{3}{2^{2/3}} c^{1/3} n^{\alpha/3} + {o}(n^{\alpha/3}), \qquad n\geqslant 1.
\end{equation}
\end{theorem}
\begin{proof} Consider the Hessenberg matrix 
$$ 
	\mathbf{H}_n = \left( 
		\begin{array}{ccccccccc}
			0 		 & 1 			& 0 &  \cdots & 0 & 0 & 0 & 0 & 0\\
			0 		 & 0 			& 1 & 0 & \cdots & 0 & 0 & 0 & 0\\
			\gamma_1 & 0 			& 0 & 1 &  0 & \cdots & 0 & 0 & 0 \\
			0 		 & \gamma_2 	& 0 & 0 & 1 & 0 & \cdots & 0 & 0 \\
			    		 &   &   \ddots&    &   &    \ddots   &   &  & \\
					&   &  &   \ddots &   &   &    \ddots   &   &  \\
			   		0 & 0  &   \cdots&  0 &   \gamma_{n-4} &   0&   0& 1& 0  \\
			   		0 & 0  & 0 & \cdots & 0 &  \gamma_{n-3} &   0&   0& 1 \\
			   		 0 &   0&  0&0  & \cdots &  0  &    \gamma_{n-2} &   0& 0  \\
		\end{array}
	\right)
$$ 
so that the recurrence relation \eqref{rec rel 3sym} can be expressed as 
$$
	\mathbf{H}_n \left(\begin{array}{c} P_0(x) \\ P_1(x) \\ \vdots \\ P_{n-1}(x)\end{array}\right)
	= x \left(\begin{array}{c} P_0(x) \\ P_1(x) \\ \vdots \\ P_{n-1}(x)\end{array}\right)
	- P_n(x) \left(\begin{array}{c} 0 \\ 0\\ \vdots \\ 1\end{array}\right)
$$
and each zero of $P_n$ is an eigenvalue of the matrix $\mathbf{H}_n$. 
The spectral radius of the matrix $\mathbf{H}_n$, 
$$
	\rho(\mathbf{H}_n) = \max \{ |\lambda| : \ \lambda \text{ is an eigenvalue of } \mathbf{H}_n\}, 
$$ 
is bounded from above by  $|| \mathbf{H}_n ||$ where $|| \cdot ||$ denotes a matrix norm (see \cite[Section 5.6]{HJ}). 
We take the matrix norm 
$$
	|| \mathbf{H}_n ||_S =|| S^{-1}  \mathbf{H}_n S ||_{\infty} = \max_{1\leqslant i\leqslant n} \left\{ \sum_{j=1}^n \left|(S^{-1} \mathbf{H}_n S)_{i,j} \right| \right\}, 
$$
where $S$ corresponds to a non-singular matrix and $(S^{-1} \mathbf{H}_n S)_{i,j}$ denotes the $i$th row and $j$th column entry of the product matrix $S^{-1} \mathbf{H}_n S$. In particular if $S$ is an invertible diagonal matrix   $S = \text{diag}(d_1,\ldots, d_k, \ldots, d_{n})$, then 
$$\begin{multlined}
	|| \mathbf{H}_n ||_S \\
	= \max \left\{ \frac{d_2}{d_1}, \frac{d_3}{d_2},    \frac{d_4 + d_1 \gamma_1}{d_3}, \ldots, \frac{d_k + d_{k-3} \gamma_{k-3}}{d_{k-1}} , \ldots, 
	\frac{d_n + d_{n-3} \gamma_{n-3}}{d_{n-1}},  \frac{ d_{n-2} \gamma_{n-2}}{d_{n}} \right\}. 
\end{multlined}$$
Setting $d_k= d^{k} (k!)^{\alpha/3}\neq 0$, for some  positive constant $d$, gives 
$$ 
	|| \mathbf{H}_n ||_S \leqslant 2^{\alpha/3}\left(d+\frac{c}{d^2}\right) n^{\alpha/3} + o(n^{\alpha/3}) \quad \text{as }\ n\to +\infty. 
$$ 
The choice of $d=\left(2c\right)^{1/3}$ gives a minimum to $\left(d+\frac{c}{d^2}\right)$, so that  
$$ 
	|| \mathbf{H}_n ||_S \leqslant  \frac{3}{4^{1/3}} \left( c \ n^{\alpha} \right)^{1/3}+ o(n^{\alpha/3}) \quad \text{as }\ n\to +\infty, 
$$ 
which implies the result. 
\end{proof}

To summarize,  Proposition \ref{prop:interlace} combined with Lemma \ref{lem: 2sym components rec rel} allow us to conclude that each $P_n^{[j]}$ has exactly $n$ real zeros $\{x_{n,k}^{[j]}\}_{k=1}^n$ with $0<x_{n,1}^{[j]}<\ldots <x_{n,n}^{[j]}$. Moreover,  between two consecutive zeros of $P_n^{[j+2]}$ there is exactly one zero of $P_n^{[j]}$ and another of $P_n^{[j+1]}$: $x_{n,k}^{[j+2]}<x_{n,k+1}^{[j]}<x_{n,k+1}^{[j+1]}<x_{n,k+1}^{[j+2]}$. Taking into consideration the asymptotic behavior for the largest zero $|x_{n,n}|$ of $P_n$ described in Theorem \ref{thm: asym zero},  we then conclude that 
$$
	x_{n,k}^{[j]} \leqslant \frac{27}{4} c \ n^{\alpha} + o(n^{\alpha}). 
$$
Therefore, for each $j\in\{0,1,2\}$, all the zeros $x_{n,k}^{(j)}$ of $P_{3n+j}$ distinct from $0$ correspond to the cubic roots of $x_{n,k}^{[j]}>0$, with $k\in\{1,\ldots, n\}$, and they all lie on the starlike set $S$ and within the disc centred at the origin with radius $b= \frac{27}{4} c \ n^{\alpha} + o(n^\alpha)$. 

Multiple orthogonal polynomials on starlike sets received attention in recent years, under different frameworks. This includes the asymptotic behavior of polynomial sequences generated by recurrence relations of the type \eqref{rec rel 3sym} when further assumptions are taken regarding specific behaviour for the $\gamma$-coefficient \cite{AKS} or for certain type of the 2-orthogonality measures in \cite{DelLop} and \cite{LG11}. Faber polynomials associated with hypocycloidal domains and stars have also been studied in \cite{HeSaff}. 

Here, we describe all the threefold $2$-\OPS s that are classical in Hahn's sense, and our study includes the  representations for the measures supported on a set containing all the zeros of the polynomial sequence. From Theorem \ref{thm:3sym 2OPS} and Proposition \ref{prop:interlace}, the support lies on  a starlike set, that, according to  Theorem \ref{thm: asym zero} is bounded if the $\gamma$-coefficients are bounded, and unbounded otherwise. 

%%%%%%%%%%%%%%

\section{Threefold symmetric Hahn-classical 2-\OPS }\label{sec:Hahn 2OPS} 

The classical orthogonal polynomial sequences of Hermite, Laguerre, Jacobi and Bessel collectively satisfy the so-called {\it Hahn property}:   
the sequence of its derivatives is again an orthogonal polynomial sequence. In the context of $2$-orthogonality, this algebraic property is portrayed as follows: 
\begin{definition} A monic 2-\OPS\ $\{ P_n\}_{n\geqslant 0}$ is "$2$-Hahn-classical"  when the sequence of its derivatives $\{ Q_n\}_{n\geqslant 0}$, with $Q_n(x) \coloneqq \frac{1}{n+1} P_{n+1}'(x)$ is also a $2$-\OPS. 
\end{definition}

The study of this type of $2$-orthogonal sequences was initiated in the works by  Douak and Maroni \cite{DM92}-\cite{DM97 II}. In those works (as well as in \cite{Mar99}) several properties of these polynomials were given, with the main pillars of the study being the structural properties, including the recurrence relations, satisfied by the polynomials. 
All in all, those studies encompassed the analysis of a nonlinear system of equations fulfilled by the recurrence coefficients. Douak and Maroni treated some special solutions to that system of equations, bringing to light several examples of these threefold symmetric "$2$-Hahn-classical" polynomials: see \cite{Douak96, DM97 I, DM97 II}. However, for those cases the support of the corresponding orthogonality measures that they found consisted of the positive real axis, which does not contain all the zeros. 
 Here, we base our analysis on the properties of the orthogonality measures and deduce the properties of the recurrence coefficients. We incorporate the works by Douak and Maroni and go beyond that by fully describing all the threefold "$2$-Hahn-classical" polynomials and bringing up explicitly the orthogonality measures along with the asymptotic behavior of the largest zero in absolute value as well as a Bochner type result for the polynomials (i.e., characterizing these polynomials via a third order differential equation). 

We start by observing that  the threefold symmetry of $\{ P_n\}_{n\geqslant 0}$ readily implies the threefold symmetry of $\{ Q_n(x) \coloneqq \frac{1}{n+1} P_{n+1}'(x)\}_{n\geqslant 0}$. This is a straightforward consequence of Definition \ref{def:3-fold}, as it suffices to take single differentiation of relation \eqref{sym}. Such property is valid for any polynomial sequence, regardless any orthogonality properties. 

Regarding threefold symmetric 2-\OPS s possessing Hahn's property, more can be said. The next result summarizes a characterization of the orthogonality measures along side with the  recurrence relations of the original sequence $\{P_n(x)\}_{n\geqslant 0}$ and the sequence of derivatives $\{Q_n(x)\coloneqq\frac{1}{n+1}P'_{n+1}(x)\}_{n\geqslant 0}$. This characterization can be found in the works \cite{DM92, DM95, Mar99}. Nonetheless, we revisit these results for a matter of completion while bringing different approaches to the original proofs, highlighting that all the structural properties for the polynomials can be derived from the corresponding 2-orthogonality measures. 

\begin{theorem}\label{thm: sym rc} Let $\{P_n\}_{n\geqslant 0}$ be a threefold symmetric 2-\OPS \ for $(u_0,u_1)$ satisfying the recurrence relation \eqref{rec rel 3sym} and let $\{Q_n(x)\coloneqq\frac{1}{n+1}P'_{n+1}(x)\}_{n\geqslant 0}$. The following statements are equivalent: 
 \renewcommand{\labelenumi}{\normalfont(\alph{enumi})}
\begin{enumerate}
\item $\{Q_n(x)\}_{n\geqslant 0}$ is a threefold symmetric 2-\OPS, satisfying the third-order recurrence relation 
\begin{equation}\label{Qn rec rel sym}
	Q_{n+1}(x) = x Q_n(x) -    \widetilde{\gamma}_{n-1} Q_{n-2}(x), 
\end{equation}
with initial conditions $Q_k(x)=x^k$ for $k\in\{0,1,2\}$.
\item The vector functional $(u_0,u_1)$ satisfies the matrix differential equation 
\begin{subequations}
 \begin{equation}\label{functional eq}
	\left( {\boldsymbol \Phi} \left[ \begin{array}{l}
	u_0\\
	u_1
	\end{array}\right]\right) '
	+ {{\boldsymbol \Psi}  \left[ \begin{array}{l}
	u_0\\
	u_1
	\end{array}\right]}
	= 
	\left[ \begin{array}{l}
	0\\
	0
	\end{array}\right],
\end{equation}
where 
\begin{equation}\label{Phi Psi}
\Phi = \left[ 
\begin{array}{c@{\qquad }c}
	 \vartheta_1   &   %%%%%%%%
   	  (1-\vartheta_1) x  \vspace{0.3cm} \\ 
   %%%%%%
    		 \dfrac{2}{\gamma_1}(1-\vartheta_2) \ x^2 &
   %%%%%%
   	   	   2\vartheta_2 -1\\ 
\end{array}
\right]
\quad \text{and}\quad 
	 {\boldsymbol \Psi} 
	= \left[ \begin{array}{cc}
	0 		& 1\\
	 \frac{2}{\gamma_1} x	& 0
	\end{array}\right]
\end{equation}
for some constants $\vartheta_1$ and $\vartheta_2$  such that
\begin{equation}\label{vartheta12}
	\vartheta_1,\vartheta_2 \neq  \frac{n-1}{n} , \quad \text{for all }\  n\geqslant 1. 
\end{equation}
\end{subequations}

\item there are coefficients $\vartheta_1,  \vartheta_2\neq  \frac{n-1}{n}, $ such that $\mathbf{U}=(u_0,u_1)$ satisfies 
 \begin{equation}
 	\label{diff eq u0}
	\Big(  \phi(x) u_0\Big)'' 
		+ \left(\frac{ 2 }{\gamma _1}(\vartheta_2 + \vartheta _1-2)x^2u_0\right)' 
	+ \frac{2}{\gamma _1} \left(\vartheta _1-2\right) x u_0
	=0  
 \end{equation}
 and 
\begin{subnumcases}{\label{u1 via u0}}
	\label{diff eq u1 via u0}
	\begin{multlined}
	\left(\vartheta _1-2\right) \left(2 \vartheta _2-1\right) u_1 \\
	=	\phi(x)u_0' 
			- \frac{2 }{\gamma _1} \left(\vartheta _1-1\right) \left(2 \vartheta _2-3\right)  x^2 u_0 , 
		\end{multlined}
			   &\text{if} $\quad \vartheta_1\neq2$,
	\\[0.3cm]
	\label{u1 2}
	x\, u_1^{\prime} = 2 u_0^{\prime} ,  
			& \text{if} $\quad \vartheta_1=2$, 
\end{subnumcases}
where 
\begin{equation}\label{phi}
\phi(x) =  \vartheta _1 \left(2 \vartheta _2-1\right)
		-\frac{2}{\gamma _1} \left(\vartheta _1-1\right) 
			\left(\vartheta _2-1\right) x^3 .
\end{equation}

\item  There exists a sequence of numbers $\{\widetilde{\gamma}_{n+1}\}_{n\geqslant 0}$ such that 
\begin{equation}\label{Pn to Qn}
	P_{n+3}(x) 	= \ Q_{n+3}(x) 
				+ \Big( (n+1) \gamma_{n+2} - (n+3)\widetilde{\gamma}_{n+1} \Big) Q_{n}(x),
\end{equation}
with initial conditions 
$
	\ P_0(x) = Q_0(x)=1  , \  
		P_1(x) = Q_1(x)  = x  \  \text{and} \  
	P_2(x) = Q_2(x) =x^2    . 
$\end{enumerate}
\end{theorem}
%%%%% 
\begin{proof}  We prove that (a) $\Rightarrow$ (d) $\Rightarrow$ (b) $\Leftrightarrow$ (c) $\Rightarrow$ (a). 

In order to see (a) implies (d), we differentiate the recurrence relation satisfied by $P_n$ to then replace the differentiated terms by its definition of $Q_n$. A substitution of the term $x Q_n(x)$ by the expression provided in \eqref{Qn rec rel sym} finally gives \eqref{Pn to Qn}. 

Now we prove that (d) implies (b). 
Any linear functional $w$ in $\mathcal{P}^\prime$ can be written as 
$$ 
	w = \sum_{k=0}^{\infty} \langle w, P_k \rangle u_k. 
$$ 
Based on this, the relation $Q_n(x)\coloneqq\frac{1}{n+1}P_{n+1}'(x)$ gives a differential relation between the dual sequence $\{v_n\}_{n\geqslant 0}$ of  $\{Q_n\}_{n\geqslant 0}$ in terms of  $\{u_n\}_{n\geqslant 0}$ (the dual sequence of  $\{P_n\}_{n\geqslant 0}$), which is 
\begin{equation}\label{vn prime}
	 v_n' = -(n+1) u_{n+1}, \qquad n\geqslant 0,  \ 
\end{equation}
whilst \eqref{Pn to Qn} leads to 
\begin{equation}\label{vn to un}
	v_n = u_n + \left( (n+1) \gamma_{n+2}-(n+3) \widetilde{\gamma}_{n+1}\right) u_{n+3} , \qquad n\geqslant 0.
\end{equation}
The choice of $n=0$ and $n=1$ in both \eqref{vn prime} and \eqref{vn to un} respectively gives
\begin{equation} \label{v0 v1 to u0 to u4}
	\left[ \begin{array}{l}
	v_0'\\
	v_1'
	\end{array}\right]
	= - 
	\left[ \begin{array}{c}
	u_1\\
	2 \, u_2
	\end{array}\right]
	\quad \text{and}\quad 
	\left[ \begin{array}{l}
	v_0\\
	v_1
	\end{array}\right]
	=  
	\left[ \begin{array}{c}
	u_0 + \left(  \gamma_{2}-3 \widetilde{\gamma}_{1}\right) u_{3}\\
	u_1 + \left( 2 \gamma_{3}-4 \widetilde{\gamma}_{2}\right) u_{4}
	\end{array}\right].
\end{equation}
As explained in \cite{Mar99} (and, alternatively, in \cite{DM92}, \cite{DM95} and \cite{Mar89}), the elements of the dual sequence $\{u_n\}_{n\geqslant0}$  of the 2-\OPS\ $\{ P_n\}_{n\geqslant 0}$  can be written as 
$$ 
\begin{array}{l}
	u_{2n} 	= E_n (x) u_0 + a_{n-1}(x) u_1 ,\\ 
	u_{2n+1}  = b_n (x) u_0 + F_{n}(x) u_1 ,
\end{array}
$$ 
where $E_n$ and ${F}_n$ are polynomials of degree $n$, while $a_n$ and ${b}_n$ are polynomials of degree less than or equal to $n$ under the assumption that $a_{-1}(x)=0$, $E_{0}=1$, $F_{0}=1$ and $b_0=0$. The recurrence relations \eqref{rec rel 3sym}  fulfilled by  $\{ P_n\}_{n\geqslant 0}$ yield (see \cite[Lemma 6.1]{Mar99})
\begin{alignat*}{3}
	&\gamma_{2n+2} F_{n+1} (x)- x F_n(x) = -a_{n-1}(x), \qquad 
	&& \gamma_{2n+3} a_{n+1} (x)- x a_n(x) = -F_{n}(x),  \\
	&\gamma_{2n+2} b_{n+1} (x)- x b_n(x) = -E_{n}(x)  \quad\text{and} \quad 
	&& \gamma_{2n+3} E_{n+2} (x)- x E_{n+1}(x) = -b_{n}(x),  
\end{alignat*}
with initial conditions $a_0 (x) =0$ and  
$E_1(x) = \frac{1}{\gamma_1} x$. In particular, we obtain: 
\begin{align*}
	b_1(x)%= -\frac{\alpha_2}{\gamma_1\gamma_2} \Big( x -\beta_0 \Big) -  \frac{1}{\gamma_2}
		= - \frac{1}{\gamma_2}, 
	\quad
	F_1(x) = \frac{1}{\gamma_2}x, \quad
	E_2(x) = \frac{1}{\gamma_1\gamma_3}  x^2 , \quad\text{and}\quad
	a_1(x) = -\frac{1}{\gamma_3}, 
\end{align*}
so that 
\begin{equation}\label{u2 u3 u4}
\begin{multlined}
	u_2 = \frac{1}{\gamma_1} x \, u_0  , \quad 
	u_3 = - \frac{1}{\gamma_2} u_0 + \frac{1}{\gamma_2}x \, u_1  
	\quad \text{and}\quad 
	u_4 = \frac{1}{\gamma_1\gamma_3}  x^2 \, u_0 -\frac{1}{\gamma_3} \, u_1. 
\end{multlined} 
\end{equation}
Consequently, the first identity in \eqref{v0 v1 to u0 to u4} reads as 
\begin{subequations}
\begin{equation}\label{DV Psi U exp}
	\left[ \begin{array}{l}
	v_0'\\
	v_1'
	\end{array}\right]
	= - 
	{\boldsymbol \Psi}
	\left[ \begin{array}{l}
	u_0\\
	u_1
	\end{array}\right]
	\quad \text{with }\quad
	 {\boldsymbol \Psi} = \left[ \begin{array}{ll}
	0 		& 1\\
	\frac{2}{{\gamma}_1} x 	&  0
	\end{array}\right]. 
\end{equation}
Using \eqref{u2 u3 u4} in the second identity in \eqref{v0 v1 to u0 to u4} leads to 
\begin{equation}\label{V Phi U}
	\left[ \begin{array}{l}
	v_0\\
	v_1
	\end{array}\right]
	={\boldsymbol \Phi} \left[ \begin{array}{l}
	u_0\\
	u_1
	\end{array}\right]
	\quad \text{with } \quad 
	 {\boldsymbol \Phi} 		
	= \left[ \begin{array}{ll}
		\phi_{0,0} &  \phi_{0,1}\\
		\phi_{1,0} &  \phi_{1,1}
	\end{array}\right],
 \end{equation}
\end{subequations}
where ({\it c.f.} \cite[Eq. (6.17)]{Mar99})
\begin{eqnarray*}
	&& \phi_{00}(x) 
			= \frac{3\widetilde{\gamma}_1}{\gamma_2} , 
	\quad  \phi_{01}(x) 
			=  \left(1 -   \frac{3\widetilde{\gamma}_1}{\gamma_2}\right) x , \\
	&& \phi_{10}(x) 
			=   2\left(\frac{1}{\gamma_1}-  \frac{2\widetilde{\gamma}_2}{\gamma_3}\right) x^2 
	\quad \text{and}\quad
	\phi_{11}(x) 
			= -1+  \frac{4\widetilde{\gamma}_2}{\gamma_3}. 
\end{eqnarray*}
We set $\vartheta_1=\frac{3\widetilde{\gamma}_1}{\gamma_2}\neq 0$ and $\vartheta_2=\frac{2\widetilde{\gamma}_2}{\gamma_3}\neq 0$, and now \eqref{functional eq}-\eqref{Phi Psi} follows after differentiating both sides of the equation in \eqref{V Phi U} and then compare with \eqref{DV Psi U exp}, which is a system of two functional equations in $(u_0,u_1)$:
\begin{equation}\label{pf functional eqs}
	\begin{cases}
		\vartheta_1 u_0^{\prime} + (1- \vartheta_1) x\, u_1^{\prime} + (2-\vartheta_1) u_1 = 0 \\
		\frac{2}{\gamma_1} (1-\vartheta_2) x^2\, u_0^{\prime} + \frac{2}{\gamma_1} (3-2\vartheta_2) x\, u_0
			- (1-2\vartheta_2)u_1^{\prime}=0. 
	\end{cases}
\end{equation}
The action of each of these equations over the monomials gives 
$$ 
	\begin{cases}
		 n \vartheta_1 (u_0)_{n-1} + \left( (n-1) - n \vartheta_1 \right) (u_1)_{n} = 0, \\
		 \dfrac{1}{\gamma_1} \left( (n-1) - n \vartheta_2 \right)  (u_0)_{n+1} 
			+ (2\vartheta_2 -1)n (u_1)_{n-1} = 0 , \qquad  n\geqslant 0,
	\end{cases}
$$ 
under the assumption that $(u_0)_{-1}=(u_1)_{-1}=0$. All the moments are well defined provided that 
$$ 
	\left( (n-1) - n \vartheta_1 \right) \left( (n-1) - n \vartheta_2 \right) \neq 0 , \qquad n\geqslant 0, 
$$ 
which corresponds to condition \eqref{vartheta12}. Under this assumption, the threefold symmetry implies that  
$$ 
	(u_0)_n = (u_1)_{n+1} =0 \quad \text{for} \quad n \not\equiv 0 \bmod 3, 
$$ 
whilst for $n\equiv 0 \bmod 3$, we have $(u_0)_{n} \ (u_1)_{n+1} \neq 0$, and these are recursively defined by 
\begin{align*}
&	\frac{2}{\gamma_1} \left( n+1 -( n+2) \vartheta_2 \right) \left( n -( n+1) \vartheta_1 \right) (u_0)_{n+3} 
	= \vartheta_1(2 \vartheta_2-1)(n+1)(n+2) (u_0)_{n}\ , \\
& 	 \left( n+1 -( n+2) \vartheta_2 \right) \left( n+3 -( n+4) \vartheta_1 \right)(u_1)_{n+4}  
		= \vartheta_1(2\vartheta_2 -1)(n+2)(n+4)   (u_1)_{n+1} \ ,
\end{align*}
where  $(u_0)_0=(u_1)_1=1$. 

Concerning the proof of  (b) $\Rightarrow $ (c), we start by noting that the system of equations  \eqref{functional eq} can be written as in \eqref{pf functional eqs} which reads as 
\begin{equation}\label{C u1 D u0}
	C \left[ \begin{array}{l}
	u_1'\\
	u_1
	\end{array}\right]
	= -D \left[ \begin{array}{l}
	u_0'\\
	u_0
	\end{array}\right],
\end{equation}
with  
$$
	C= \left[ \begin{array}{cc}
		\left(1-\vartheta _1\right) x &  2-\vartheta _1\\\
		2\vartheta _2-1 \ &  0
	\end{array}\right]
	\qquad \text{and } \qquad 
	D= \left[ \begin{array}{cc}
		\vartheta_1 &  0 \\
		\frac{2}{\gamma_1} (1-\vartheta_2)x^2 & \frac{2}{\gamma_1} (3-2\vartheta_2)x
	\end{array}\right]. 
$$
Observe that $\det C = -(2-\vartheta _1)(2\vartheta _2-1)$ is constant. The matrix $C$ is nonsingular if $\vartheta_1\neq 2$. For this case, a left multiplication by its adjoint matrix $C^{*}$ gives 
\begin{subnumcases}{\label{pf u1 u1prime}}
	\label{pf u1 prime}
	\begin{multlined}
	\left(\theta _1-2\right) \left(2 \theta _2-1\right) u_1' \\
	=  -  \frac{2}{\gamma_1}\left(\vartheta _1-2\right) \left(1-  \vartheta _2\right)x^2 u_0 ' + \frac{2}{\gamma _1} \left(\vartheta _1-2\right) \left(2 \vartheta _2-3\right) x u_0 , 
		\end{multlined}&
	   \\[0.3cm]
	\left(\theta _1-2\right) \left(2 \theta _2-1\right)  u_1 
	= \phi(x) u_0' - \frac{2}{\gamma _1} \left(\vartheta _1-1\right) \left(2 \vartheta _2-3\right) x^2 u_0 \label{pf u1} , & 
\end{subnumcases}
where $\phi(x)=\det \Phi$ which is given in \eqref{phi}. 
We differentiate \eqref{pf u1} once and compare with \eqref{pf u1 prime} to obtain \eqref{diff eq u0} and \eqref{diff eq u1 via u0}. 

When $\vartheta_1=2$, the functional equations \eqref{pf functional eqs} (or equivalently \eqref{C u1 D u0}) become 
$$ 
\begin{cases}
		2 u_0^{\prime} = x\, u_1^{\prime} , \\
		\frac{2}{\gamma_1} (1-\vartheta_2) x^2\, u_0^{\prime} + \frac{2}{\gamma_1} (3-2\vartheta_2) x\, u_0
			- (1-2\vartheta_2)u_1^{\prime}=0 ,
	\end{cases}
$$ 
and this gives \eqref{diff eq u0} and \eqref{u1 2}.

The implication (c) $\Rightarrow $ (b) can be trivially obtained by performing reciprocal operations to the ones just described. 

The proof of (b) $\Rightarrow$ (a) consists of showing that $\{ Q_n \}_{n\geqslant 0}$ is 2-orthogonal for $(v_0,v_1)$ given by \eqref{V Phi U} based on  \eqref{functional eq}-\eqref{Phi Psi} together with the 2-orthogonality of $\{ P_n \}_{n\geqslant 0}$ with respect to $(u_0,u_1)$ and on account of conditions \eqref{vartheta12}. We will omit the details which can be followed in \cite[Proposition 6.2, p.324]{Mar99}, with the obvious adaptations to the threefold symmetric case. 
\end{proof}

The vector functional $\mathbf{U}=(u_0,u_1)$ satisfies a matrix equation of Pearson type \eqref{functional eq}. This somehow mimics the properties of the  classical orthogonal polynomials. However, we refrain ourselves from calling such polynomials $2$-orthogonal classical, since other characteristic properties of the classical orthogonal polynomials can be interpreted in the context of $2$-orthogonality (or multiple orthogonality) and give rise to completely different sets of polynomials.

As a straightforward consequence of statement (c) in the latter result, we derive equations for the moments of the {\it Hahn-classical} vector linear functional $(u_0,u_1)$: 
\begin{corollary} Under the same assumptions of Theorem \ref{thm: sym rc},   the  moments of $u_0$ and $u_1$ satisfy 
\begin{equation}\label{Sym moment eq u0}
\begin{cases}
\begin{multlined}
	\frac{2}{\gamma_1}\Big[ (3n+2)\Big( (3n+1) \left(\vartheta _1-1\right) \left(\vartheta _2-1\right) 
	 -  (\vartheta_2 + \vartheta _1-2)   \Big)
	 -  \left(\vartheta _1-2\right)   \Big](u_0)_{3n+3} \\[-0.2cm]
	= (3n+2)(3n+1)\vartheta _1 \left(2 \vartheta _2-1\right) (u_0)_{3n} \ , \\
\end{multlined} \\
	(u_0)_{3n+1} = (u_0)_{3n+2} =0 
		\quad \text{and } \quad 
	(u_0)_0 =1 ,
	\end{cases}
\end{equation}
while 
\begin{equation}\label{Sym moment eq u1}
\begin{cases}
\begin{multlined}
	( u_1)_{3n+1} = \frac{2 \left(\vartheta _1-1\right) \left( (3n+2) \vartheta _2- (3n+1) \right) }
		{\gamma _1\left(\vartheta _1-2\right)  \left(2 \vartheta _2-1\right) }
		(u_0)_{3n+3}\\[-0.3cm]
		-  (3n+1)   \vartheta _1 \left(2 \vartheta _2-1\right) (u_0)_{3n}, \\
	\end{multlined}
	& \text{if} \quad \vartheta_1 \neq 2\\
	( u_1)_{3n+1} = \frac{2(3n+1)}{3n+2} (u_0)_{3n}	& \text{if} \quad \vartheta_1 = 2
	\\[0.4cm]
	(u_1)_{3n} = (u_1)_{3n+2} =0 
		\quad \text{and } \quad 
	(u_1)_1 =1 .
\end{cases}
\end{equation}
\end{corollary}
\begin{proof} 
From \eqref{diff eq u0}-\eqref{u1 via u0} it follows 
$$ 
	\langle	\Big(  \phi(x) u_0\Big)'' 
		+ \left(\frac{ 2 }{\gamma _1}(\vartheta_2 + \vartheta _1-2)x^2u_0\right)' 
	+ \frac{2}{\gamma _1} \left(\vartheta _1-2\right) x u_0 ,  x^n \rangle
	=0  , \qquad n\geqslant 0, 
$$ 
 and 
$$ \begin{cases} 
	\langle \left(\vartheta _1-2\right) \left(2 \vartheta _2-1\right) u_1 ,  x^n \rangle
	=	\langle \phi(x)u_0^{\, \prime} 
			- \frac{2 }{\gamma _1} \left(\vartheta _1-1\right) \left(2 \vartheta _2-3\right)  x^2 u_0,  x^n \rangle , 
			   &  \text{if} \  \vartheta_1\neq2,
	\\
	  x\, u_1^{\, \prime} = 2  \langle  u_0^{\, \prime} ,  x^n \rangle,  
			& \text{if} \  \vartheta_1=2, 
\end{cases}$$ 
which, on account of \eqref{properties functionals}, leads to \eqref{Sym moment eq u0}.

Similarly, by taking into account the operations defined in \eqref{pf u1 action}, we deduce that  relations \eqref{u1 via u0} imply
 \eqref{Sym moment eq u1}.
\end{proof}

The "$2$-Hahn-classical" polynomials satisfy a third order recurrence relation \eqref{rec rel 2OPS} and the  $\gamma$-recurrence coefficients have a specific rational structure. 
Here, we show that such expression for the $\gamma$-coefficients actually characterizes the  threefold symmetric "$2$-Hahn-classical" polynomials, by proving the reciprocal condition found in the work by Douak and Maroni  \cite{DM92}. For a matter of completion, we obtain the $\gamma$-coefficients directly from the functional equations \eqref{functional eq}, rather than from algebraic manipulations on the recurrence relations. 

\begin{theorem}\label{th: gammas} Let $\{ P_n \}_{n\geqslant 0}$ be a monic 2-\OPS. Then  $\{ P_n \}_{n\geqslant 0}$ satisfies  \eqref{rec rel 3sym} with 
 \begin{eqnarray}\label{sym gammas}
	\gamma_{n+2}
	= \frac{n+3}{n+1}\frac{  n(\vartheta_n -1) +1  }{ (n+4)(\vartheta_{n+1} -1) +1 }\gamma_{n+1} , 
\end{eqnarray}
where 
\begin{equation}\label{vartheta gen exp}
	\vartheta_n = \left(\frac{1-(-1)^n}{2}\right) \frac{1-\frac{n+1}{2} (1-\vartheta_1)}{1-\frac{n-1}{2} (1-\vartheta_1)}
			+\left( \frac{1+(-1)^n}{2} \right)\frac{1-\frac{n}{2} (1-\vartheta_2)}{1-(\frac{n}{2}-1) (1-\vartheta_2)} ,
			\qquad n\geqslant 1,
\end{equation}
with $\vartheta_1,\vartheta_2$ subject to \eqref{vartheta12}, 
if and only if  the sequence $\{ Q_n(x)\coloneqq \frac{1}{n+1} P_{n+1}'(x) \}_{n\geqslant 0}$ is $2$-orthogonal satisfying the recurrence relation \eqref{Qn rec rel sym} with 
\begin{equation}\label{gamma tilde}
	\widetilde{\gamma}_{n} 
	= \frac{n}{n+2} \vartheta_n \gamma_{n+1}, \quad \text{for}\quad n\geqslant 1. 
\end{equation}
\end{theorem}

\begin{proof} We start by proving that if $\{ Q_n \}_{n\geqslant 0}$ satisfies the recurrence relation  \eqref{Qn rec rel sym} with the $\widetilde{\gamma}$-coefficients given by \eqref{gamma tilde}, then the $\gamma$-coefficients in the recurrence relation of  $\{ P_n \}_{n\geqslant 0}$ are given by \eqref{sym gammas}-\eqref{vartheta gen exp}. The assumption means that $\{ Q_n \}_{n\geqslant 0}$ is 2-orthogonal and therefore $\{ P_n \}_{n\geqslant 0}$ is Hahn-classical. According to Theorem \ref{thm: sym rc}, $\{ P_n \}_{n\geqslant 0}$ is $2$-orthogonal with respect to a vector functional $(u_0,u_1)$ satisfying the differential equation \eqref{functional eq} and  $\{ Q_n \}_{n\geqslant 0}$ is $2$-orthogonal for  $\mathbf{V}=(v_0,v_1)$ given by \eqref{DV Psi U exp}-\eqref{V Phi U}. The relation $Q_n(x) = \frac{1}{n+1}P_{n+1}'(x)$ combined with the properties \eqref{properties functionals} implies 
\begin{align*}
	\langle  v_0, x^n Q_{2n} \rangle = - \frac{1}{2n+1}\langle  \left(x^n v_0\right)' ,  P_{2n+1} \rangle
\end{align*}
and 
\begin{align*}
	\langle  v_1, x^n Q_{2n+1} \rangle = - \frac{1}{2n+2} \langle  \left(x^n v_1\right)' ,  P_{2n+2} \rangle 
\end{align*}
for any $n\geqslant 0$. In the latter identities we replace $(v_0',v_1')$ and $(v_0,v_1)$ by the respective expressions given in \eqref{DV Psi U exp} and \eqref{V Phi U}, 
to obtain 
\begin{align*}
	 	&\langle  v_0, x^n Q_{2n} \rangle 
			= (2n+1)^{-1} (1-n\phi_{0,1}'(0))  \langle u_1,  x^{n}P_{2n+1} \rangle 
			,  \quad \text{for}\quad n\geqslant 0,
	\\
	 	&\langle  v_1, x^n Q_{2n+1} \rangle 
			= (2n+2)^{-1} \left(\psi'(0) - n \frac{\phi_{1,0}''(0)}{2}\right) \langle u_0,   x^{n+1} P_{2n+2} \rangle
			,  \quad \text{for}\quad n\geqslant 0, 
\end{align*}
which are the same as 
\begin{align*}
	 	&\langle  v_0, x^n Q_{2n} \rangle 
			= (2n+1)^{-1} (1-(1-\vartheta_1)n)  \langle u_1,  x^{n}P_{2n+1} \rangle 
			,  \quad \text{for}\quad n\geqslant 0,\\
	 	&\langle  v_1, x^n Q_{2n+1} \rangle 
			= (2n+2)^{-1} \frac{2}{\gamma_1} \left(1 - (1-\vartheta_2)n \right) \langle u_0,   x^{n+1} P_{2n+2} \rangle
			,  \quad \text{for}\quad n\geqslant 0. 
\end{align*}
Taking into account \eqref{gammas},  from the latter we obtain \eqref{gamma tilde} if we set $\vartheta_n$ as in  \eqref{vartheta gen exp}. Observe that 
$$ 
	\vartheta_{2n+1} =\frac{(n+1)  \vartheta_1-n}{n\vartheta_1 - (n-1)}
	\quad \text{and} \quad
	\vartheta_{2n+2} = \frac{(n+1)\vartheta_2 - n}{n \vartheta_2 - (n-1)}, \qquad n\geqslant 0, 
$$ 
and from this, we can clearly see that $\vartheta_n$ is actually a solution to the Riccati equation
\begin{equation}\label{Ricati eq}
	\vartheta_{n+3} + \frac{1}{\vartheta_{n+1} } = 2, \quad \text{for} \quad n\geqslant 0.  
\end{equation}

Conversely, 
if $\{ P_n \}_{n\geqslant 0}$  satisfies \eqref{rec rel 3sym} with the $\gamma$-coefficients given by \eqref{sym gammas} where $\vartheta_n$ is given by \eqref{vartheta gen exp}, then we prove that $\{ Q_n \}_{n\geqslant 0}$  satisfies a third order recurrence relation \eqref{Qn rec rel sym} with $\widetilde{\gamma}$-coefficients as in 
\eqref{gamma tilde}. 
We differentiate the recurrence relation \eqref{rec rel 3sym}  once and then take into account the definition of $Q_{n}(x)\coloneqq \frac{1}{n+1} P'_{n+1}(x)$ obtain the structural relation 
$$ %\label{proof Pn to Qn}
	P_n (x) = (n+1) Q_n(x) - n x Q_{n-1}(x) + (n-2) \gamma_{n-1} Q_{n-3}. 
$$ 
In \eqref{rec rel 3sym}, we replace $P_{n+1}, \, P_{n}$ and $P_{n-2}$ by the expressions provided in the latter identity to obtain  
\begin{equation}\label{proof1 Qnp1}
	\begin{multlined}
		(n+2) Q_{n+1}(x)
		= 2(n+1) x Q_n(x) - n x^2 Q_{n-1}(x) + 2(n-2) \gamma_{n-1} x Q_{n-3}(x)\\
		-(n+1)(\gamma_{n-1}+\gamma_{n}) Q_{n-2}(x) 
		- (n-4)\gamma_{n-3}\gamma_{n-1}Q_{n-5}(x), 
	\end{multlined}
\end{equation}
which is valid for any $n\geqslant 0$, under the assumption that $Q_{-n}(x)=0$. As $\{ Q_n \}_{n\geqslant 0}$ is a basis for $\mathcal{P}$, there are coefficients 
$\xi_{n+1,\nu}$ such that 
\begin{equation}\label{proof xQn}
	x Q_n(x) = \sum_{\nu=0}^{n+1} \xi_{n+1,\nu} Q_{\nu}(x), 
\end{equation}
where $\xi_{n+1,n+1+\nu}=0$ for any $\nu\geqslant 0$ and $\xi_{n+1,n+1}=1$ because  $Q_n$ is monic. Based on this, we can also write
\begin{equation}\label{x2 Qn}
	x^2 Q_{n}(x) = \sum_{\nu=1}^{n+2 } \sum_{\sigma=\nu-1}^{n+1}  \xi_{n+1,\sigma} \xi_{\sigma+1,\nu} Q_{\nu}(x) 
				+  \sum_{\sigma=0}^{n+1}  \xi_{n+1,\sigma} \xi_{\sigma+1,0} Q_{0}(x). 
\end{equation}
The threefold symmetry of $\{ P_n \}_{n\geqslant 0}$ readily implies that of $\{ Q_n \}_{n\geqslant 0}$ and this means that 
$\xi_{n,\nu}=0$ whenever $n+\nu \neq 0 \bmod 3$, so that $\xi_{n,n-1}, \, \xi_{n,n-2},  \, \xi_{n,n-4},  \, \xi_{n,n-5}, \ldots $ are all equal zero. 
We replace the terms $xQ_n, \ xQ_{n-3}$ and $x^2Q_{n-2}$ by the respective expressions given by \eqref{proof xQn} and \eqref{x2 Qn} in the relation \eqref{proof1 Qnp1} to obtain 
\begin{equation}\label{proof2 Qnp1}
	\begin{multlined}
		(n+2) Q_{n+1}(x)
		= 2(n+1) \sum_{\nu=0}^{n+1} \xi_{n+1,\nu} Q_{\nu}(x)  + 2(n-2) \gamma_{n-1} \sum_{\nu=0}^{n-2} \xi_{n-2,\nu} Q_{\nu}(x)\\
		- n \left( \sum_{\nu=1}^{n+1 } \sum_{\sigma=\nu-1}^{n}  \xi_{n,\sigma} \xi_{\sigma+1,\nu} Q_{\nu}(x) 
				+  \sum_{\sigma=0}^{n}  \xi_{n,\sigma} \xi_{\sigma+1,0} Q_{0}(x)\right)\\ 
		-(n-1)(\gamma_{n-1}+\gamma_{n}) Q_{n-2}(x) 
		- (n-4)\gamma_{n-3}\gamma_{n-1}Q_{n-5}(x). 
	\end{multlined}
\end{equation}
For $n\geqslant 2$, we  equate  the coefficients of $Q_{n-2}$, giving
\begin{equation}\label{eq xi n_2}
	0= 2(n+1) \xi_{n+1,n-2} +  2(n-2) \gamma_{n-1}    -n \xi_{n,n-3} -(n-1)(\gamma_{n-1}+\gamma_{n}) . 
\end{equation}
In particular, for $n=2$ the latter becomes 
$$ 
	0 = 4\xi_{3,0}-\gamma_{1}-\gamma_{2}  , 
$$ 
and, because of \eqref{sym gammas}, we have 
$
	\gamma_{1} = \frac{1}{3}{( 4\vartheta_{1} -3)}\gamma_{2}
$
so that 
$$ 
	\xi_{3,0} = \frac{\vartheta_1}{3} \gamma_2. 
$$ 
If we assume that for some $n\geqslant 3$, 
$$ 
	\xi_{n,n-3} = \frac{n-2}{n} \vartheta_{n-2} \gamma_{n-1}, 
$$ 
then \eqref{eq xi n_2} implies 
$$ 
	(n+2) \xi_{n+1,n-2} = \bigl((n-2) \vartheta_{n-2} -(n-3) \bigr)\gamma_{n-1}    +(n-1)\gamma_{n} , 
$$ 
which, after writing  $\gamma_{n-1}$ in terms of $\gamma_n$ via \eqref{sym gammas}, reads as 
$$ 
	 \xi_{n+1,n-2} = \frac{n-1}{n+1} \vartheta_{n-1} \gamma_{n}  , 
$$ 
and therefore we conclude that 
\begin{equation}\label{xi n+3}
	 \xi_{n+3,n} = \frac{n+1}{n+3} \vartheta_{n+1} \gamma_{n+2}  , \qquad \text{for all } \ n\geqslant 0. 
\end{equation}
Equating the coefficients of $Q_{n-5}$ in \eqref{proof2 Qnp1} leads to 
\begin{equation}\label{eq xi n+5}
\begin{multlined}
	0=  (n+2) \xi_{n+1,n-5}    + \left( 2(n-2) \gamma_{n-1}- n \xi_{n,n-3}  \right)  \xi_{n-2,n-5} \\
		- n   \xi_{n,n-6} 
		- (n-4)\gamma_{n-3}\gamma_{n-1}, \ %, \\
\end{multlined}
\end{equation}
for $n \geq 5$.
The particular choice of $n=5$ in the latter identity becomes 
$$ 
	0= 7 \xi_{6,0}    + \left(2  
		-  \vartheta_{3}  \right) \vartheta_{1} \gamma_{2}\gamma_{4}  - \gamma_{2}\gamma_{4},
$$ 
after replacing $ \xi_{5,2} $ and $ \xi_{3,0} $ by the expressions provided by \eqref{xi n+3}. Using the identity \eqref{Ricati eq} for $n=4$, we then conclude that
$\xi_{6,0}  =0$. Now, suppose that $ \xi_{n,n-6} =0 $ for some $n\geqslant 6$.  Identity \eqref{eq xi n+5} tells that 
$$ 
	 (n+2) \xi_{n+1,n-5}    = -  \left( 2(n-2) \gamma_{n-1}- n \xi_{n,n-3}  \right)  \xi_{n-2,n-5} 
		+ (n-4)\gamma_{n-3}\gamma_{n-1} ,
$$ 
which, because of \eqref{xi n+3}, becomes 
$$ 
	 (n+2) \xi_{n+1,n-5}    = - (n-4) \left( 2  -  \vartheta_{n-2}  \right) \vartheta_{n-4} \gamma_{n-1}  \gamma_{n-3}
		+ (n-4)\gamma_{n-3}\gamma_{n-1} ,
$$ 
and hence, due to \eqref{Ricati eq}, we conclude 
$$ %\label{xi n+5}
	\xi_{n+6,n} = 0 ,  \qquad \text{for all }  n\geqslant 0. 
$$ 
Now, equating the coefficients of $Q_{n+1-3j}$ for $j\geqslant 3$ in \eqref{proof2 Qnp1} gives
$$ 
\begin{multlined}
		0 
		= (n+2)  \xi_{n+1,n+1-3j}    + \left(  2(n-2) \gamma_{n-1}   - n \xi_{n,n-3}\right) \xi_{n-2,n+1-3j} \\
		- n \sum_{\sigma=n-3j}^{n-6}  \xi_{n,\sigma} \xi_{\sigma+1,n+1-3j} ,
	\end{multlined}
$$ 
which is also 
$$ 
\begin{multlined}
		0 
		= (n+2)  \xi_{n+1,n+1-3j}    + \left(  2(n-2) \gamma_{n-1}   - n \xi_{n,n-3}\right) \xi_{n-2,n+1-3j}  \\
		- n \sum_{\sigma=2}^{j}  \xi_{n,n-3\sigma} \xi_{n+1-3\sigma,n+1-3j} . 
	\end{multlined}
$$ 
Suppose that $\xi_{n,n-3k}=0$ for each $k=2,3,\ldots, j-1$ and $n\geqslant 3k$, then the latter becomes 
$$ 
	(n+2)  \xi_{n+1,n+1-3j}     
	= n \xi_{n,n-3j} , 
$$ 
which yields 
\begin{subequations}
\begin{equation}\label{pf: xi n+3j A}
	\xi_{n+1,n+1-3j}  %=\left( \prod_{\ell = 3j}^{n} \frac{\ell}{\ell+2}\right) \xi_{3j,0}
	= \frac{3j (3j+1)}{(n+2)(n+1)} \xi_{3j,0}. 
\end{equation}
In particular this implies that 
\begin{equation}\label{pf: xi n+3j B}
	\xi_{3j+1,1}   = \frac{3j }{(3j+2)} \xi_{3j,0} \quad \text{ and } \quad \xi_{3j+2,2}   = \frac{3j +1 }{(3j+3)} \xi_{3j,0} .
\end{equation}
\end{subequations}
Finally we compare the coefficients of $Q_0$ in \eqref{proof2 Qnp1} to obtain 
$$ 
	\begin{multlined}
		0 
		= (n+2)   \xi_{n+1,0}    + 2(n-2) \gamma_{n-1}   \xi_{n-2,0} 
		- n \left( \sum_{\sigma=0}^{n-1}  \xi_{n,\sigma} \xi_{\sigma+1,0} \right), 
	\end{multlined}
$$ 
for $n \geq 6$.
Recall that $ \xi_{k,0} =0$ for $k\neq 0\bmod 3$, so that the latter identity simplifies to 
\begin{equation}\label{pf xi n0 eq}
	\begin{multlined}
		0 
		= (3n+3)   \xi_{3n+3,0}    + 6n \gamma_{3n}   \xi_{3n,0} 
		- (3n+2) \left( \sum_{\sigma=1}^{n-2}  \xi_{3n+2,3\sigma+2} \xi_{3\sigma+3,0} \right) -n \xi_{3n+2,2} \xi_{3,0} , 
	\end{multlined}
\end{equation}
for $n \geq 6$.
We have already seen that $ \xi_{6,0}=0$. Proceeding by induction, we promptly observe that if $\xi_{3j,0}=0$ for $j=0,1,\ldots, n$, then 
identities \eqref{pf: xi n+3j A}--\eqref{pf: xi n+3j B} allow us to conclude from \eqref{pf xi n0 eq} that $\xi_{3(n+1),0}=0$ and this implies 
$$ 
	\xi_{n,k}=0, \qquad \text{for }\ 0\leqslant k \leqslant n-4. 
$$ 
As a result, we conclude that 
$$ 
	xQ_n = Q_{n+1} + \xi_{n+1,n-2} Q_{n-2}, \quad \text{with }\quad  \xi_{n+1,n-2} =  \frac{n-1}{n+1} \vartheta_{n-1} \gamma_{n}\neq 0, \qquad \text{for }  n\geqslant 2, 
$$ 
with $Q_j(x) = x^j$ for $j=0,1,2$. This means that $\{Q_n\}_{n\geqslant 0}$ is $2$-orthogonal and threefold symmetric. 
\end{proof}

In  \cite{DM97 II} Douak and Maroni have highlighted several properties of threefold-symmetric (therein referred to as "2-symmetric") $2$-classical polynomials including a differential equation of third order. Here we show that such differential equation actually characterizes these polynomials, bringing to the theory a Bochner-type characterization. 
 
\begin{proposition}\label{prop: 3rd ODE Pn} Let $\{P_n\}_{n\geqslant 0}$ be a  threefold symmetric 2-orthogonal polynomial sequence satisfying \eqref{rec rel 3sym}. 
The sequence $\{P_n\}_{n\geqslant 0}$ is Hahn-classical  if and only if each $P_n$ satisfies  
\begin{equation} \label{3rd order eqdiff Sym}
		(a_n x^3- b_n) P_{n}''' + c_n x^2 P_{n}'' + d_nx P_{n}' = e_n P_{n} ,
\end{equation}
where 
\begin{subequations}
\begin{eqnarray}
	a_n &=& 
		(\vartheta_n- 1) (\vartheta_{n+1} - 1) \label{an}\\ 
	b_n &=& \frac{\gamma _n \left((n-1) \vartheta _{n-1}-n+2\right) \left(n \vartheta _n-n+1\right) \left((n+1) \vartheta _{n+1}-n\right)}
				{n (n+1) }
				\label{bn}\\ 
	c_n &=& \vartheta_{n} \vartheta_{n+1}-1 -(n-3) ( \vartheta_{n} -1)( \vartheta_{n+1}-1)\label{cn}\\ 
	d_n &=& n  \vartheta_{n+1} - (n-1)  \vartheta_{n} (2  \vartheta_{n+1}-1)\label{dn}\\
	e_n &=& n  \vartheta_{n+1}, \label{en}
\end{eqnarray}
\end{subequations}
for all $n \geq 1$, 
and    $\vartheta_n$ given in \eqref{vartheta gen exp} with initial values $\vartheta_1=\frac{3(\gamma_1+\gamma_2)}{4\gamma_2}$ and $\vartheta_2 = \frac{3 \left(\gamma _1+\gamma _2\right)}{10 \gamma _3}+\frac{4}{5}$ subject to $\vartheta_{1},\vartheta_{2}\notin\{\frac{n-1}{n} :\  n\geqslant 1\}$ .
\end{proposition}
\begin{proof} 
The necessary condition was proved in \cite[Proposition 3.2]{DM97 II}, where Douak and Maroni have shown that  threefold-symmetric (therein referred to as "2-symmetric") 
$2$-orthogonal polynomials satisfying Hahn's property are solutions to \eqref{3rd order eqdiff Sym} under the definitions  \eqref{an} and \eqref{cn}-\eqref{en}  with 
$$ 
	b_n =  \frac{\gamma _{n+3} \left((n+3) \vartheta _{n}-n+2\right) \left((n+4) \vartheta _{n+1}-(n+3)\right) \left((n+5) \vartheta _{n+2}-(n+3)\right)}
				{(n+3)(n+4) } ,
$$ 
where  $\vartheta_n$ is given in \eqref{vartheta gen exp}. It turns out that $b_n$ can be written as in \eqref{bn} because, according to Theorem \ref{th: gammas}, the $\gamma$-recurrence coefficients are recursively given by \eqref{sym gammas} which implies 
$$ 		\begin{multlined}
	\gamma_{n+3}
	= \frac{(n+4)(n+3)}{n(n+1)}
	\\\times\frac{ \Big( (n+1)(\vartheta_{n+1}-1) +1  \Big) \Big( n(\vartheta_n -1) +1  \Big)\Big( (n-1)(\vartheta_{n-1}-1) +1  \Big)}
	{\Big( (n+5)(\vartheta_{n+2} -1) +1 \Big)\Big( (n+4)(\vartheta_{n+1} -1) +1 \Big)\Big( (n+3)(\vartheta_{n} -1) +1 \Big)}\ \gamma_{n}. 
	\end{multlined}
$$ 

Conversely, suppose that the $2$-orthogonal polynomial sequence $\{P_n\}_{n\geqslant 0}$ fulfils  \eqref{3rd order eqdiff Sym} under the definitions \eqref{an}-\eqref{en} with $\vartheta_n$ given in \eqref{vartheta gen exp}. Observe that $\vartheta_n$ satisfies the Riccati equation \eqref{Ricati eq} with initial conditions $\vartheta_1=\frac{3(\gamma_1+\gamma_2)}{4\gamma_2}$ and $\vartheta_2 = \frac{3 \left(\gamma _1+\gamma _2\right)}{10 \gamma _3}+\frac{4}{5}$, both assumed to be different from $\frac{k-1}{k}$ for all integers $k\geqslant 1$. 
For values of $n\geqslant 3$, we consider the expansion $\ds P_n(x) = x^n + \lambda_n x^{n-3} + \ldots$  and insert it in the differential equation \eqref{3rd order eqdiff Sym}. We equate the coefficients of $x^n$ and this gives 
$$ 
	 \lambda_{n}  =- \frac{(n-2) (n-1) n b_n}{3 \left((n-3) \vartheta _n-n+4\right) \left((n-2) \vartheta _{n+1}-n+3\right)}
	 , \qquad n\geqslant 0. 
$$ 
On the other hand, from the recurrence relation \eqref{rec rel 3sym} we deduce 
$$ 
	\gamma_n = \lambda_{n+1} - \lambda_{n+2}, \quad n\geqslant 0. 
$$ 
After combining the latter two expressions we obtain  
\begin{align*}
	\gamma_n 
	=& - \frac{(n-1) n (n+1) b_{n+1}}{3 \left((n-2) \vartheta _{n+1}-n+3\right) \left((n-1) \vartheta _{n+2}-n+2\right)}\\
	&+  \frac{n (n+1) (n+2) b_{n+2}}{3 \left((n-1) \vartheta _{n+2}-n+2\right) \left(n \vartheta _{n+3}-n+1\right)} , 
\end{align*}
where $b_n$ is given in \eqref{bn}. Based on \eqref{Ricati eq}, we consider the following substitutions in the latter expression 
$$ 
	\left((n-2) \vartheta _{n+1}-n+3\right) = \frac{\left((n-1) \vartheta _{n-1}-(n-2)\right) }{\vartheta_{n-1}} 
$$
and
$$
 \left(n \vartheta _{n+3}-(n-1)\right) = \frac{\left((n+1) \vartheta _{n+1}-n\right) }{\vartheta_{n+1}} ,
$$
to find
$$ \begin{multlined}
\gamma_n 	=n \vartheta _{n} \left((n+2) \vartheta _{n+2}-(n+1)\right)
	\Big(	- \gamma_{n+1}  \frac{(n-1)    \vartheta _{n-1}     \left((n+1) \vartheta _{n+1}-n\right) }
		{3  (n+2) \left((n-1) \vartheta _{n-1}-(n-2)\right) } 
		\\[0.3cm]
	 +  \gamma_{n+2} 	
	\frac{(n+1)  \vartheta _{n+1}
		\left((n+3) \vartheta _{n+3}-(n+2)\right)}
	{ 3(n+3) \left(n \vartheta _{n}-(n-1) \right) } \Big) ,
\end{multlined}$$ 
which is the same as 
\begin{multline}\label{pf diff eq}
 \frac{\gamma_n}{n\left((n+3) \vartheta _{n}-(n+2)\right)}	
 = 
  \gamma_{n+2} 	\frac{ (n+1) \left((n+4) \vartheta _{n+1}-(n+3)\right)}{ 3(n+3) \left(n \vartheta _{n}-(n-1) \right) } \\
 -\ \gamma_{n+1}  \frac{(n-1)     \vartheta _{n-1}     \left((n+1) \vartheta _{n+1}-n\right)  }
		{3  (n+2) \left((n-1) \vartheta _{n-1}-(n-2)\right) } 
\end{multline}
after taking the following substitutions (derived from  \eqref{Ricati eq})
$$ 
	\vartheta _{n} \left((n+2) \vartheta _{n+2}-(n+1)\right) = \left((n+3) \vartheta _{n}-(n+2)\right) ,
$$ 
and
$$ 
	\vartheta _{n+1}\left((n+3) \vartheta _{n+3}-(n+2)\right) =\left((n+4) \vartheta _{n+1}-(n+3)\right) .
$$ 
We subtract $ \frac{ \gamma_{n+1} } {  (n+2) \left((n-1) \vartheta _{n-1}-(n-2)\right) }$ from both sides of \eqref{pf diff eq} and this leads to
$$ \begin{multlined}
 \frac{-3 \gamma_n}{n\left((n+3) \vartheta _{n}-(n+2)\right)\gamma_{n+1}} 
 		\left( 1 - 	  \frac{ \gamma_{n+1} }{\gamma_n} \frac{n\left((n+3) \vartheta _{n}-(n+2)\right)}{  (n+2) \left((n-1) \vartheta _{n-1}-(n-2)\right) }   \right)
		\\[0.2cm]
\qquad= \left( 
			 1 -  \frac{\gamma_{n+2}}{\gamma_{n+1}} 	\frac{ (n+1) \left((n+4) \vartheta _{n+1}-(n+3)\right)}{ (n+3) \left(n \vartheta _{n}-(n-1) \right) }\right) ,
\end{multlined}$$ 
which implies 
$$ \begin{multlined}
	\left(\prod_{\ell=1}^n \frac{(-3) \gamma_\ell}{\ell\left((\ell +3) \vartheta _{\ell}-(\ell +2)\right)\gamma_{\ell +1}}\right) 
	\left( 1 - 	  \frac{ \gamma_{2} }{\gamma_1} \frac{\left(4 \vartheta _{1}-3\right)}{  3 }   \right)\\
	=  \left( 
			 1 -  \frac{\gamma_{n+2}}{\gamma_{n+1}} 	\frac{ (n+1) \left((n+4) \vartheta _{n+1}-(n+3)\right)}{ (n+3) \left(n \vartheta _{n}-(n-1) \right) }\right). 
\end{multlined}$$ 
The assumption on the initial value for $\vartheta_1$ readily implies the left-hand side of the latter equality to be zero and therefore we conclude 
$$ 
	0 =  \left( 
			 1 -  \frac{\gamma_{n+2}}{\gamma_{n+1}} 	\frac{ (n+1) \left((n+4) \vartheta _{n+1}-(n+3)\right)}{ (n+3) \left(n \vartheta _{n}-(n-1) \right) }\right)
	\ \text{for all }  n\geqslant 1. 	
$$ 
Now, Theorem \ref{th: gammas} ensures the $2$-orthogonality of the sequence  {$\{ Q_n(x)\coloneqq \frac{1}{n+1} P_{n+1}'(x) \}_{n\geqslant 0}$}, which means that $\{P_n\}_{n\geqslant 0}$ is Hahn-classical. 
\end{proof}

\begin{remark}
After a single differentiation of \eqref{3rd order eqdiff Sym}, we obtain a differential equation for the polynomials $Q_n$ and these satisfy: 
$$  %\label{3rd order eqdiff Sym Qn}
		(a_n x^3- b_n) Q_{n}''' + (c_n+3a_n) x^2 Q_{n}'' + (d_n+2c_n) x Q_{n}' = (e_n-d_n) Q_{n}, \qquad n\geqslant 0, 
$$ 
with $a_n,b_n,c_n,d_n$ and $e_n$ given by  \eqref{an}--\eqref{en}. 
\end{remark}

Theorem \ref{th: gammas} shows that a threefold symmetric $2$-orthogonal polynomial sequence is {\it Hahn-classical} if and only if the $\gamma$-recurrence coefficients in 
\eqref{rec rel 3sym} can be written as \eqref{sym gammas}, provided that $\vartheta_1,\vartheta_2\neq \frac{n}{n+1}$ for all positive integers $n$.

If $\vartheta_1,\vartheta_2\geqslant 1$, then $\gamma_n$ and $\widetilde{\gamma}_n$  are both positive for all integers $n\geqslant 1$. 
Furthermore, from \eqref{sym gammas}-\eqref{gamma tilde} we readily see that $\gamma_n$ and $\widetilde{\gamma}_{n} $ are two rational functions in $n$, both having the same asymptotic behavior: 
$$
	\gamma_{n} = c n^{\alpha} +{o}(n^{\alpha})\quad \text{as} \quad n\to +\infty. 
$$
Hence, as a result of Theorem \ref{thm:3sym 2OPS} together with Theorem \ref{thm: asym zero} and Theorem \ref{thm: sym rc} the two linear functionals associated with a Hahn-classical threefold symmetric $2$-\OPS\ admit the following integral representation. 
 
\begin{theorem}\label{thm:int rep classical} Let $\{P_n\}_{n\geqslant0}$ be a threefold symmetric and 2-\OPS\  with respect to the vector linear functional $(u_0,u_1)$ fulfilling \eqref{diff eq u0}-\eqref{u1 via u0}. If $\vartheta_1,\vartheta_2 \geqslant 1$, then $\{P_n\}_{n\geqslant0}$  satisfies the recurrence relation \eqref{rec rel 3sym} with $\gamma_n>0$ and 
$u_0$ and $u_1$ admit the integral representation
\begin{equation}\label{int rep for uk classical}
\begin{multlined}
\langle u_k, f \rangle = \\
\frac{1}{3}\left( \int_{0}^{b} f(x) \mathcal{U}_k(x) \, \dd x
		+ \omega^{2k-1} \int_{0}^{b\omega} \! f(x) \mathcal{U}_k(  \omega^2 x)\,  \dd x
		+  \omega^{1-2k}\int_{0}^{b\omega^2} \!\!\!\!\!\! f(x) \mathcal{U}_k(  \omega x)\,  \dd x\right)
		, 
\end{multlined}
\end{equation} 
(with $k=0,1$)  for any polynomial $f$, with \ $\omega = \e^{ 2\pi i/3} $ \  and \ $b=\lim\limits_{n\to \infty}\left( \frac{27}{4} \gamma_n\right)$,  provided that there exist two twice differentiable functions $\mathcal{U}_0$ and $\mathcal{U}_1$ mapping $[0,b)$  to $\mathbb{R} $   such that $\mathcal{U}_0$ is solution to 
\begin{align} \label{weight eq U0}
	 & \Big( \phi(x) \mathcal{U}_0 (x)\Big)'' 
		+ \left(\frac{ 2(\vartheta_2 + \vartheta _1-2) }{\gamma _1}x^2 \mathcal{U}_0(x)\right)' 
		 	+ \frac{2 \left(\vartheta _1-2\right) }{\gamma _1}x \mathcal{U}_0(x) = \lambda_0 g_0(x)\ , 
\end{align}
and  $\mathcal{U}_1$ is given by
%\begin{align}
\begin{subnumcases}{\hspace{-0.4cm}\label{weight U1}}
	\label{weight eq U1 gen}
	\begin{multlined}
	\left(\vartheta _1-2\right) \left(2 \vartheta _2-1\right) \mathcal{U}_1(x)\\
	=	
		\phi(x)\mathcal{U}_0'(x)
			- \frac{2 \left(\vartheta _1-1\right) \left(2 \vartheta _2-3\right)x^2 }{\gamma _1} \mathcal{U}_0(x) 
			+\lambda_1 g_1(x) \ ,
		\end{multlined}
			&\hspace{-0.5cm} if $\vartheta_1\neq 2$, 
			\\[0.2cm] 
			\label{weight eq U1 2}
	x \mathcal{U}_1^\prime(x) = 2 \, \mathcal{U}_0^\prime(x) +\lambda_1 g_1(x) \ ,
	& \hspace{-0.5cm}if $\vartheta_1=2$, 
\end{subnumcases}
%\end{align}
where $ \phi(x) =  \vartheta _1 \left(2 \vartheta _2-1\right)-\frac{2 \left(\vartheta _1-1\right) 
			\left(\vartheta _2-1\right) }{\gamma _1}x^3 ,$  
and satisfying
\begin{eqnarray} \label{Sym cond for weights U0}
	&&	\lim_{x\to b} f(x) \frac{\dd^l }{\dd x^l}\mathcal{U}_0(x) = 0  
		, \quad\text{for any} \  l\in\{0,1\} \ \text{ and } \  f\in\mathcal{P},\\
	&&	\int_{0}^{b} \!\! \mathcal{U}_0(x)\, \dd x=1, \notag
\end{eqnarray}
where $\lambda_k$ is a complex constant (possibly zero) and $g_k$ a function whose moments vanish identically on the support of $\mathcal{U}_k$.
\end{theorem}

\begin{proof} 
According to Theorem \ref{thm:3sym 2OPS}, the threefold symmetry of a $2$-\OPS \ ensures the existence of two orthogonality measures $\mu_0$ and $\mu_1$ (respectively defined by the vector functional $(u_0,u_1)$) supported on a starlike set $S$ on the three rays of the complex plane. On the other hand, Theorem \ref{thm: sym rc} tells that $u_0$ satisfies \eqref{diff eq u0} and $u_1$ fulfils \eqref{diff eq u1 via u0}-\eqref{u1 2}. Based on these equations, we seek an integral representation for both $u_0$ and $u_1$ via a weight function $ \mathcal{U}$ such that \eqref{int rep for uk classical} holds for any polynomial $f$. Such a representation readily ensures the threefold symmetry of $(u_0,u_1)$. Identity \eqref{diff eq u0}  is the same as  
$$ 
		\left\langle \left( \phi(x)u_0\right)'' 
		+ \left(\frac{ 2(\vartheta_2 + \vartheta _1-2) }{\gamma _1}x^2u_0\right)' 
		 	+ \frac{2 \left(\vartheta _1-2\right) }{\gamma _1}x u_0 \ ,\ f \right\rangle=0, 
			\qquad   \forall f\in\mathcal{P}, 
$$ 
which, because of \eqref{properties functionals}, reads as 
\begin{equation}\label{pf u0 eq f}
	\left\langle u_0, \phi(x) f''(x) 
			-  \frac{ 2(\vartheta_2 + \vartheta _1-2) }{\gamma _1}x^2 f'(x) 
			+ \frac{2 \left(\vartheta _1-2\right) }{\gamma _1}x f(x)
			 \right\rangle=0, \qquad  \forall f\in\mathcal{P},
\end{equation}
with 
$$\phi(x) = \vartheta _1 \left(2 \vartheta _2-1\right)-\frac{2 \left(\vartheta _1-1\right) 
			\left(\vartheta _2-1\right) }{\gamma _1}x^3. $$
We seek a function $\mathcal{W}_0$, at least twice differentiable, defined on an open set $D$ containing the piecewise differentiable curve $\mathcal{C}$ containing all the zeros of $\{P_n\}_{n\geqslant0}$ that is a subset of $S=\Gamma_0\cup \Gamma_1 \cup \Gamma_2$ (represented in Fig.\ref{fig: 3 rays})  and   
such that 
\begin{equation}\label{pf: int rep u0}
	\langle u_0,f \rangle = \int_{\mathcal{C}} f(x) \mathcal{W}_0(x)\, \dd x
\end{equation}
holds for every polynomial $f$. 
In the light of Theorem \ref{thm: asym zero},  
the support of the weight function is $\mathcal{C}=\widetilde{\Gamma}_0\cup\widetilde{\Gamma}_1\cup\widetilde{\Gamma}_2$
where $\widetilde{\Gamma}_1$ and $\widetilde{\Gamma}_2$ are the two straight lines starting at some point $b\omega$ and $b\omega^2$ (respectively) and ending at the origin, while $\widetilde{\Gamma}_0$ corresponds to the straight line on the positive real axis starting at the origin and ending at $b$. Here $b$ is a positive real number or can represent a point at infinity. 
Thus, from \eqref{pf u0 eq f}, the weight function $\mathcal{W}_0$ we seek must be such that 
$$
	\int_{\mathcal{C}} \mathcal{W}_0(x) \left(\phi(x) f''(x) 
			-  \frac{ 2(\vartheta_2 + \vartheta _1-2) }{\gamma _1}x^2 f'(x) 
			+ \frac{2 \left(\vartheta _1-2\right) }{\gamma _1}x f(x) \right)
			\, \dd x
			= 0,\ \forall f\in\mathcal{P}. 
$$
We use (complex) integration by parts to deduce 
\begin{align*}
	& \int_{\mathcal{C}} \left(  \left(  \phi(x) \mathcal{W}_0(x) \right)''  
			+\left( \frac{ 2(\vartheta_2 + \vartheta _1-2) }{\gamma _1}x^2  \mathcal{W}_0(x) \right)' 
			+ \frac{2 \left(\vartheta _1-2\right) }{\gamma _1}x  \mathcal{W}_0(x) \right) f(x)
			\, \dd x
	\\
	& + \sum_{j\in\{1,2\}} \left.  \left( f'(x) \phi(x) \mathcal{W}_0(x) +   f(x)\left(\frac{ 2(\vartheta_2 + \vartheta _1-2) }{\gamma _1}x^2  \mathcal{W}_0(x) 
			- (\phi(x) \mathcal{W}_0(x))'\right)\right)\right|_{\omega^j b}^{b}\\
	&		= 0 ,
\end{align*}
 which must hold for every polynomial $f$. This condition is fulfilled if each of the following conditions hold: 
 \begin{enumerate} 
 \item[(a)] $\mathcal{W}_0$ is a solution of the second order differential equation 
\begin{equation}\label{pf U0 ODE}
	\left(  \phi(x) \mathcal{W}_0(x) \right)''  
			+\left( \frac{ 2(\vartheta_2 + \vartheta _1-2) }{\gamma _1}x^2  \mathcal{W}_0(x) \right)'  
			+ \frac{2 \left(\vartheta _1-2\right) }{\gamma _1}x \mathcal{W}_0(x)
	= \lambda_0 g_0(x), 
\end{equation}
for some constant $\lambda_0$ (possibly zero) and a function $g_0$ representing the null linear functional on the vector space of polynomials supported on $\mathcal{C}$, that is 
$$
	\int_{\mathcal{C}} g_0(x) f(x)\, \dd x =0 , \qquad \text{for every polynomial }  \ f. 
$$

\item[(b)] The solution $\mathcal{W}_0$ of \eqref{pf U0 ODE} satisfies the boundary conditions
\begin{equation}\label{pf: bound cond U}
	 \left.  \left( f'(x) \phi(x) \mathcal{W}_0(x) +   f(x)\left(\frac{ 2(\vartheta_2 + \vartheta _1-2) }{\gamma _1}x^2  \mathcal{W}_0(x) 
			- (\phi(x) \mathcal{W}_0(x))'\right)\right)\right|_{\omega^j b}^{b}=0,  %\  (\text{with } j=1,2).
\end{equation}
with $j=1,2$,  for every polynomial $f$. 

\item[(c)] All the moments of $\mathcal{W}_0$ coincide with those of $u_0$, that is, 
$$
	 (u_0)_n 
	 = \int_{S}  \mathcal{W}_0(x) x^n \, \dd x   , \qquad n\geqslant 0.
$$
This corresponds essentially to \eqref{pf: int rep u0}, insofar as $\{x^n\}_{n\geqslant 0}$ forms a basis for $\mathcal{P}$. In particular, and in the light of Proposition \ref{prop:3fold sym}, the threefold symmetry of the linear functional $u_0$ means that $(u_0)_n = 0$ if $n\neq 0\bmod 3$.
 \end{enumerate}

Observe that if $y$ is a solution of \eqref{pf U0 ODE}, then so are the functions $\omega^j \ y(\omega^j x) $ for $j=1,2$.  Taking into consideration the threefold symmetry of $u_0$, it follows that 
\begin{equation}\label{pf: W0}
	\mathcal{W}_0(x) = \left\{ \begin{array}{ccl}
		\frac{1}{3}\mathcal{U}_0(x) &\text{if}& x\in\widetilde{\Gamma}_0=[0,b), \\[0.2cm]
		-\frac{1}{3}\omega^2\mathcal{U}_0(\omega^2 x) &\text{if}& x\in\widetilde{\Gamma}_1,  \\[0.2cm]
		-\frac{1}{3}\omega \mathcal{U}_0(\omega x) &\text{if}& x\in\widetilde{\Gamma}_2,  \\[0.2cm]
		0 &\text{if}& x\notin{\mathcal{C}}=\widetilde{\Gamma}_0\cup\widetilde{\Gamma}_1\cup\widetilde{\Gamma}_2. 
	\end{array}\right.
\end{equation}
where $\mathcal{U}_0: [0,b) \to \mathbb{R}$ is an at least twice differentiable function that is a  solution of  the differential equation \eqref{pf U0 ODE} satisfying the conditions \eqref{Sym cond for weights U0}, so that \eqref{pf: bound cond U} holds. Here,  $\widetilde{\Gamma}_1$ and $\widetilde{\Gamma}_2$ are the two straight lines starting at some point $b\omega$ and $b\omega^2$ (respectively), with $\omega = \e^{2\pi i /3}$,  and ending at the origin, while $\widetilde{\Gamma}_0$ corresponds to the straight line on the positive real axis starting at the origin and ending at $b$. Hence  \eqref{int rep for uk classical} is proved for $k=0$. 

Similarly, an integral representation for the linear functional $u_1$ can be obtained from \eqref{u1 via u0} by seeking a differentiable function  $\mathcal{W}_1$ defined on $\mathcal{C}$ such that 
\begin{equation}\label{pf: int u1}
	\langle u_1, f(x) \rangle = \int_{\mathcal{C}} f(x) \mathcal{W}_1(x)\, \dd x, \ \text{for every polynomial } f.  
\end{equation}
If $\vartheta_1\neq 2$, then $u_1$ satisfies \eqref{diff eq u1 via u0} and this gives
$$
	\langle u_1, f \rangle 
	=  \left(\vartheta _1-2\right)^{-1} \left(2 \vartheta _2-1\right)^{-1}
	\langle \phi(x) u_0' - \frac{2 \left(\vartheta _1-1\right) \left(2 \vartheta _2-3\right) }{\gamma _1}x^2 u_0, f \rangle. 
$$
Based on the properties \eqref{properties functionals}, the latter becomes
\begin{equation}\label{pf u1 action}
	\langle u_1, f \rangle 
	=  \left(\vartheta _1-2\right)^{-1} \left(2 \vartheta _2-1\right)^{-1}
	\langle u_0 , - ( \phi f)' - \frac{2 \left(\vartheta _1-1\right) \left(2 \vartheta _2-3\right) }{\gamma _1}x^2 f \rangle. 
\end{equation}
Taking into consideration \eqref{pf: int rep u0}, we then have 
$$\begin{multlined}
	\langle u_1, f \rangle 
	=
	  - \left(\vartheta _1-2\right)^{-1} \left(2 \vartheta _2-1\right)^{-1} \\
	\times\int_{\mathcal{C}}  \mathcal{W}_0(x)  
	\left(\left( \phi(x) f(x)\right)' - \frac{2 \left(\vartheta _1-1\right) \left(2 \vartheta _2-3 \right) }{\gamma _1}x^2 f(x) \right)\, \dd x. 
\end{multlined}
$$
We perform integration by parts on the first term of the integral, to obtain 
$$\begin{multlined}
 	 \left(\vartheta _1-2\right) \left(2 \vartheta _2-1\right)  \langle  u_1, f \rangle \\ 
  \begin{array}{l}
 	=\displaystyle\sum_{j\in{1,2}}\left( \Big.  \phi(x) \mathcal{W}_0(x)f(x) \Big|_{b\omega^j}^b\right)\\[0.4cm]
	\displaystyle + \int_{\mathcal{C}}   f(x) 
	\left(  \phi(x)   \mathcal{W}_0'(x) - \frac{2 \left(\vartheta _1-1\right) \left(2 \vartheta _2-3 \right) }{\gamma _1}x^2\mathcal{W}_0(x) \right)\, \dd x. 
	\end{array}
\end{multlined}
$$
The vanishing conditions \eqref{Sym cond for weights U0} at the end points of the contour readily imply that the integrated term vanishes identically for every polynomial $f$ considered. As a consequence, a weight function associated with $u_1$ corresponds to 
$$\begin{multlined}
	\mathcal{W}_1 (x ) 
	= 
	\left(\vartheta _1-2\right)^{-1} \left(2 \vartheta _2-1\right)^{-1}
	\left(  \phi(x)   \mathcal{U}'(x) - \frac{2 \left(\vartheta _1-1\right) \left(2 \vartheta _2-3 \right) }{\gamma _1}x^2\mathcal{W}_0(x) \right) 
	\\
	+ \lambda_1 g_1(x),  
\end{multlined}$$
where $\lambda_1$ represents a complex constant (possibly zero) and $g_1$ a function representing the null linear functional. 
By construction, this choice for the weight function guarantees that \eqref{pf: int u1} is well defined and, in particular, all the moments satisfy the threefold symmetry property described in Proposition \ref{prop:3fold sym}.
So, we conclude that \eqref{int rep for uk classical} also holds for $k=1$ with $	\mathcal{U}_1$ given by \eqref{weight eq U1 gen} when  $\vartheta_1\neq2$. 

If $\vartheta_1=2$, then using an entirely analogous approach we deduce \eqref{weight eq U1 2}  from  \eqref{u1 2}. 
\end{proof}

\begin{remark}\label{rem: weightsQn} An alternative to the representation \eqref{int rep for uk classical} stated in Theorem \ref{thm:int rep classical} is 
$$
	\langle u_k,f \rangle = \int_{\mathcal{C}} f(x) \mathcal{W}_k(x)\, \dd x, \qquad (k=0,1),
$$
where $\mathcal{W}_k(x)$ is given in \eqref{pf: W0}. 

The weight functions $\mathcal{W}_0$ and $\mathcal{W}_1$ are {\it threefold symmetric} in the sense that they satisfy the following rotational invariant property: 
$$ 
	\omega^j \mathcal{W}_{k}(\omega^j x) = \mathcal{W}_{k}(x), \qquad j=0,\pm 1,\pm2, \ldots , \ \text{for each }\ k=0,1. 
$$ 
\end{remark}

We consider the following {\it equivalence relation} between two polynomial sequences $\{P_n\}_{n\geqslant 0}$ and $\{B_n\}_{n\geqslant 0}$: 
\begin{equation}\label{equiv rel}
	\{P_n\}_{n\geqslant 0} \ \sim \ \{B_n\}_{n\geqslant 0}
	\ \text{iff} \
		\exists \: a\in\mathbb{C}\backslash\{0\} ,\, b\in\mathbb{C} \ \text{such that } \  
	B_n(x) = a^{-n} P_n(ax+b), 
\end{equation}
for all  $n\geqslant 0$.
In this case, if  $\{P_n\}_{n\geqslant 0}$ satisfies \eqref{rec rel 2OPS} then $\{B_n\}_{n\geqslant 0}$ satisfies 
$$
		B_{n+1}(x) = \left(x-\frac{\beta_n-b}{a}\right) B_n(x) - \frac{\alpha_n}{a^2}B_{n-1}(x) - \frac{\gamma_{n-1}}{a^3} B_{n-2}(x) , \qquad n\geqslant 1,  
$$
with initial conditions $B_{-2}(x) = B_{-1}(x) =0 $ and $B_0(x)=1$.

Observe that $\vartheta_n$ is a solution of a Riccati equation \eqref{Ricati eq} for which $\vartheta_n=1$ is a trivial solution. Depending on the initial conditions 
$\vartheta_1$ and $\vartheta_2$, there are four sets of independent solutions which in fact give rise to four equivalence classes of the threefold symmetric 2-Hahn-classical polynomials. In other words, up to a linear transformation of the variable, there are at most four distinct families of threefold symmetric 2-orthogonal Hahn-classical polynomials, which we single out: 
\begin{itemize}
\item[] {\bf Case A}: $\vartheta_1 =1= \vartheta_2$. This implies that $\vartheta_n =1 $ for all $n\geqslant 0$.

\item[]  {\bf Case $\textbf{B}_1$}: $\vartheta_1 \neq1$ but $ \vartheta_2=1 $ so that by setting $\vartheta_1 = \frac{\mu+2}{\mu+1}$ it follows that 
$$
	\vartheta_{2n-1} = \frac{n+\mu+1}{n+\mu} \quad \text{and}\quad \vartheta_{2n} =  1\ , \qquad n\geqslant 1. 
$$

\item[] {\bf Case $\textbf{B}_2$}: $\vartheta_1 =1$ but $ \vartheta_2\neq1 $ so that by setting $\vartheta_2 = \frac{\rho +2}{\rho+1}$ it follows that
$$
	\vartheta_{2n-1} = 1 \quad \text{and}\quad \vartheta_{2n} =  \frac{n+\rho+1}{n+\rho} \ , \qquad n\geqslant 1. 
$$

\item[] {\bf Case C}: $\vartheta_1 \neq1$ and $ \vartheta_2\neq 1 $ and hence by setting $\vartheta_1 = \frac{\mu+2}{\mu+1}$ and $\vartheta_2 = \frac{\rho +2}{\rho+1}$ it follows that 
$$
	\vartheta_{2n-1} = \frac{n+\mu+1}{n+\mu} \quad \text{and}\quad \vartheta_{2n} =  \frac{n+\rho+1}{n+\rho} \ , \qquad n\geqslant 1. 
$$
\end{itemize}

All these cases were highlighted in \cite{DM92}, where expressions for the recurrence coefficients for the three components of the cubic decomposition of each of these threefold symmetric Hahn-classical polynomials were deduced. Only some of these polynomials were studied in detail: those in case A in \cite{Douak96} and a few subcases of Case C in \cite{DM97 I,DM97 II}. The representations for the orthogonality measures that we describe in the next sections are new. The support of these orthogonality measures contain all the zeros of the polynomial sequences. Even for particular cases that already appeared in the literature, the integral representations for the orthogonality measures were either not given or given on the positive real line (with oscillating terms), which is only part of the starlike set $S$. 
In the next subsections we fully describe all these cases in detail.  

As observed in \cite{DM92}, the following limiting relations take place 
$$
	\text{case C } \xrightarrow[\ \rho\to\infty\ ]{ } \text{B}_1 \xrightarrow[\mu\to\infty]{} \text{case A },  
$$
and also 
$$
	\text{case C } \xrightarrow[\ \mu\to\infty\ ]{ } \text{B}_2 \xrightarrow[\ \rho\to\infty\ ]{ }  \text{case A }. 
$$
It turns out that cases $\text{B}_1$ and $\text{B}_2$ are related to each other by differentiation, as explained in Section \ref{subs: caseB1} and Section \ref{subs: caseB2}.

\subsection{Case A}\label{subs: caseA} ~ 

In the light of Theorem \ref{th: gammas}, for this choice of initial conditions one has $\vartheta_n=1$ for all integers $n\geqslant 1$, so that the $\gamma$-recurrence coefficients are  given by 
$
	 \gamma_{n+1}
	%= \frac{n+3}{n+1}\  \gamma_{n+1}
	=(n+1)(n+2)\frac{\gamma_1}{2}  
$
 and  
$\widetilde{\gamma}_{n+1} = \gamma_{n+1}$ for $ n\geqslant 0$. As a consequence, $Q_n(x) = P_n(x)$ for all $n\geqslant0$ and, for this reason, the $2$-\OPS \ $\{P_n\}_{n\geqslant 0}$ is an {\it Appell sequence}.

Since the 2-orthogonality property is invariant under any linear transformation, we can set $\gamma_1=2$, and, with this choice, recalling \eqref{sym gammas}-\eqref{gamma tilde} it follows  that
%\begin{equation}\label{A: gamma}
$$ 
	\gamma_{n+1} = \widetilde{\gamma}_{n+1}=  (n+2)(n+1), \qquad n\geqslant0, 
$$ 
%\end{equation}
so that 
\begin{equation}\label{A: rec rel}
\begin{array}{l}
	P_{n+1}(x) = xP_n(x) -  n(n-1)P_{n-2}(x),  \qquad  n\geqslant 2,\\
	P_0(x)=1, \ P_1(x) =x \ \text{ and } \ P_2(x) = x^2. 
\end{array}
\end{equation}

\begin{figure}[th] 
\centerline{\psfig{file=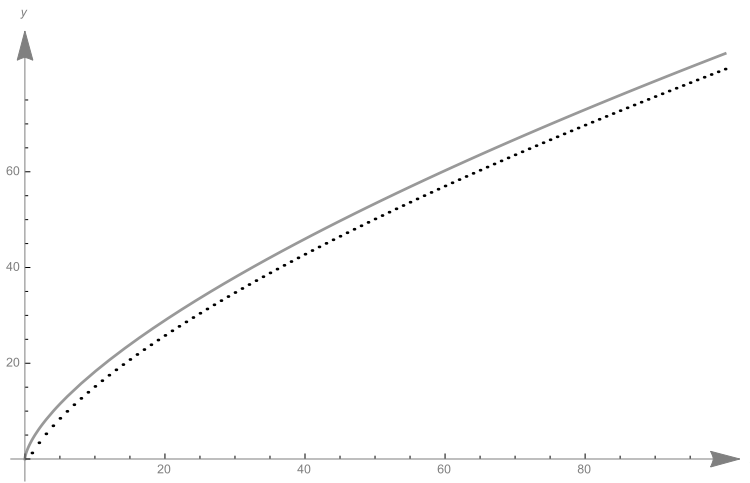,width=10cm}} 
\caption{Case A: Plot of the largest zero in absolute value of $P_{3n}$  against the curve $y=\frac{3^{5/3}}{2^{2/3}}x^{2/3}$.}
\label{fig: largest zero Case A}
\end{figure}

Figure \ref{fig: largest zero Case A} illustrates the behavior of the largest zero in absolute value (plot generated in {\it Mathematica}): 
bounded from above by the curve $y=\frac{3}{2^{2/3}} n^{2/3}$ suggested by 
Theorem \ref{thm: asym zero}.

Regarding the integral representation of the corresponding orthogonality measures, we have: 

\begin{proposition} The threefold symmetric polynomial sequence $\{P_n\}_{n\geqslant 0}$ defined by the recurrence relation \eqref{A: rec rel} is 2-orthogonal with respect to $(u_0,u_1)$ admitting the integral representation  \eqref{int rep for uk classical}, where 
$$
	\mathcal{U}_0(x) = \mathrm{Ai}(  x)
	\ \text{ and } \ 
	\mathcal{U}_1(x) = - \mathrm{Ai}'(  x), 
$$
and $b=+\infty$. Moreover, the sequence $\{Q_n(x)\coloneqq\frac{1}{n+1}P_{n+1}'(x)\}_{n\geqslant 0}$ coincides with $\{P_n\}_{n\geqslant 0}$ ({\it i.e.}, $\{P_n\}_{n\geqslant 0}$ is an Appell sequence). 
\end{proposition}

Here and in what follows $\text{Ai}$ and $\text{Bi}$ are the \textit{Airy functions} of the first and second kind (see \cite[\S 9]{DLMF}), respectively. 

\begin{proof} 
Under the assumptions, by virtue of Theorem \ref{th: gammas} and Theorem \ref{thm: sym rc}, this polynomial sequence is $2$-orthogonal with respect to $(u_0,u_1)$  satisfying  \eqref{diff eq u0}-\eqref{u1 via u0} which reads as 
$$ %\label{A: functional diff eq}
	\left\{\begin{array}{l}
		   u_0'' - x\ u_0=0 ,
		\vspace{0.2cm}
		\\		
	u_1  =	- u_0',  
	\end{array}\right.
$$ 
from which we conclude that the corresponding sequence of the moments $\{(u_0)_n\}_{n\geqslant0}$ and  $\{(u_1)_n\}_{n\geqslant0}$ satisfy
$$
	 (u_0)_{n+3} = 	  (n+1)(n+2)  (u_0)_{n}
	\quad  \text{ and }\quad 
	(u_1)_n = n (u_0)_{n-1} ,
$$
with initial conditions $(u_0)_0=1$ and $(u_0)_1=(u_0)_2=(u_1)_0=0$. This implies 
\begin{equation}
\label{A: moments}
\begin{array}{lcll}
	(u_0)_{3n} &=& \frac{(3n)!}{3^{n} (n!)} & \text{and }\quad 
	(u_0)_{3n+1} =(u_0)_{3n+2} =0 ,  \vspace{0.2cm}\\ 
	(u_1)_{3n+1} &=&  \frac{(3n+1)!}{3^{n} (n!)} & \text{and }\quad 
	(u_1)_{3n} =(u_1)_{3n+2} =0 , \qquad n\geqslant 0.
\end{array}
\end{equation}
According to Theorem \ref{thm:int rep classical},  $u_0$ and $u_1$ admit the representation  \eqref{int rep for uk classical}  provided that there exist two twice differentiable functions $\mathcal{U}_0$ and $\mathcal{U}_1$ from $\mathbb{R}$  to $\mathbb{R} $   that are solutions to 
\begin{equation} \label{A1 eq W}
	\left\{\begin{array}{l}
		\mathcal{U}_0'' (x) 
		 	 - x \,\mathcal{U}_0(x) = \lambda_0 g_0(x),
		\vspace{0.2cm}
		\\		
	-\mathcal{U}_1(x)
	=	
		 \mathcal{U}_0'(x) + \lambda_1 g_1(x) \ ,
	\end{array}\right. 
\end{equation}
and satisfying
\begin{eqnarray} \label{A: cond for weights U0}
	&&	\lim_{x\to {\omega^j}\infty} f(x) \frac{\dd^l }{\dd x^l}\mathcal{U}_0(x) = 0  
		, \quad\text{for }   j,l\in\{0,1,2\} \ \text{ and } \  f\in\mathcal{P},\\
	&&	\int_{0}^{\infty} \!\! \mathcal{U}_0(x)\, \dd x=1 , \notag
\end{eqnarray}
where $\lambda_k \in \mathbb{C}$ (possibly zero) and $g_k(x)\neq0$ are rapidly decreasing functions, locally integrable, representing the null functional ($k=0,1$). For $\lambda_0=0$, the general solution of the first equation in \eqref{A1 eq W} can be written as
$$
	y(x) = c_1 \text{Ai}(  x) + c_2 \text{Bi}(  x) , 
$$
for arbitrary constants $c_1,c_2$. Observe that \eqref{A: cond for weights U0} is realized if we take $c_2=0$ (see \cite[(9.7.5)-(9.7.8)]{DLMF}). Since 
 \cite[(9.10.17)]{DLMF}
$$
	 \int_{ 0}^{+\infty} x^n \text{Ai}(x) \, \dd x = \frac{\Gamma(n+1)}{3^{\frac{n}{3}+1} \Gamma(\frac{n}{3}+1)} < +\infty, \qquad \text{for all } n\geqslant 0, 
$$
the result now follows if we take $c_1=1$ and $\lambda_1=0$. 
\end{proof}

\medskip

\subsubsection{The differential equation} 

Under these assumptions \eqref{3rd order eqdiff Sym} becomes 
$$  
	-  P_{n}''' (x) + x P_{n}'(x) =n P_{n}(x), \qquad n\geqslant 0, 
$$ 
whose general solution is given by 
\begin{eqnarray*}
	P_n(x) 
	&=& c_0 \; \pFq{1}{2}{-\frac{n}{3}}{\frac{1}{3},\frac{2}{3}}{\frac{x^3}{9}}
	+ c_1\; x  \; \pFq{1}{2}{-\frac{n-1}{3}}{\frac{2}{3},\frac{4}{3}}{\frac{x^3}{9}}
	+ c_2 \; x^2 \; \pFq{1}{2}{-\frac{n-2}{3}}{\frac{4}{3},\frac{5}{3}}{\frac{x^3}{9}}, 
\end{eqnarray*}
with integration constants $c_0,\ c_1$ and $c_2$. Observe that for each $n$ there is one polynomial solution and we have 
\begin{eqnarray*}
	& P_{3n}(x) = P_n^{[0]} (x^3) & \text{with} \quad  P_n^{[0]} (x) = (-9)^n (1/3)_n (2/3)_n\;  \pFq{1}{2}{-n}{\frac{1}{3},\frac{2}{3}}{\frac{x}{9}} , \\
	& P_{3n+1}(x)= x \ P_n^{[1]} (x^3) & \text{with} \quad  P_n^{[1]} (x)  = (-9)^n (2/3)_n (4/3)_n\;  \pFq{1}{2}{-n}{\frac{2}{3},\frac{4}{3}}{\frac{x}{9}} , \\
	& P_{3n+2}(x)= x^2 \ P_n^{[2]} (x^3) & \text{with} \quad  P_n^{[2]} (x)  = (-9)^n (4/3)_n (5/3)_n\;  \pFq{1}{2}{-n}{\frac{4}{3},\frac{5}{3}}{\frac{x}{9}}. 
\end{eqnarray*}

\subsubsection{The cubic decomposition}

In \cite[\S 5]{BCD00} the cubic decomposition of an Appell 2-orthogonal sequence has been highlighted. This also happens to be threefold symmetric, and the corresponding polynomials  were called Hermite-type 2-orthogonal polynomials. The three $2$-\OPS s  $\{P_n^{[j]}\}_{n\geqslant 0}$ ($j=0,1,2$) in the cubic decomposition  of $\{P_n\}_{n\geqslant 0}$ are $2$-orthogonal with respect to weights involving modified Bessel functions of the second kind, studied in \cite{BCD00, VAY}. 

Following Lemma \ref{lem: 2sym components rec rel}, each of the three polynomial sequences $\{P_n^{[j]} \}_{n\geqslant 0}$, with  $j\in\{0,1,2\}$, is $2$-orthogonal and satisfies the third order recurrence relation \eqref{CD rec rel}, where  
%$$\gamma_{n+1} =  (n+2)(n+1),$$
$$\begin{array}{l}
	\beta_n^{[j]} =  3 \left(j^2+6 j n+j+9 n^2\right)+9 n+2, \qquad n\geqslant 0,\\
	\alpha_n^{[j]} = 3 (j+3 n-2) (j+3 n-1)^2 (j+3 n)
			, \qquad n\geqslant 1, \\
	\gamma_n^{[j]} = (j+3 n-2) (j+3 n-1) (j+3 n) (j+3 n+1) (j+3 n+2) (j+3 n+3) ,  \qquad n\geqslant 2.  \\
	\end{array}
$$
It should be noted that the Appell polynomials $\{P_n\}_{n\geqslant 0}$ and the components in the cubic decomposition were treated in \cite{BCD00} under a different normalisation, namely by considering $2/\gamma_1 =9$.  However the integral representation provided in  \cite{BCD00} is supported on the positive real axis, and therefore different from the one given here. In \cite[Cor. 5,7]{LOHerm} the authors have also studied these polynomials, where the focus was put on algebraic properties, including generating functions. 

\begin{remark} This Appell  sequence $\{P_n\}_{n\geqslant 0}$  has already appeared in the literature in other contexts, often not recognized as $2$-\OPS s. For instance, it is linked to the Vorob'ev-Yablonski polynomials associated with rational solutions of the second Painlev\'e equations \cite{CM03}. 
Furthermore, in \cite{Widder} Widder studied the so-called Airy transform defined as follows
$$
	u(x,t)= \int_{-\infty}^{+\infty} f(y) \frac{1}{(3t)^{1/3}} \text{Ai}\left(\frac{y-x}{3t^{1/3}}\right)\, \dd y.
$$
Up to a scaling, it maps the sequence of monomials to the Appell $2$-\OPS \ $\{P_n\}_{n\geqslant 0}$ \cite[\S 8]{Widder}. For this reason, in \cite[\S 4.2.3]{SoaVal} (and also \cite{BDS}) they have been referred to as "Airy polynomials". 
\end{remark}

Figure \ref{fig: zeros of Case A} corroborates the statement in Proposition \ref{prop:interlace}: the zeros on the positive real axis of three consecutive polynomials interlace and all the other non-zero zeros are rotations of $2\pi/3$ of them.

\begin{figure}[th] 
\centerline{\psfig{file=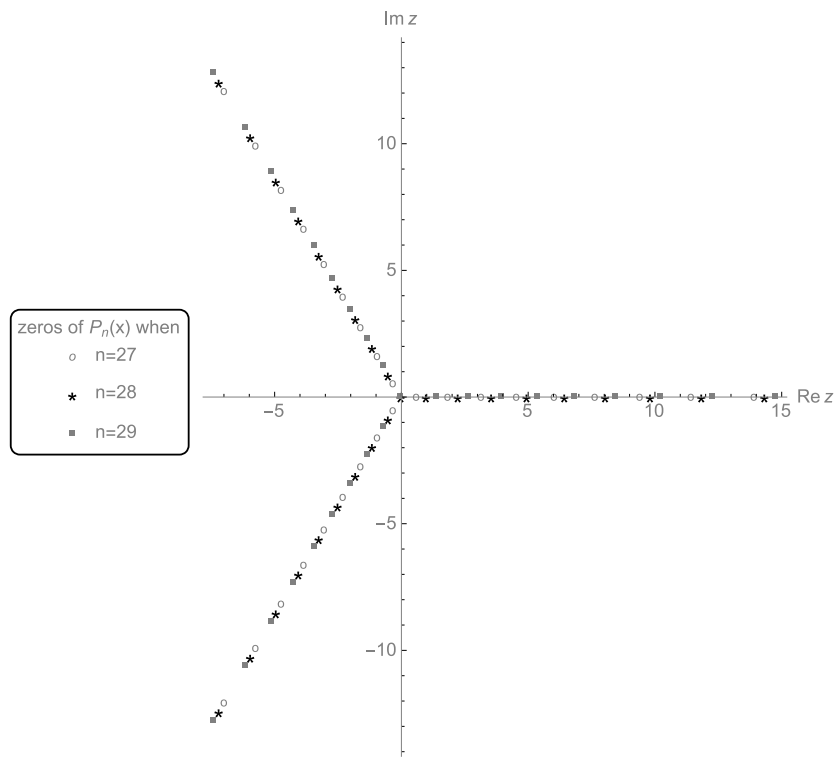,width=11cm}} 
\caption{Case A: Zeros of $P_{27}$ (empty circles), $P_{28}$ (stars) and $P_{29}$ (solid squares). }
\label{fig: zeros of Case A}
\end{figure}

\pagebreak

%%%%%%%%%%%%%%%
%\texorpdfstring{B\textsubscript{1}}{ }
\subsection{Case $B_1$}\label{subs: caseB1} ~ 

In this case, $\vartheta_{2n}=1$ and $\vartheta_{2n-1} = \frac{n+\mu+1}{n+\mu}$, under the constraint that $\mu\neq -n$, for all $n\geqslant1$. Therefore, the threefold symmetric sequence $\{P_n(\cdot;\mu;\gamma_1)\}_{n\geqslant 0}$ satisfies the recurrence relation \eqref{rec rel 3sym}, where 
\begin{align*}
& 	\gamma_{2n +1} = \frac{(n+1)(2n+1) (\mu+2)}{(3n+\mu+2)} \gamma_1, \\
&	\gamma_{2n +2 } =  \frac{(n+1)(2n+3)(n+\mu+1) (\mu+2)}{(3n+\mu+2)(3n+\mu+5)} \gamma_1, \quad n\geqslant 0.   
\end{align*}
Observe that $\gamma_1$ is a mere scaling factor and, from this point forth, we set $\gamma_1 = \frac{2}{3(\mu+2)}$. The discussion is therefore about on 
$\{P_n(\cdot;\mu)\coloneqq P_n(\cdot;\mu,\frac{2}{3(\mu+2)})\}_{n\geqslant 0}$. The general case  $\{P_n(\cdot;\mu;\gamma_1)\}_{n\geqslant 0}$ can be deduced after a linear transformation of the variable: 
$$
	P_n(x;\mu;\gamma_1) =a^{-n} P_n\left(ax;\mu\right), \quad \text{with } \ a= \left(\tfrac{2}{3(\mu+2)\gamma_1}\right)^{1/3},  \qquad n\geqslant 0.
$$
The two sequences are equivalent, under \eqref{equiv rel}.

Hence, $\{P_n(\cdot;\mu)\}_{n\geqslant 0}$ satisfies  \eqref{rec rel 3sym} with $\gamma_{n+1}\coloneqq\gamma_{n+1}(\mu)$ given by 
\begin{equation}\label{B1: gamma}
	  \gamma_{2n +1} =\frac{2}{3 } \frac{(n+1)(2n+1)  }{(3n+\mu+2)}, \quad
	  \gamma_{2n +2} =  \frac{2}{3 }\frac{(n+1)(2n+3)(n+\mu+1) }{(3n+\mu+2)(3n+\mu+5)}  , \qquad n\geqslant 0,
\end{equation}
and it is $2$-orthogonal for $(u_0,u_1)$  for which 
\begin{equation}\label{Sym eq for u0 and u1 B1}
	  \frac{1}{3}  u_0'' + x^2 u_0'- (\mu -2) x u_0  = 0
	  \quad \text{and} \quad 
	\begin{cases}
		u_1
			=	-\frac{ (\mu +2)}{  \mu }  \Big( u_0' + 3x^2 u_0\Big)
			& \text{if}\quad \mu\neq0 , \\
		x u_1' = 2 u_0'& \text{if}\quad \mu=0. 
	\end{cases}
\end{equation}
Consequently, 
\begin{equation}\label{B mom u0}
	(u_0)_{3n} 
		= \frac{(\frac{1}{3})_n (\frac{2}{3})_n  }{    \left(\frac{\mu +2}{3}\right)_n} 
	\quad \text{and} \quad 
	\ (u_0)_{3n+1} =(u_0)_{3n+2} =0 , 
\end{equation}
and 
\begin{equation}\label{B mom u1}
	\left(u_1\right)_{3n+1} 
		= \frac{\left(\frac{2}{3}\right)_n \left(\frac{4}{3}\right)_n}{\left(\frac{\mu+5}{3}\right)_n}
	\quad \text{and} \quad 
	\ (u_1)_{3n} =(u_1)_{3n+2} =0. 
\end{equation}
Observe that in order to have $\gamma_{n+1}>0$ for all $n\geqslant 0$, the constraint $\mu>-1$ needs to be imposed. 

Particular choices of the parameter $\mu$ make the $\gamma$-coefficients linear functions in $n$, namely: 
$$
	\text{for } \ \mu=-1/2: \quad \gamma_{2n }(-1/2) = \frac{4n}{27}, \quad \gamma_{2n +1}(-1/2)= \frac{4(n+1)}{9},
$$
whilst 
$$
	\text{for } \ \mu=1: \quad \gamma_{2n } (1)= \frac{2(2n+1)}{27}, \quad \gamma_{2n +1}(1)= \frac{2(2n+1)}{9}.
$$

For each $\mu>-1$, we have 
$$ 
	\gamma_{2n+1}(\mu) = \frac{4n}{9} +  \frac{2(5-2\mu)}{27}\mathcal{O}(1)
		\quad \text{whilst} \quad 
	\gamma_{2n+2}(\mu) = \frac{4n}{27}  +  \frac{2(7+2\mu)}{81}\mathcal{O}(1) \qquad  n\to+\infty.
$$ 
In the light of Theorem \ref{thm: asym zero},   an upper bound for the  absolute value of the largest zero $|x_{n,n}|$ of $P_n$ is given by \eqref{larger zero} with $c=\frac{4}{9}$ and $\alpha=1$ so that we obtain 
$$ 
		|x_{n,n}| \leqslant  3^{1/3} n^{1/3} + {o}(n^{1/3}), \qquad n\geqslant 1. 
$$ 
The accuracy of the result is illustrated in Figure \ref{fig: largest zeroB1}, where we have only plotted the positive zeros of $P_{3n}(x)$, but similar results hold for the zeros of $P_{3n+1}(x)$ and $P_{3n+2}(x)$, which are just rotations of the positive zeros.

\begin{figure}[th] 
\centerline{\psfig{file=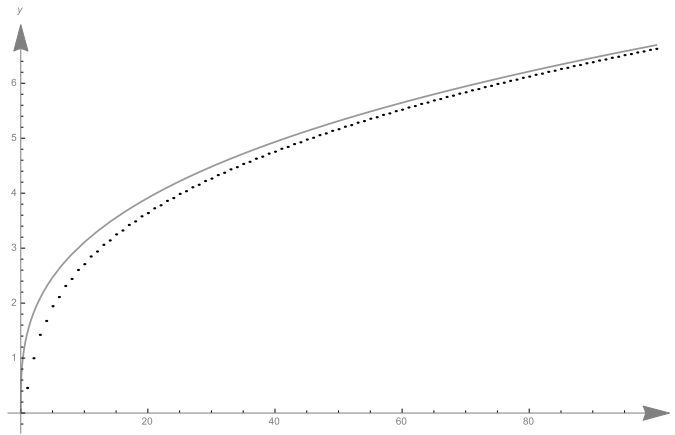,width=9cm}} 
\caption{Plot of of the curve $y=  3^{1/3}  x^{1/3}$ (grey solid line) and the largest zeros in absolute value (black dots) of $P_{3n}(x;3)$ with $n$ ranging from $0$ to $100$.  } 
\label{fig: largest zeroB1}
\end{figure}

The zeros of three consecutive polynomials interlace and lie on the three-starlike set $S$ : Figure \ref{zeros PB1} illustrates this. 
\begin{figure}[th] 
\centerline{\psfig{file=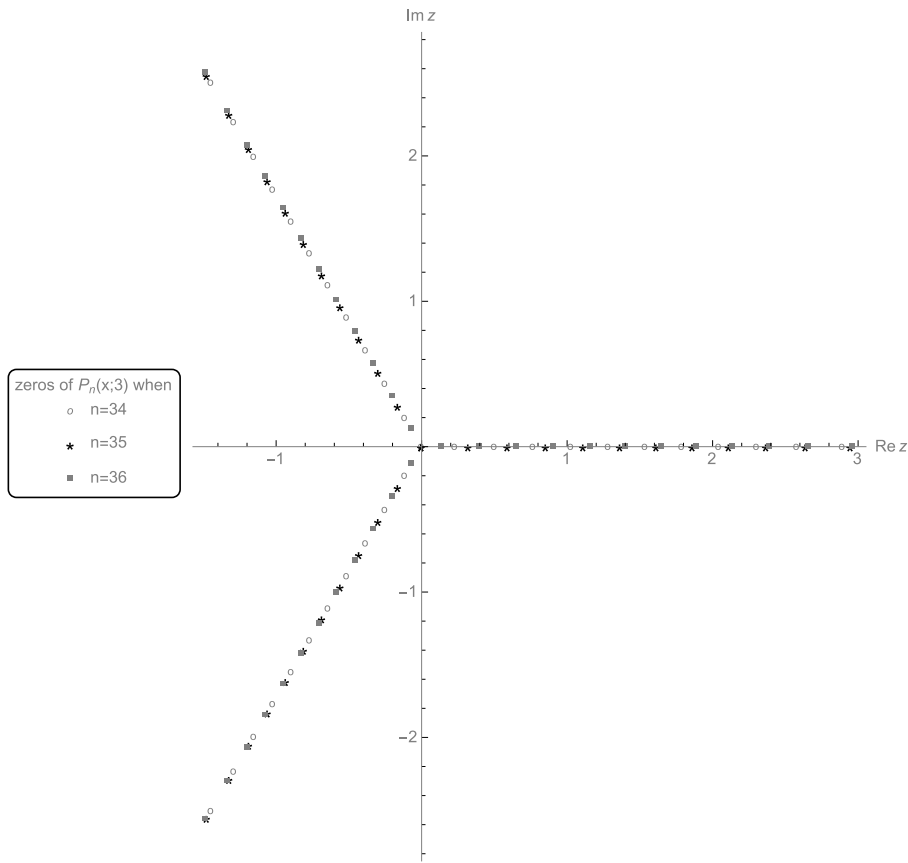,width=12cm}} 
\caption{Zeros of $P_{34}(x;\mu)$ (circle), $P_{35}(x;\mu)$ (star) and $P_{36}(x;\mu)$ (square) with $\mu=3$, where $P_{n}(x;\mu)$ is the $2$-\OPS\ studied in case B1. }
\label{zeros PB1}
\end{figure}

 In the light of Theorem \ref{thm:3sym 2OPS}, the two orthogonality weights can be expressed via the {\it confluent hypergeometric function of the second kind} $\mathbf{U}(a,b;x)$, which  admits the following integral representation \cite[(13.4.4)]{DLMF}
$$ 
	\mathbf{U}(a,b;x) = \frac{1}{\Gamma(a)}\int_{0}^{\infty} t^{a-1} (t+1)^{-a+b-1} \e^{-t x}\,  \dd t ,
$$ 
provided that $\Re(a)>0$ and $|\arg (x) |< \pi/2$, whilst 
$$ 
	\mathbf{U}(0,b;x) = 1,
$$ 
and one has the identity  $\mathbf{U}(a,b;x) = x^{1-b}\mathbf{U}(a-b+1,2-b;x)  $.

\begin{proposition}\label{prop: B1 rep} Let $\mu>-1$. The threefold symmetric polynomial sequence $\{P_n(\cdot:\mu)\}_{n\geqslant 0}$ defined by the recurrence relation \eqref{rec rel 3sym}, with $\gamma_n$ given by \eqref{B1: gamma}, is 2-orthogonal with respect to $(u_0,u_1)$ admitting the integral representation  \eqref{int rep for uk classical} where 
\begin{align}
&	\mathcal{U}_0(x) \coloneqq
	\mathcal{U}_0(x;\mu) 
		\ =\  	\frac{3  \Gamma(\tfrac{\mu+2}{3})}{\Gamma(\tfrac{1}{3})\Gamma(\tfrac{2}{3}) }
			\e^{-x^3} \mathbf{U}(\tfrac{\mu}{3},\tfrac{2}{3};  x^3) ,
	 \label{U0 case B1}\\[2mm]
& 	\label{U1 case B1} 
	\mathcal{U}_1(x)\coloneqq\mathcal{U}_1(x;\mu) %= (\mu +2) U_0'(x)+ x^2 U_0(x)
		\ = \ \frac{9    \Gamma \left(\frac{\mu +5}{3}\right) }
			{\Gamma \left(\frac{1}{3}\right) \Gamma\left(\frac{2}{3}\right)}
			   	x^2 \ e^{-x^3}\ \mathbf{U}\left(\tfrac{\mu }{3}+1,\tfrac{5}{3},x^3\right) ,
\end{align}
and $b=+\infty$, 
where $\mathbf{U}$ represents the Kummer (confluent hypergeometric) function of second kind.
\end{proposition}

\begin{proof} Under the assumptions, it follows that $u_0$ and $u_1$ satisfy \eqref{Sym eq for u0 and u1 B1} and $\{P_n(\cdot:\mu)\}_{n\geqslant 0}$ satisfies the recurrence relation \eqref{rec rel 3sym} with $\gamma_n>0$ for all $n\geqslant1$. 
In the light of Theorem \ref{thm:3sym 2OPS}, there exists a function $\mathcal{W}_0$ and a domain $\mathcal{C}$ containing all the zeros of $\{P_n\}_{n\geqslant0}$ so that 
$$
	\langle u_0, f(x) \rangle  = \int_\mathcal{C} f(x) \mathcal{W}_0(x)\, \dd x ,  
$$
is valid for every polynomial $f$. By virtue of Theorem \ref{thm: asym zero} and the asymptotic behavior of $\gamma_n$ for large $n$, the curve that contains all the zeros of $\{P_n\}_{n\geqslant0}$ corresponds to the starlike set $S$ in Fig.~\ref{fig: 3 rays}. 
According to Theorem \ref{thm:int rep classical},  $u_0$ and $u_1$ admit the representation  \eqref{int rep for uk classical},  provided that there exist two twice differentiable functions $\mathcal{U}_0$ and $\mathcal{U}_1$ from $\mathbb{R}$  to $\mathbb{R} $   that are solutions to 
\begin{equation} \label{B1 eq W}
	\begin{cases}
		 \frac{1}{3} \mathcal{U}_0''(x)+ x^2 \mathcal{U}_0'(x) -  (\mu -2) x \mathcal{U}_0(x)   = \lambda_0 g_0(x)
		 & \text{if}\quad \mu>-1,
		\vspace{0.2cm}
		\\		
	\mathcal{U}_1(x)
	=	 \frac{2}{3}   \mathcal{U}_0'(x)+2 x^2 \mathcal{U}_0(x) + \lambda_1 g_1(x) \ 
	& \text{if}\quad \mu\neq 0 \ \text{and} \ \mu>-1,
	\\
	x \mathcal{U}_1^\prime(x) = 2 \mathcal{U}_0^\prime(x) +\lambda_1 g_1(x) 
	& \text{if}\quad \mu=0,
	\end{cases}
\end{equation}
and satisfying
\begin{eqnarray} \label{B: cond for weights U0}
	&&	\lim_{x\to {\omega^j}\infty} f(x) \frac{\dd^l }{\dd x^l}\mathcal{U}_0(x) = 0  
		, \quad\text{for }   j,l\in\{0,1,2\} \ \text{ and } \  f\in\mathcal{P},\\
	&&	\int_{0}^{\infty} \!\! \mathcal{U}_0(x) \dd x=1 , 
	\label{B: cond2 for weights U0} 
\end{eqnarray}
where $\lambda_k \in \mathbb{C}$ (possibly zero) and $g_k\neq0$ are rapidly decreasing functions, locally integrable, representing the null functional ($k=0,1$). For $\lambda_0=0$, the general solution of the first equation in \eqref{B1 eq W} can be written as
\begin{equation}\label{pfB1 gensol}
	y(x) = c_1  \;  \pFq{1}{1}{\frac{2-\mu }{3}}{ \frac{2}{3}}{ -x^3 } +  c_2\;  x \; \pFq{1}{1}{1-\frac{\mu }{3}}{ \frac{4}{3}}{ -x^3 } ,
\end{equation}
with $c_1,c_2$ two arbitrary constants.  Based on  \cite[Eqs. (13.2.39) and (13.2.42)]{DLMF} we have 
$$
	\e^{-x^3} \mathbf{U}\left(\tfrac{\mu}{3},\tfrac{2}{3},x^3\right) = \frac{\Gamma(\frac{1}{3})}{\Gamma(\frac{\mu+1}{3})}  
	\pFq{1}{1}{\frac{2-\mu }{3}}{ \frac{2}{3}}{ -x^3 } 
	+  \frac{\mu\Gamma(-\frac{1}{3})}{3\Gamma(\frac{\mu}{3}+1)}  \; x \; \pFq{1}{1}{1-\frac{\mu }{3}}{ \frac{4}{3}}{ -x^3 } ,
$$
which is valid when $b$ is not an integer. Observe that (see \cite[(13.7.3)]{DLMF})
$$
	G(x;\mu) = \e^{-x^3} \mathbf{U}\left(\tfrac{\mu}{3},\tfrac{2}{3},x^3\right)\sim  \e^{-x^3} x^{-\mu} \sum_{s=0}^{\infty}\frac{{\left(\tfrac{\mu}{3}\right)_{s}}{\left(\tfrac{\mu+1}{3}\right)_{s}}}{s!}(-z)^{-s} ,
$$
and also that \cite[(13.3.27)]{DLMF} 
$$
	\frac{\dd }{\dd x} G(x;\mu) = -3x^2  \e^{-x^3} \mathbf{U}(\tfrac{\mu}{3},\tfrac{5}{3},x^3),
$$
so that, for both $l=0,1$, we obtain $\ds\lim_{x\to \infty} f(x) \frac{\dd^l G(x;\mu) }{\dd x^l} = 0 $ for every polynomial $f$. 

According to \cite[Eq. (7.621.6)]{GR}, the following identity 
\begin{equation}\label{mellin Kummer}
	\int_0^{+\infty}  t^{b-1} \mathbf{U}\left(a,c; t\right) \e^{-t} \, \dd t 
	= \frac{\Gamma(b)\Gamma(b-c+1)}{\Gamma(a+b-c+1)}, 
\end{equation}
holds, provided that  $\Re(b)>\max(0, \Re(c)-1)$. Therefore
$$	
	\int_{0}^{+\infty} x^n \e^{-x^3} \mathbf{U}(\tfrac{\mu}{3},\tfrac{2}{3}; x^3)\,  \dd x
	=\frac{1}{3 } \int_{0}^{+\infty}  t^{\tfrac{n+1}{3}-1}  \ \e^{-t} \mathbf{U}(\tfrac{\mu}{3},\tfrac{2}{3};t) \, \dd t
	= \frac{1}{3 } \frac{\Gamma(\tfrac{n+1}{3})\Gamma(\tfrac{n+2}{3} )}{\Gamma(\frac{n+\mu+2}{3})} , 
$$
for $n\geqslant 0$, which, upon the substitution $n\rightarrow 3n$, becomes 
$$	
	\int_{0}^{+\infty} x^{3n} \e^{-x^3} \mathbf{U}(\tfrac{\mu}{3},\tfrac{2}{3};  x^3)\,  \dd x
	 = \frac{\Gamma(\tfrac{1}{3})\Gamma(\tfrac{2}{3}) }{3  \Gamma(\tfrac{\mu+2}{3})}  
	\ \frac{ (\tfrac{1}{3})_n (\tfrac{2}{3} )_n}{  (\tfrac{\mu+2}{3})_n }   , \qquad n\geqslant 0. 
$$

Consequently, the particular choices of 
$$
	c_1 = \frac{3  \Gamma(\tfrac{\mu+2}{3})}{\Gamma(\tfrac{2}{3})\Gamma(\frac{\mu+1}{3})} 
	\quad \text{and}\quad 
	c_2 
	 = \frac{-3\mu  \Gamma(\tfrac{\mu+2}{3}) }{\Gamma(\frac{\mu}{3}+1)\Gamma(\tfrac{1}{3}) }
$$
in the general solution \eqref{pfB1 gensol} gives the function 
$
	\mathcal{U}_0(x;\mu)
$ in \eqref{U0 case B1}, which 
meets the requirements \eqref{B: cond for weights U0}-\eqref{B: cond2 for weights U0}. Furthermore, we have 
$$ 
	(u_0)_{n} = \left(\int_{\Gamma_0} \dd x- \omega^2  \int_{\Gamma_1}\dd x-\omega  \int_{\Gamma_2} \dd x\right)\left( x^n \ \frac{ \Gamma(\tfrac{\mu+2}{3})}{\Gamma(\tfrac{1}{3})\Gamma(\tfrac{2}{3}) }\e^{-x^3} \mathbf{U}(\tfrac{\mu}{3},\tfrac{2}{3};  x^3)\right),
$$ 
which matches with  \eqref{B mom u0}. 

Now, for the integral representation of the linear functional $u_1$ we consider \eqref{B1 eq W} with $\lambda_1=0$ 
and this gives 
$$ \begin{cases}
\begin{array}{lcl}
	\mathcal{U}_1(x;\mu) 
			& = & \ds -\frac{ (\mu +2)}{  \mu }  \Big( \mathcal{U}_0' (x)+ 3x^2 \mathcal{U}_0(x)\Big) \\
			&=& \ds  \frac{9    \Gamma \left(\frac{\mu +5}{3}\right) }
			{\Gamma \left(\frac{1}{3}\right) \Gamma\left(\frac{2}{3}\right)}
			   	x^2 \ e^{-x^3}\ \mathbf{U}\left(\tfrac{\mu }{3}+1,\tfrac{5}{3},x^3\right)\\[0.8cm]
	\end{array}
				& \text{if } \quad \mu\neq 0, \\
	 \mathcal{U}_1(x;\mu)  = \frac{9    \Gamma \left(\frac{5}{3}\right) }
			{\Gamma \left(\frac{1}{3}\right) \Gamma\left(\frac{2}{3}\right)} \Gamma\left( \frac{2}{3},x^3\right)
	& \text{if}\quad \mu=0,
\end{cases}$$ 
where, for the first line of the latter identity we have used \cite[(13.3.27)]{DLMF} as well as the contiguous relation \cite[(13.3.10)]{DLMF}. 
The second line of the latter identity corresponds to the first line after taking $\mu=0$ because of \cite[Eq. (13.6.6)]{DLMF}, so that we obtain \eqref{U1 case B1}. 
Here, $\Gamma(\alpha,z)$ is the incomplete Gamma function: $\ds\Gamma(\alpha,z) = \int_z^{+\infty} t^{\alpha-1}\e^{-t}\, \dd t \ $ provided that $\alpha>0$. 
Moreover, based on \eqref{mellin Kummer}, we have 
$$
	\int_{0}^{+\infty} x^{3n+1} \mathcal{U}_1(x;\mu)\, \dd x = \frac{(2/3)_n (4/3)_n}{\left(\frac{\mu+5}{3}\right)_n}, 
	\qquad  \quad \mu>-1,
$$
which coincides with the moments of $u_1$ given in \eqref{B mom u1}.
\end{proof}

\subsubsection{Particular cases} 

For the following choices of the parameter $\mu$ the expressions  \eqref{U0 case B1}-\eqref{U1 case B1}  relate to other functions. Namely we have: 

\begin{itemize}
\item $\mu=-1/2$, then 
\begin{eqnarray*}
&&	\mathcal{U}_0(x;-1/2) =  \frac{3}{2} \sqrt{\frac{3}{\pi }} e^{-x^3} U\left(-\frac{1}{6},\frac{2}{3},x^3\right)
\quad \text{and} \quad 
 	\mathcal{U}_1(x;-1/2) =\frac{9 \sqrt{3} e^{-\frac{x^3}{2}} x K_{\frac{1}{3}}\left(\frac{x^3}{2}\right)}{4 \pi },
 \end{eqnarray*}
where $K_\nu$ is the modified Bessel function of the second kind (see \cite[Ch. 10.25]{DLMF}). 

\item $\mu=0$, then 
\begin{align*}
&	\mathcal{U}_0(x;0) 
		\ =\  	\frac{3   }{\Gamma(\tfrac{1}{3})  } \e^{-x^3}  
	\quad \text{and} \quad 
	\mathcal{U}_1(x;0) %= (\mu +2) U_0'(x)+ x^2 U_0(x)
		\ = \ \frac{9    \Gamma \left(\frac{5}{3}\right) }
			{\Gamma \left(\frac{1}{3}\right) \Gamma\left(\frac{2}{3}\right)}
			   	  \Gamma\left(\frac{2}{3};x^3\right). 
\end{align*}

\item $\mu=1$, then
\begin{eqnarray*}
&&	\mathcal{U}_0(x;1) =  \frac{3 \sqrt{3} e^{-\frac{x^3}{2}} \sqrt{x} K_{\frac{1}{6}}\left(\frac{x^3}{2}\right)}{2
   \pi ^{3/2}}
   \quad \text{and} \quad 
    	\mathcal{U}_1(x;1) = \frac{9 \sqrt{3} e^{-x^3} U\left(\frac{2}{3},\frac{1}{3},x^3\right)}{2 \pi }. 
 \end{eqnarray*}
 
\item  $\mu=2$, then  
\begin{eqnarray*}
&&	\mathcal{U}_0(x;2) =  \frac{\sqrt{3} \Gamma \left(\frac{1}{3}\right) }{2 \pi }\Gamma \left(\tfrac{1}{3},x^3\right)
\quad \text{and} \quad 
 	\mathcal{U}_1(x;2) =\frac{2 \sqrt{3} \Gamma \left(\frac{1}{3}\right) }{\pi }\Gamma\left(\tfrac{2}{3},x^3\right).
 \end{eqnarray*}
\end{itemize}

~

\subsubsection{The differential equation}

Following Lemma \ref{prop: 3rd ODE Pn}, the polynomial $P_{n}(x;\mu)$ is a solution of the differential equation 
$$ %\label{B1 3rd ODE}
	-\frac{2}{3}y'''(x)
	+  2 x^2 y''(x)+ 2 x \left(\mu +\frac{3}{4} \left((-1)^n+3\right)-\frac{n}{2}\right) y'(x)
	=2n\left(\mu +\frac{n}{2}+\frac{3 (-1)^n}{4}+\frac{5}{4}\right) y(x) , %, \ n\geqslant 0, 
$$ 
whose general solution can be written as 
\begin{eqnarray*}
y(x) 
&=&	c_1 \, 
   \pFq{2}{2}{-\frac{n}{3},\frac{n}{6}+\frac{(-1)^n}{4}+\frac{\mu }{3}+\frac{5}{12}}{\frac{1}{3},\frac{2}{3}}
   	{x^3}
	+ 	c_2  \; x \; 
		\pFq{2}{2}{\frac{1}{3}-\frac{n}{3},\frac{n}{6}+\frac{(-1)^n}{4}+\frac{\mu}{3}+\frac{3}{4}}
		{\frac{2}{3},\frac{4}{3}}{x^3}\\
&&	+ c_3\;  x^2 \; 
	\pFq{2}{2}{\frac{2}{3}-\frac{n}{3},\frac{n}{6}+\frac{(-1)^n}{4}+\frac{\mu}{3}+\frac{13}{12}}{\frac{4}{3},\frac{5}{3}}{x^3} .
\end{eqnarray*}

For each positive integer $n$, there is only one (monic) polynomial solution, but it is not always the same: it depends on whether $n$ equals $0,1$ or $2 \bmod {3}$. 
To be precise, we have 
\begin{eqnarray*}
	 P_{3n}(x;\mu) &\coloneqq & P_n^{[0]}(x^3;\mu) \\
	&=&  \frac{(-1)^n  \left(\frac{1}{3}\right)_n \left(\frac{2}{3}\right)_n}{\left(\frac{n}{2}+\frac{1}{4} (-1)^{3 n}+\frac{\mu
   }{3}+\frac{5}{12}\right){}_n}
		\pFq{2}{2}{-n,\frac{n}{2}+\frac{(-1)^n}{4}+\frac{\mu }{3}+\frac{5}{12}}{\frac{1}{3},\frac{2}{3}}
   	{x^3} , \\[0.2cm]
	P_{3n+1}(x;\mu)&\coloneqq & x P_n^{[1]}(x^3;\mu)\\
		& =&  x \ \frac{(-1)^n   \left(\frac{2}{3}\right)_n \left(\frac{4}{3}\right)_n}{\left(\frac{n}{2}+\frac{1}{4} (-1)^{n+1}+\frac{\mu
   }{3}+\frac{11}{12}\right){}_n}		
   		\pFq{2}{2}{-n,\frac{\mu }{3}+\frac{n}{2}+\frac{1}{4} (-1)^{n+1}+\frac{11}{12}}{\frac{2}{3},\frac{4}{3}}
   	{x^3}, \\[0.2cm]
	P_{3n+2}(x;\mu)&\coloneqq & x^2 P_n^{[2]}(x^3;\mu) \\
		&=& x^2\   \frac{(-1)^n  \left(\frac{4}{3}\right)_n \left(\frac{5}{3}\right)_n }{\left(\frac{n}{2}+\frac{1}{4} (-1)^{n+2}+\frac{\mu
   }{3}+\frac{17}{12}\right){}_n}
		\pFq{2}{2}{-n, \frac{\mu }{3}+\frac{n}{2}+\frac{1}{4} (-1)^{n+2}+\frac{17}{12}}{\frac{4}{3},\frac{5}{3}}
   	{x^3} .
\end{eqnarray*}
 These polynomial sequences $\{P_n^{[k]}(\cdot;\mu)\}_{n\geqslant 0}$, with $k\in\{0,1,2\}$, are precisely the $2$-\OPS s in the cubic decomposition of $\{P_n(\cdot;\mu)\}_{n\geqslant 0}$. In fact, from Lemma \ref{lem: 2sym components rec rel}, these three $2$-\OPS s are not threefold symmetric and satisfy the recurrence relation \eqref{CD rec rel},  whose recurrence coefficients are given in the Appendix for completeness.  These coefficients have been computed in  \cite[Tableau 4, 8 and 12 - Case C]{DM92}, for a different choice of the "free" parameter $\gamma_1$. We have included them here for a matter of completeness. 

%%%%%%%%%%%%%%%

\subsubsection{The sequence of derivatives}

Following Theorem \ref{thm: sym rc} along with Theorem \ref{th: gammas},  $\{Q_n(\cdot;\mu)\coloneqq\frac{1}{n+1} P_{n+1}'(x;\mu)\}_{n\geqslant 0}$ is $2$-orthogonal for the vector functional $(v_0,v_1)$, which admits an integral representation via two weight functions $(\widetilde{\mathcal{W}}_0(\cdot;\mu),\widetilde{\mathcal{W}}_1(\cdot;\mu))$ with support on the three-starlike set $S$ which are  given by 
$$
	\left[ \begin{array}{l}
	\widetilde{\mathcal{W}}_0(x;\mu)\\
	\widetilde{\mathcal{W}}_1(x;\mu)
	\end{array}\right]
	 =\left[\begin{array}{cc}
 \frac{\mu+2}{\mu +1} & -\frac{x}{\mu +1} \\
 0 & 1 \\
\end{array}\right]
\left[ \begin{array}{l}
	\mathcal{W}_0(x;\mu)\\
	\mathcal{W}_1(x;\mu)
	\end{array}\right]. 
$$
Proposition \ref{prop: B1 rep} allows us to conclude 
$$
	\langle v_k,f \rangle 
	=  \left(\int_{\Gamma_0} \dd x- \omega^2  \int_{\Gamma_1}\dd x-\omega  \int_{\Gamma_2} \dd x\right)\left( f(x) \ \mathcal{V}_k(x;\mu)\right), \qquad \forall f\in\mathcal{P}, \   k=0,1, 
$$
where 
\begin{eqnarray*}
	\mathcal{V}_0(x;\mu) &=&   \frac{\mu+2}{\mu +1} \mathcal{U}_0(x;\mu)  -\frac{x}{\mu +1}\mathcal{U}_1(x;\mu) \ = \   \mathcal{U}_0(x;\mu+3) ,\\
	\mathcal{V}_1(x;\mu) &=& \mathcal{U}_1(x;\mu). 
\end{eqnarray*}
Furthermore,  $\{Q_n(\cdot;\mu)\}_{n\geqslant 0}$ satisfies the second order recurrence relation \eqref{Qn rec rel sym} where 
\begin{equation}\label{sym gamma tilde B1}
\begin{array}{l}
	\ds \widetilde{\gamma}_{2n }\coloneqq\widetilde{\gamma}_{2n }(\mu) =\frac{2}{3}  \frac{ n (2 n+1)}{(3 n+2+\mu)} , \qquad n\geqslant 1,  \vspace{0.3cm}\\
	\ds \widetilde{\gamma}_{2n+1 }\coloneqq\widetilde{\gamma}_{2n +1}(\mu) =\frac{2}{3} \frac{(n+1) (2 n+1) (n +2+\mu) }{(3 n+2+\mu) (3 n+5+ \mu)}  , \qquad n\geqslant 0.   \\
\end{array}
\end{equation}

For each integer $n\geqslant 0$, the polynomial $Q_n(x;\mu)$ is a solution to the differential equation 
$$
	-\frac{2}{3} y^{(3)}(x)+2 x^2 y''(x) -\frac{1}{2} x \left(-4 \mu +3 \left((-1)^n-5\right)+2 n\right) y'(x)
	= \frac{1}{2} n \left(4 \mu +2   n-3 (-1)^n+11\right) y(x),
$$
from which we deduce 
\begin{eqnarray*}
	Q_{3n}(x;\mu) &=& \frac{(-1)^n (\frac{1}{3})_n (\frac{2}{3})_n}{(\frac{n}{2}-\frac{(-1)^{n}}{4} +\frac{\mu   }{3}+\frac{11}{12})_n} \; 
	\pFq{2}{2}{-n,\frac{n}{2}-\frac{(-1)^{n}}{4} +\frac{\mu
   }{3}+\frac{11}{12}}{\frac{1}{3},\frac{2}{3}}{x^3}, \\
	Q_{3n+1}(x;\mu) &=& x \, \frac{(-1)^n (\frac{2}{3})_n (\frac{4}{3})_n}{(\frac{\mu }{3}+\frac{n}{2}+\frac{(-1)^{n}}{4} +\frac{17}{12})_n} \;  
		\pFq{2}{2}{-n,\frac{\mu }{3}+\frac{n}{2}+\frac{(-1)^{n}}{4} +\frac{17}{12}}{\frac{2}{3},\frac{4}{3}}{x^3} ,\\ 
	Q_{3n+2}(x;\mu) &=& x^2 \,  \frac{(-1)^n (\frac{4}{3})_n (\frac{5}{3})_n}{(\frac{\mu }{3}+\frac{n}{2}-\frac{(-1)^{n}}{4} +\frac{23}{12})_n} \;
	 	\pFq{2}{2}{-n,\frac{\mu }{3}+\frac{n}{2}-\frac{(-1)^{n}}{4} +\frac{23}{12}}{\frac{4}{3},\frac{5}{3}}{x^3}
		,\qquad n\geqslant 0.
\end{eqnarray*}

 It turns out that the forthcoming case $\text{B}_2$ is closely connected to this case $\text{B}_1$: the sequence of derivatives in case $B_1$ belongs to case $B_2$.

\subsection{Case  $B_2$}\label{subs: caseB2} ~

For this family  $\vartheta_{2n+1}=1$ and we can write $\vartheta_{2n}=\frac{n+\rho+1}{n+ \rho}$, 
under the relation $\vartheta_2 = \frac{\rho +2}{\rho +1}$ with $\rho\neq-n$.  As a result, the threefold symmetric Hahn-classical sequence 
$\{P_n(\cdot;\rho;\gamma_1)\}_{n\geqslant 0}$ satisfies the recurrence \eqref{rec rel 3sym} with 
$$
\begin{array}{l}
	\ds \gamma_{2n } =\frac{n(2n+1) (\rho+3)}{(3n+\rho)} \gamma_1, \qquad n\geqslant 1,  \vspace{0.3cm}\\
	\ds \gamma_{2n +1} = \frac{(n+1)(2n+1)(n+\rho) (\rho+3)}{(3n+\rho)(3n+\rho+3)} \gamma_1, \qquad n\geqslant 0.  
\end{array}
$$
Analogously to the former cases, the parameter $\gamma_1$ is redundant for the study, and therefore we can choose a representative value for $\gamma_1$. 
Here we set $\gamma_1 = \frac{2}{3(\rho+3)}$. Hence the $2$-\OPS\ $\{P_n(\cdot;\rho)\coloneqq P_n(\cdot;\rho;\frac{2}{3(\rho+3)})\}_{n\geqslant 0}$ satisfies \eqref{rec rel 3sym}  with 
\begin{equation}\label{sym gamma B2}
\begin{array}{l}
	\ds \gamma_{2n }\coloneqq \gamma_{2n }(\rho)  =\frac{2n(2n+1) }{3(3n+\rho)} , \qquad n\geqslant 1,  \vspace{0.3cm}\\
	\ds \gamma_{2n +1} \coloneqq \gamma_{2n+1 }(\rho)= \frac{2(n+1)(2n+1)(n+\rho) }{3(3n+\rho)(3n+\rho+3)} , \qquad n\geqslant 0.  
\end{array}
\end{equation}
Furthermore, the $2$-\OPS\ $\{Q_n(x;\rho)\coloneqq \frac{1}{n+1}P_{n+1}'(x;\rho)\}_{n\geqslant 0}$ satisfies \eqref{Qn rec rel sym} with 
\begin{equation}\label{sym gamma tilde B2}
\begin{array}{l}
	\ds \widetilde{\gamma}_{2n }\coloneqq\widetilde{\gamma}_{2n }(\rho) =  \frac{2 n (2 n+1)  (n+\rho +1)}{3(3 n+\rho ) (3 n+\rho +3)}, \qquad n\geqslant 1,  \vspace{0.3cm}\\
	\ds \widetilde{\gamma}_{2n +1}\coloneqq \widetilde{\gamma}_{2n }(\rho)=\frac{2 (n+1) (2 n+1) }{3(3 n+\rho +3)}, \qquad n\geqslant 0.   
\end{array}
\end{equation}

Observe that the recurrence coefficients uniquely determine a $2$-\OPS. A comparison between \eqref{B1: gamma} and \eqref{sym gamma tilde B2} 
and between \eqref{sym gamma tilde B1} and \eqref{sym gamma B2} shows that 
$$
	Q_n^{\text{case }\text{B}_2}(x;\mu) = P_n^{\text{case }\text{B}_1}(x;\mu+1) , \qquad \text{for all } n\geqslant 0, 
$$
while 
$$
	Q_n^{\text{case }\text{B}_1}(x;\mu) = P_n^{\text{case }\text{B}_2}(x;\mu+2) , \qquad \text{for all } n\geqslant 0,
$$
 where the notation $P_n^{\text{case }\text{B}_1}$ and $ P_n^{\text{case }\text{B}_2}$ is used for the threefold $2$-\OPS s defined by the recurrence coefficients \eqref{B1: gamma}, in case $\text{B}_1$, and  \eqref{sym gamma B2}, in case $\text{B}_2$, respectively. Likewise, $Q_n^{\text{case }\text{B}_1} $ and $Q_n^{\text{case }\text{B}_1} $ denote the monic derivatives of $P_n^{\text{case }\text{B}_1}$ and $ P_n^{\text{case }\text{B}_2}$, respectively. 

From this, we conclude that the  the %$2$-\OPS s 
polynomial sequences $\{Q_n^{\text{case }\text{B}_1}(x;\mu)\coloneqq \frac{1}{n+1}\frac{\dd}{\dd x}P_{n+1}^{\text{case }\text{B}_1}(x;\mu)\}_{n\geqslant 0}$
and $\{Q_n^{\text{case }\text{B}_2}(x;\rho)\coloneqq \frac{1}{n+1}\frac{\dd}{\dd x}P_{n+1}^{\text{case }\text{B}_2}(x;\rho)\}_{n\geqslant 0}$ are also Hahn-classical. Furthermore, we have 
$$
	\frac{1}{(n+2)(n+1)}\frac{\dd^2}{\dd x^2}P_{n+2}(x;\mu)
	= P_{n}(x;\mu+3) ,
$$
in both cases $\text{B}_1$ and $\text{B}_2$.
These observations also serve as an alternative to the proof of the following result:

\begin{proposition}  Let $\rho>0$. The threefold symmetric polynomial sequence $\{P_n(\cdot: \rho)\}_{n\geqslant 0}$ defined by the recurrence relation \eqref{rec rel 3sym} with $\gamma_n$ given by \eqref{sym gamma B2} is 2-orthogonal with respect to $(u_0,u_1)$ admitting the integral representation  \eqref{int rep for uk classical} where 
\begin{eqnarray*}
&&	\mathcal{U}_0(x) \coloneqq\mathcal{U}_0(x; \rho) 
		= \frac{3  \Gamma \left(\frac{\rho +3}{3}\right)}{\Gamma \left(\frac{1}{3}\right) \Gamma \left(\frac{2}{3}\right)}
		 \e^{-x^3}\mathbf{U}\left(\frac{\rho +1}{3},\frac{2}{3},x^3\right)	,
%	 \label{U0 case B2}
\\
&& 	\mathcal{U}_1(x)\coloneqq\mathcal{U}_1(x; \rho) %= (\mu +2) U_0'(x)+ x^2 U_0(x)
	=\frac{9  \Gamma \left(\frac{\rho +3}{3}\right) }{\Gamma \left(\frac{1}{3}\right) \Gamma \left(\frac{2}{3}\right)}
	x^2 \e^{-x^3} \mathbf{U}\left(\frac{\rho +1}{3},\frac{5}{3},x^3\right),
\end{eqnarray*}
and $b=+\infty$, 
where $\mathbf{U}$ represents the Kummer (confluent hypergeometric) function of second kind.
\end{proposition}
\begin{proof} The proof is entirely analogous to the proof of Proposition \ref{prop: B1 rep}.
\end{proof}

In fact, we have 
$$
	\mathcal{U}_0^{\text{case } \text{B}_2}(x; \rho) 
	= \mathcal{U}_0^{\text{case } \text{B}_1}(x; \rho+1) 
	\quad \text{and}\quad 
	\mathcal{U}_1^{\text{case } \text{B}_2}(x; \rho) 
	= \mathcal{U}_1^{\text{case } \text{B}_1}(x; \rho-2),
$$
provided that $\rho>1$.

\subsection{Case C}\label{subs: caseC} ~ 

 Here we have 
$$ 
	\vartheta_{2n-1} = \frac{n+\mu+1}{n+\mu} \quad \text{and}\quad \vartheta_{2n} =  \frac{n+\rho+1}{n+\rho} \ , \qquad n\geqslant 1. 
$$ 
so that  \eqref{sym gammas}  reads as  
$$ 
\begin{array}{l}
	\ds \gamma_{2n } =  \frac{2n(2n+1)(n+\mu) }
						{(3n+\mu-1)(3n+\mu+2)(3n+\rho)} \frac{ (\mu+2)(\rho+3)\gamma_1}{2}, \qquad n\geqslant 1,  \vspace{0.3cm}\\
	\ds \gamma_{2n +1} = \frac{2(n+1)(2n+1)(n+\rho)  }
						{(3n+\mu+2)(3n+\rho)(3n+\rho+3)} \frac{ (\mu+2)(\rho+3)\gamma_1}{2}, \qquad n\geqslant 0.   \\
\end{array}
$$ 

The Hahn-classical polynomial sequence obtained under these assumptions clearly depends  on the pair of parameters $(\mu,\rho)$  and on $\gamma_1$, so it makes sense to incorporate this information, and therefore we shall refer to it as $\{P_n(\cdot;\mu,\rho;\gamma_1)\}_{n\geqslant 0}$. 
Similar to the precedent cases, without loss of generality, it suffices to study the sequence for a particular choice of $\gamma_1$. 
A scaling of the variable would then reproduce all the other sequences $\{P_n(\cdot;\mu,\rho;\gamma_1)\}_{n\geqslant 0}$ within this equivalence class. Hence, 
we set $\gamma_1 = \dfrac{2}{ (\mu+2)(\rho+3) }$, and the analysis is on  
$$ 
	\left\{P_n(\cdot;\mu,\rho)\coloneqq P_n\left(\cdot;\mu,\rho; \tfrac{2}{ (\mu+2)(\rho+3)}\right)\right\}_{n\geqslant 0},
$$ 
which satisfies the recurrence relation \eqref{rec rel 3sym} where the $\gamma$-coefficients are given by
\begin{equation}\label{sym gamma C rep}
\begin{array}{l}
	\ds \gamma_{2n } \coloneqq\gamma_{2n } (\mu,\rho) =  \frac{2n(2n+1)(n+\mu) }
						{(3n+\mu-1)(3n+\mu+2)(3n+\rho)} , \qquad n\geqslant 1, \vspace{0.3cm}\\
	\ds \gamma_{2n +1} \coloneqq\gamma_{2n } (\mu,\rho)= \frac{2(n+1)(2n+1)(n+\rho)  }
						{(3n+\mu+2)(3n+\rho)(3n+\rho+3)}  , \qquad n\geqslant 0.   \\
\end{array}
\end{equation}
Observe that $\gamma_{n+1}>0$ for all $n\geqslant 0$ provided that $\mu+1, \rho>0$ and that
$$ 
	\gamma_n = \frac{4}{27} +o(1) , \quad \text{as}\quad n\to +\infty. 
$$ 
Theorem \ref{thm: asym zero} ensures that the largest zero of $P_n(x;\mu,\rho)$ in absolute value is always less than or equal to $1$. Figure \ref{Fig Largest zero C} illustrates the curve defined by these zeros. 
\begin{figure}[th] 
\centerline{\psfig{file=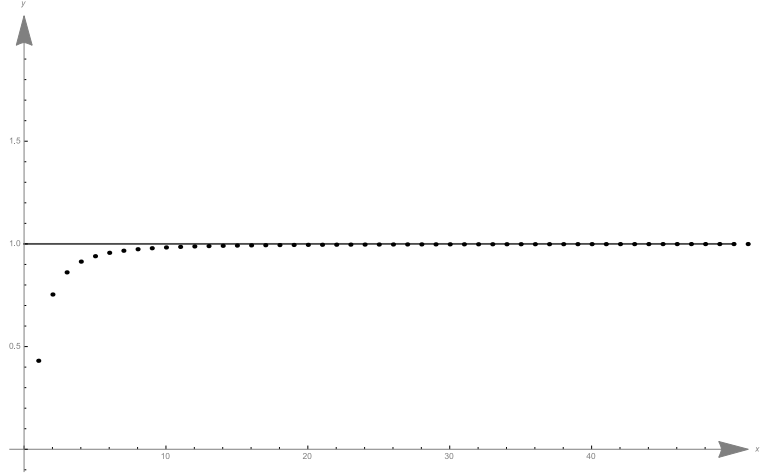,width=10cm}} 
	\caption{Plot of the largest zero in absolute value of $P_{3n}(x;\mu,\rho)$ for each $n=1,\ldots,100$ when $\mu=2,\rho=3$, against the curve $y=1$. }
	\label{Fig Largest zero C}
\end{figure}

Regarding the orthogonality measures for the polynomial sequence $\{P_n(\cdot;\mu,\rho)\}_{n\geqslant 0}$  we have: 

\begin{proposition}\label{prop: C rep} Let $\mu>-1$ and $\rho>0$. The threefold symmetric polynomial sequence $\{P_n(\cdot;\mu, \rho)\}_{n\geqslant 0}$ defined by the recurrence relation \eqref{rec rel 3sym} with $\gamma_n$ given by \eqref{sym gamma C rep}  is $2$-orthogonal with respect to $(u_0,u_1)$, admitting the integral representation  \eqref{int rep for uk classical} with $b=1$ and  
\begin{eqnarray}
&&	\mathcal{U}_0(x) \coloneqq\mathcal{U}_0(x;\mu, \rho) 
		=  \frac{3\Gamma \left(\frac{\mu+2}{3}\right)\Gamma \left(\frac{\rho }{3}+1\right)}
		{\Gamma \left(\frac{1}{3}\right)\Gamma \left(\frac{2}{3}\right)\Gamma\left(\frac{\mu +\rho +2}{3} \right)  }
	(1-x^3)^{\frac{\mu+\rho-1}{3}}  \ \pfq {\frac{\mu}{3}, \frac{\rho+1}{3}}{\frac{\mu+\rho+2}{3} }{1-x^3}, 
	 \label{U0 case C}\\
&& 	\mathcal{U}_1(x)\coloneqq\mathcal{U}_1(x; \mu,\rho) %= (\mu +2) U_0'(x)+ x^2 U_0(x)
	= \frac{3 \Gamma \left(\frac{\mu+5}{3}\right)\Gamma \left(\frac{\rho }{3}+1\right)}
		{\Gamma \left(\frac{2}{3}\right)\Gamma \left(\frac{4}{3}\right)\Gamma\left(\frac{\mu +\rho +2}{3} \right)  }x^2 (1-x^3)^{\frac{\mu+\rho-1}{3}}
		  \pfq {\frac{\mu}{3}+1, \frac{\rho+1}{3}}{\frac{\mu+\rho+2}{3} }{1-x^3}. \quad
	 \label{U1 case C}
\end{eqnarray} 
\end{proposition}

\begin{proof} According to Theorem \ref{thm: sym rc}, $\{P_n(\cdot;\mu, \rho)\}_{n\geqslant 0}$ is $2$-orthogonal with respect to the vector of linear functionals $(u_0,u_1)$ fulfilling the distributional equations 
\begin{align*}
	&\left(  1- x^3\right)u_0 '' +  x^2 (\mu +\rho -4) u_0 '-  (\mu -2) (\rho -1) x u_0  =0, 
	\\[0.5cm]
	&\begin{cases}
	  \frac{\mu   } {(\mu +2)  } u_1 
	   =  \left( x^3-1\right)u_0 ' 
	-  (\rho -1) x^2 \ u_0 & \text{if} \quad \mu\neq 0, \\
	x u_1' = 2 u_0'	
		& \text{if} \quad \mu= 0, 
	\end{cases}
\end{align*}
which yields 
$$\begin{array}{l}
	(u_0)_{3n} = \ds \frac{   \left(\frac{1}{3}\right)_n \left(\frac{2}{3}\right)_n  }{\left(\frac{\mu +2}{3}\right)_n \left(\frac{\rho +3}{3}\right)_n}
	,\quad 
	(u_0)_{3n+1} =  (u_0)_{3n+2} =0, \qquad n\geqslant0,
\end{array}$$ as well as  
$$\begin{array}{l}
	(u_1)_{3n+1} = \ds \frac{   \left(\frac{2}{3}\right)_n \left(\frac{4}{3}\right)_n  }{\left(\frac{\mu +5}{3}\right)_n \left(\frac{\rho +3}{3}\right)_n} ,
	\quad 
   	(u_1)_{3n} =  (u_1)_{3n+2} =0, \qquad n\geqslant0,
\end{array}
$$
because $(u_0)_0=(u_1)_1=1$ and $(u_0)_1=(u_0)_2=(u_1)_0=(u_1)_2=0$.

Since $\lim_{n\to \infty} \gamma_n = \frac{4}{27}$, then by virtue of Theorem \ref{thm:int rep classical}, $u_0$ and $u_1$  admit the integral representation 
\eqref{int rep for uk classical} with $b=1$, provided that there exists a pair of functions $\mathcal{U}_0$ and $\mathcal{U}_1$ with support on $[0,1)$ satisfying \eqref{weight eq U0}-\eqref{weight U1} and subject to the boundary condition \eqref{Sym cond for weights U0}. To be precise, under the assumptions, we seek $\mathcal{U}_0(x;\mu,\rho)$ and $\mathcal{U}_1(x;\mu,\rho)$ defined on defined on $[0,1)$ such that 
\begin{align}\label{C: eq diff U0 lambda}
	&
	\left(  1- x^3\right)\mathcal{U}_0''(x;\mu,\rho)+  x^2 (\mu +\rho -4) \mathcal{U}_0'(x;\mu,\rho)-  (\mu -2) (\rho -1) x \mathcal{U}_0(x;\mu,\rho)  = \lambda_0 g_0(x),
		 \\[0.3cm]
	& \begin{cases}
	 \ds \frac{\mu   } {(\mu +2)  }\mathcal{U}_1(x;\mu,\rho)
	   =  \left( x^3-1\right)\mathcal{U}_0'(x;\mu,\rho) 
	-  (\rho -1) x^2 \mathcal{U}_0(x;\mu,\rho) + \lambda_1 g_1(x)  & \text{if} \quad \mu\neq 0, \\[0.4cm]
	x \mathcal{U}_1^{\prime}(x;\mu,\rho)= 2\mathcal{U}_0^{\prime}(x;\mu,\rho)+ \lambda_1 g_1(x)& \text{if} \quad \mu= 0, 
	\end{cases}
	\label{C: eq diff U1 lambda}
\end{align}
where $\lambda_j$, for $j\in\{0,1\}$, are constants (possibly zero) and $g_j$ are functions representing the null linear functional, { i.e.}, 
$$
	 \int_{\mathcal{C}} g_j(x) x^n\, \dd x = 0, \qquad n\geqslant 0, 
$$
and, in addition, 
\begin{eqnarray} \label{C: cond for weights U0}
	&&	\lim_{x\to 1} f(x) \frac{\dd^l }{\dd x^l}\mathcal{U}_0(x;\mu,\rho) = 0  
		, \quad\text{for any} \  l\in\{0,1\} \ \text{ and } \  f\in\mathcal{P},\\
	&&	\int_{0}^{1} x^{3n} \mathcal{U}_0(x;\mu,\rho)\, \dd x
		=  \frac{   \left(\frac{1}{3}\right)_n \left(\frac{2}{3}\right)_n  }{\left(\frac{\mu +2}{3}\right)_n \left(\frac{\rho +3}{3}\right)_n}
		,\qquad n\geqslant 0. 
	\label{C: moments U0}
\end{eqnarray}

For $\lambda_0=0$, the differential equation \eqref{C: eq diff U0 lambda} becomes 
$$ 
	 \left(  1- x^3\right) \mathcal{U}_0''(x;\mu,\rho)+  x^2 (\mu +\rho -4) \mathcal{U}_0'(x;\mu,\rho)-  (\mu -2) (\rho -1) x \mathcal{U}_0(x;\mu,\rho)  =0 ,
$$ 
whose general solution can be written as 
$$
	\mathcal{U}_0(x;\mu,\rho) = c_1  \ \ \pfq{\frac{2-\mu}{3},\frac{1-\rho}{3}}{\frac{2}{3} }{x^3} 
	+ c_2  \ x \ \pfq {1-\frac{\mu}{3}, \frac{2-\rho}{3}}{\frac{4}{3} }{x^3} ,
$$
with integration constants $c_1$ and $c_2$ which must be chosen such that \eqref{C: cond for weights U0}-\eqref{C: moments U0} hold.  
By taking 
$$ 
	c_2=  -\frac{  \Gamma \left(\frac{2}{3}\right) \Gamma \left(\frac{\mu +1}{3}\right)
   \Gamma \left(\frac{\rho +2}{3}\right)}{\Gamma \left(\frac{4}{3}\right) \Gamma
   \left(\frac{\mu }{3}\right) \Gamma \left(\frac{\rho +1}{3}\right)} c_1, 
$$ 
we have 
\begin{eqnarray*}
	\mathcal{U}_0(x;\mu,\rho) &=& c_1\frac{ \Gamma \left(\frac{\mu+1}{3}\right) \Gamma \left(\frac{\rho+2}{3}\right)}{ \Gamma \left(\frac{1}{3}\right)} \\
	&&\left(   \frac{ \Gamma \left(\frac{4}{3}\right)}{ \Gamma \left(\frac{\mu+1}{3}\right) \Gamma \left(\frac{\rho+2}{3}\right)} 
			\ \pfq{\frac{2-\mu}{3},\frac{1-\rho}{3}}{\frac{2}{3} }{x^3} 
	- \frac{ \Gamma \left(\frac{2}{3}\right) \ x }{ \Gamma \left(\frac{\mu}{3}\right) \Gamma \left(\frac{\rho+1}{3}\right)}  
			 \ \pfq {1-\frac{\mu}{3}, \frac{2-\rho}{3}}{\frac{4}{3} }{x^3} \right)\\
	&=& c_1\frac{ \Gamma \left(\frac{\mu+1}{3}\right) \Gamma \left(\frac{\rho+2}{3}\right)}{ \Gamma \left(\frac{1}{3}\right)}
	(1-x^3)^{\frac{\mu+\rho-1}{3}}  \ \pfq {\frac{\mu}{3}, \frac{\rho+1}{3}}{\frac{\mu+\rho+2}{3} }{1-x^3} ,
\end{eqnarray*}
where, for the latter identity, we have used \cite[(15.10.18)]{DLMF}. In fact, for this choice of $c_2$, it follows that  $\lim\limits_{x\to 1^{-}}\mathcal{U}_0(x;\mu,\rho) =0 $, because 
$$ 
	\pFq{2}{1}{a,b}{c}{1}= \frac{\Gamma(c) \Gamma(c-a-b)}{\Gamma(c-a)\Gamma(c-b)} 
$$ 
is valid whenever $\Re(c-a-b)>0$. Moreover, condition \eqref{C: cond for weights U0} is fulfilled, and, in addition, we successively have 
\begin{eqnarray*}
	\int_0^1 x^{3n} \mathcal{U}_0(x;\mu,\rho)\,\dd x 
	&=& c_1\frac{ \Gamma \left(\frac{\mu+1}{3}\right) \Gamma \left(\frac{\rho+2}{3}\right)}{ \Gamma \left(\frac{1}{3}\right)}
		\int_0^1 x^{3n} (1-x^3)^{\frac{\mu+\rho-1}{3}}  \ \pfq {\frac{\mu}{3}, \frac{\rho+1}{3}}{\frac{\mu+\rho+2}{3} }{1-x^3} \, \dd x\\
	&=& c_1\frac{ \Gamma \left(\frac{\mu+1}{3}\right) \Gamma \left(\frac{\rho+2}{3}\right)}{3 \Gamma \left(\frac{1}{3}\right)}
		\int_0^1 (1-x)^{n-\frac{2}{3}} x^{\frac{\mu+\rho-1}{3}}  \ \pfq {\frac{\mu}{3}, \frac{\rho+1}{3}}{\frac{\mu+\rho+2}{3} }{x} , \dd x\\
	&=& c_1\frac{ \Gamma \left(\frac{2}{3}\right)\Gamma \left(\frac{\mu+1}{3}\right) \Gamma \left(\frac{\rho+2}{3}\right)\Gamma\left(\frac{\mu +\rho +2}{3} \right)}
				{3  \Gamma \left(\frac{\mu+2}{3}\right)\Gamma \left(\frac{\rho }{3}+1\right)}
		\frac{\left(\frac{1}{3}\right)_n \left(\frac{2}{3}\right)_n }
		{\left(\frac{\mu+2}{3}\right)_n \left(\frac{\rho }{3}+1\right)_n},
\end{eqnarray*}
where, for the last identity, we have used \cite[(7.512.4)]{GR}. If we take
$$
	c_1 = \frac{3  \Gamma \left(\frac{\mu+2}{3}\right)\Gamma \left(\frac{\rho }{3}+1\right)}
	{ \Gamma \left(\frac{2}{3}\right)\Gamma \left(\frac{\mu+1}{3}\right) \Gamma \left(\frac{\rho+2}{3}\right)\Gamma\left(\frac{\mu +\rho +2}{3} \right)}
$$
then \eqref{C: moments U0} is fulfilled. As a result, the first orthogonality measure can be represented as in \eqref{int rep for uk classical} (case $k=0$) where $\mathcal{U}_0$ is given by \eqref{U0 case C}. 

Now, in order to obtain the second orthogonality measure, we take $\lambda_1=0$ in \eqref{C: eq diff U1 lambda}, which involves a derivative of 
$\mathcal{U}_0(x;\mu,\rho) $, that, according to \cite[(15.5.4)]{DLMF}, corresponds to 
$$ \begin{multlined}
	\mathcal{U}_0'(x;\mu,\rho) = -  \frac{3\Gamma \left(\frac{\mu+2}{3}\right)\Gamma \left(\frac{\rho }{3}+1\right)}
		{\Gamma \left(\frac{1}{3}\right)\Gamma \left(\frac{2}{3}\right)\Gamma\left(\frac{\mu +\rho +2}{3} \right)  }
		\left(\mu+\rho-1\right)
	x^2 (1-x^3)^{\frac{\mu+\rho+2}{3}-2}  \\
	\times \pfq {\frac{\mu}{3}, \frac{\rho+1}{3}}{\frac{\mu+\rho+2}{3} -1}{1-x^3} .
\end{multlined}
$$ 
For $\mu\neq0$, it follows from \eqref{C: eq diff U1 lambda} that 
\begin{align*}
	 \frac{\mu   } {(\mu +2)  }\mathcal{U}_1(x;\mu,\rho)
	=&  \frac{9\Gamma \left(\frac{\mu+2}{3}\right)\Gamma \left(\frac{\rho }{3}+1\right)}
		{\Gamma \left(\frac{1}{3}\right)\Gamma \left(\frac{2}{3}\right)\Gamma\left(\frac{\mu +\rho +2}{3} \right)  }x^2 (1-x^3)^{\frac{\mu+\rho-1}{3}} \\
	&	\qquad \Bigg( \left(\tfrac{\mu+\rho-1}{3}\right)
	 \ \pfq {\frac{\mu}{3}, \frac{\rho+1}{3}}{\frac{\mu+\rho-1}{3} }{1-x^3}
	 -  \left(\tfrac{\rho-1}{3}\right) \pfq {\frac{\mu}{3}, \frac{\rho+1}{3}}{\frac{\mu+\rho+2}{3} }{1-x^3}\Bigg)\\[0.3cm]
	= & \frac{3\mu\Gamma \left(\frac{\mu+2}{3}\right)\Gamma \left(\frac{\rho }{3}+1\right)}
		{\Gamma \left(\frac{1}{3}\right)\Gamma \left(\frac{2}{3}\right)\Gamma\left(\frac{\mu +\rho +2}{3} \right)  }x^2 (1-x^3)^{\frac{\mu+\rho-1}{3}}
		  \pfq {\frac{\mu}{3}+1, \frac{\rho+1}{3}}{\frac{\mu+\rho+2}{3} }{1-x^3} ,
\end{align*}
where, for the last identity we have used \cite[(15.5.15)]{DLMF}
$$ 
	(c-1)  \, \pFq{2}{1}{a,b}{c-1}{1-z}-(-a+c-1)  \, \pFq{2}{1}{a,b}{c}{1-z}
	= a   \; \pFq{2}{1}{a+1,b}{c}{1-z}.
$$ 
As a result, we obtain the expression \eqref{U1 case C} for $\mathcal{U}_1(x;\mu,\rho)$ and we have 
$$ 
	\int_{0}^1 x^{3n+1}\mathcal{U}_1(x;\mu,\rho)\, \dd x 
	= \frac{   \left(\frac{2}{3}\right)_n \left(\frac{4}{3}\right)_n  }{\left(\frac{\mu +5}{3}\right)_n \left(\frac{\rho +3}{3}\right)_n}
	= (u_1)_{3n+1}, \qquad n\geqslant 0. 
$$ 
When $\mu=0$, then \eqref{C: eq diff U1 lambda} reads as 
$$ 
	 \mathcal{U}_1^{\prime}(x;0,\rho)
	= -\frac{18  \Gamma \left(\frac{\rho }{3}+1\right) }{\Gamma \left(\frac{1}{3}\right) \Gamma
   \left(\frac{\rho -1}{3}\right)} x \left(1-x^3\right)^{\frac{\rho -4}{3}} ,
$$ 
so that 
$$ 
	 \mathcal{U}_1 (x;0,\rho)
	 = K -\frac{9 x^2 \Gamma \left(\frac{\rho }{3}+1\right) }
	 	{\Gamma   \left(\frac{1}{3}\right) \Gamma \left(\frac{\rho -1}{3}\right)} 
   \pFq{2}{1}{\frac{2}{3},\frac{4-\rho }{3}}{\frac{5}{3}}{x^3},  
$$ 
for some integration constant $K$. 
We set $K= \frac{6\Gamma \left(\frac{\rho }{3}+1\right)\Gamma \left(\frac{2 }{3}\right)}{\Gamma \left(\frac{1 }{3}\right)\Gamma \left(\frac{\rho+1 }{3}\right)}$, use 
identity \cite[(15.10.18)]{DLMF} to obtain 
$$ 
	 \mathcal{U}_1 (x;0,\rho)
	 = \frac{6\Gamma \left(\frac{\rho }{3}+1\right)}
		{\Gamma \left(\frac{1}{3}\right)\Gamma\left(\frac{\rho +2}{3} \right)  }x^2 (1-x^3)^{\frac{\rho-1}{3}}
		  \pfq { 1, \frac{\rho+1}{3}}{\frac{\rho+2}{3} }{1-x^3}, 
$$ 
which is the same as \eqref{U1 case C} with $\mu=0$. 
\end{proof}

\subsubsection{Differential equation and the cubic decomposition of $\{P_n(x;\mu,\rho)\}_{n\geqslant 0}$}

Let us assume $\mu>-1$ and $\rho>0$. For each positive integer $n$, the polynomial $P_n(x;\mu,\rho)$ satisfies the differential equation
$$ 
\begin{array}{l}
	\left(x^3-1\right) y^{(3)}(x)+x^2 (\mu +\rho +5) y''(x) \vspace{0.2cm}\\
	+ \frac{1}{8} x \left(8 \mu  \rho +14 \mu -6 n^2-4 n (\mu +\rho +2)-3 (-1)^n (2 \mu -2 \rho +1)+18
   \rho +27\right) y'(x) \vspace{0.2cm}\\
    = \frac{1}{16} n \left(4 \mu +2 n+3 (-1)^n+5\right) \left(2 n-3 (-1)^n+4   \rho +3\right) y(x). 
\end{array}
$$ 
The general solution of the latter equation can be written as 
$$ 
\begin{array}{ll}
	P_n(x;\mu,\rho)
	&= c_1 \,
		\pFq{3}{2}{-\frac{n}{3},\frac{n}{6}+\frac{(-1)^n}{4}+\frac{\mu }{3}+\frac{5}{12},\frac{n}{6}-\frac{(-1)^{n}}{4} +\frac{\rho}{3}+\frac{1}{4}}
		{\frac{1}{3},\frac{2}{3}}
		{x^3}
  \\
   &+\
   c_2 x \ 
   	\pFq{3}{2}{\frac{1}{3}-\frac{n}{3},\frac{n}{6}+\frac{(-1)^n}{4}+\frac{\mu }{3}+\frac{3}{4},\frac{n}{6}-\frac{(-1)^{n}}{4}+\frac{\rho  }{3}+\frac{7}{12}}
	{\frac{2}{3},\frac{4}{3}}
	{x^3}
   \\
   &+\ c_3 x^2 \ 
   \pFq{3}{2}{\frac{2}{3}-\frac{n}{3},\frac{n}{6}+\frac{(-1)^n}{4}+\frac{\mu }{3}+\frac{13}{12},\frac{n}{6}-\frac{(-1)^{n}}{4}+\frac{\rho}{3}+\frac{11}{12}}
   {\frac{4}{3},\frac{5}{3}}
   {x^3}. 
\end{array}
$$ 
which has only one polynomial solution for each $n\geqslant 0$. Similar to the precedent cases, for each positive integer $n$, there is only one (monic) polynomial solution and it differs depending on whether $n$ equals $0,1$ or $2 \bmod {3}$. To be precise, we have 
\begin{eqnarray*}
	P_{3n}(x;\mu,\rho)&\coloneqq& P_n^{[0]}(x^3;\mu,\rho) \\
	&=& \frac{(-1)^n \left(\frac{1}{3}\right)_n \left(\frac{2}{3}\right)_n}{\left(\frac{n}{2}+\frac{1}{4} (-1)^{3
   n}+\frac{\mu }{3}+\frac{5}{12}\right){}_n \left(\frac{n}{2}-\frac{1}{4} (-1)^{3 n}+\frac{\rho
   }{3}+\frac{1}{4}\right){}_n}\\
   &&\cdot \ 	\pFq{3}{2}{-n,\frac{n}{2}+\frac{1}{4} (-1)^{3 n}+\frac{\mu
   }{3}+\frac{5}{12},\frac{n}{2}-\frac{1}{4} (-1)^{3 n}+\frac{\rho
   }{3}+\frac{1}{4}}{\frac{1}{3},\frac{2}{3}}{x^3} , \\
	 P_{3n+1}(x;\mu,\rho)&\coloneqq& x P_n^{[1]}(x^3;\mu,\rho) \\
	 &=&  x \ \frac{(-1)^n \left(\frac{2}{3}\right)_n \left(\frac{4}{3}\right)_n \,}
	 {\left(\frac{n}{2}-\frac{1}{4} (-1)^{3
   n}+\frac{\mu }{3}+\frac{11}{12}\right){}_n \left(\frac{n}{2}+\frac{1}{4} (-1)^{3 n}+\frac{\rho
   }{3}+\frac{3}{4}\right){}_n} \\
   &&  \cdot \ 		\pFq{3}{2}{-n,\frac{n}{2}-\frac{1}{4} (-1)^{3 n}+\frac{\mu
   }{3}+\frac{11}{12},\frac{n}{2}+\frac{1}{4} (-1)^{3 n}+\frac{\rho
   }{3}+\frac{3}{4}}{\frac{2}{3},\frac{4}{3}}{x^3} ,
   	 \\
	P_{3n+2}(x;\mu,\rho)&\coloneqq&x^2 P_n^{[2]}(x^3;\mu,\rho) \\
	&=& x^2\  \frac{(-1)^n \left(\frac{4}{3}\right)_n \left(\frac{5}{3}\right)_n \,}
	{\left(\frac{n}{2}+\frac{1}{4} (-1)^{3
   n}+\frac{\mu }{3}+\frac{17}{12}\right){}_n \left(\frac{n}{2}-\frac{1}{4} (-1)^{3 n}+\frac{\rho
   }{3}+\frac{5}{4}\right){}_n}\\
  	&&  \cdot \ 		\pFq{3}{2}{-n,\frac{n}{2}+\frac{1}{4} (-1)^{3 n}+\frac{\mu
   }{3}+\frac{17}{12},\frac{n}{2}-\frac{1}{4} (-1)^{3 n}+\frac{\rho
   }{3}+\frac{5}{4}}{\frac{4}{3},\frac{5}{3}}{x^3} .
  \end{eqnarray*}
  
  Here, the sequences $\{P_n^{[k]}(\cdot;\mu,\rho)\}_{n\geqslant 0}$, with $k\in\{0,1,2\}$, are precisely the $2$-\OPS s in the cubic decomposition of $\{P_n(\cdot;\mu,\rho)\}_{n\geqslant 0}$. From Lemma \ref{lem: 2sym components rec rel}, these three $2$-\OPS s are not threefold symmetric and satisfy the recurrence relation \eqref{CD rec rel}. 
These coefficients have been computed in  \cite[Tableau 1, 5 and 9 - Case A]{DM92}, for a different choice of the "free" parameter $\gamma_1$. We have included them in the Appendix  
for completeness, where we have used the software {\it Mathematica}. 

\subsubsection{The sequence of derivatives}

The $2$-\OPS\ $\{Q_n(x;\mu,\rho)\coloneqq\frac{1}{n+1} P_{n+1}'(x;\mu,\rho)\}_{n\geqslant 0}$ satisfies the relation \eqref{Qn rec rel sym} with 
\begin{equation}\label{sym gammaQ C rep}
\begin{array}{l}
	\ds \widetilde{\gamma}_{2n } \coloneqq \widetilde{\gamma}_{2n }(\mu,\rho)
				= \frac{2 n (2 n+1) (n+\rho +1)}{(\mu +3 n+2) (3 n+\rho ) (3 n+\rho +3)}
				 , \qquad n\geqslant 1, \vspace{0.3cm}\\
	\ds \widetilde{\gamma}_{2n +1} \coloneqq \widetilde{\gamma}_{2n +1} (\mu,\rho)
				= \frac{2 (n+1) (2 n+1) (\mu +n+2)}{(\mu +3 n+2) (\mu +3 n+5) (3 n+\rho +3)}
				, \qquad n\geqslant 0,   \\
\end{array}
\end{equation}
whose expressions are derived from \eqref{sym gammas} and \eqref{sym gamma C rep}. 
A straightforward comparison between \eqref{sym gamma C rep} and \eqref{sym gammaQ C rep} readily shows that 
$$
	 \widetilde{\gamma}_{n+1} (\mu,\rho) = \gamma_{n+1} (\rho+1,\mu+2), \qquad n\geqslant 0. 
$$
This, combined with the uniqueness of a polynomial sequence defined by the recurrence relation  \eqref{Qn rec rel sym} implies 
$
	Q_n(x;\mu,\rho) = P_n(x;\rho+1,\mu+2) 
$
which means  
\begin{equation}\label{C: D of Pn}
	\frac{1}{n+1} P_{n+1}'(x;\mu,\rho)=P_n(x;\rho+1,\mu+2)
	, \qquad n\geqslant 0.
\end{equation}

\subsubsection{Particular cases} 

The so called {\it Humbert polynomials} introduced in \cite{Hum} correspond to $\{P_n(x;\frac{3\nu-1}{2},\frac{3\nu}{2})\}_{n\geqslant 0}$, up to a scaling of the variable. In fact, by setting $\mu=\frac{3\nu-1}{2}$ and $\rho=\frac{3\nu}{2}$, this $2$-\OPS\ 
satisfies
$$\begin{multlined}
	P_{n+2}(x;\tfrac{3\nu-1}{2},\tfrac{3\nu}{2}) \\
	= x P_{n+1}(x;\tfrac{3\nu-1}{2},\tfrac{3\nu}{2})  - \frac{4}{27} \frac{n (n+1) (3 \nu +n-1)}{(\nu +n-1) (\nu +n) (\nu +n+1)}P_{n-1}(x;\tfrac{3\nu-1}{2},\tfrac{3\nu}{2}) ,
	\end{multlined}
$$
with initial conditions $P_0=1, \ P_{1}(x)=x$ and $P_2(x)=x^2$. Baker \cite[p.60]{Baker} gave explicit expressions for this particular case via ${}_3F_2$ hypergeometric series and these coincide with those obtained above. {\it Pincherle polynomials}, introduced in \cite{Pin}, are a particular case obtained with $\nu=1/2$, which received special attention. For instance, they were discussed in \cite{Mar92} and several properties were further analyzed in \cite{LO} with the focus of the study on the algebraic properties, including generating functions, as well as the analysis of the components arising in the cubic decomposition, but integral representations for the orthogonality measures were not included. 
Prior to this work, Douak and Maroni \cite{DM97 I}  analyzed in detail the case where $\nu=1$ (and therefore $\mu=1$ and $\nu=3/2$), which gives a threefold symmetric sequence with constant $\gamma$-coefficients, namely 
$$
	P_{n+2}(x;1,\tfrac{3}{2}) = x P_{n+1}(x;1,\tfrac{3}{2})   - \tfrac{4}{27} P_{n-1}(x;1,\tfrac{3}{2}) ,
$$
with the same initial conditions. In other words, this means that the associated sequence coincides with the original one, and, regarding the nature of the problem, they referred to this polynomial sequence as {\it Chebyshev type polynomials}. Therein, an integral representation for the orthogonality functionals was given, with support on an interval on the real line (which obviously does not contain all the zeros of the polynomial sequence  \cite[Theorem 4.1]{DM97 I}), which is, from our point of view, a bit artificial. Following Proposition \ref{prop: C rep}, the polynomial sequence $\{P_{n}(x;1,\tfrac{3}{2})\}_{n\geqslant0}$ is $2$-orthogonal with respect to $(u_0,u_1)$ which admit the integral representation  \eqref{int rep for uk classical} with $b=1$ and  
\begin{align*}
&	\mathcal{U}_0(x;1, \tfrac{3}{2}) 
%		=  \frac{3\Gamma \left(\frac{\mu+2}{3}\right)\Gamma \left(\frac{\rho }{3}+1\right)}
%		{\Gamma \left(\frac{1}{3}\right)\Gamma \left(\frac{2}{3}\right)\Gamma\left(\frac{\mu +\rho +2}{3} \right)  }
%	(1-x^3)^{\frac{\mu+\rho-1}{3}}  \ \pfq {\frac{\mu}{3}, \frac{\rho+1}{3}}{\frac{\mu+\rho+2}{3} }{1-x^3}
	=  \frac{9\sqrt{3}}{4\pi} \Bigg( \left( 1+\sqrt{1-x^3}\right)^{1/3}- \left( 1- \sqrt{1-x^3}\right)^{1/3} \Bigg), 
	\\
& 	\mathcal{U}_1(x; 1, \tfrac{3}{2}) %= (\mu +2) U_0'(x)+ x^2 U_0(x)
%	= \frac{3 \Gamma \left(\frac{\mu+5}{3}\right)\Gamma \left(\frac{\rho }{3}+1\right)}
%		{\Gamma \left(\frac{2}{3}\right)\Gamma \left(\frac{4}{3}\right)\Gamma\left(\frac{\mu +\rho +2}{3} \right)  }x^2 (1-x^3)^{\frac{\mu+\rho-1}{3}}
%		  \pfq {\frac{\mu}{3}+1, \frac{\rho+1}{3}}{\frac{\mu+\rho+2}{3} }{1-x^3}
		  =  \frac{27 \sqrt{3}}{8 \pi } \left(\left(\sqrt{1-x^3}+1\right)^{2/3}-\left(1-\sqrt{1-x^3}\right)^{2/3}\right),
\end{align*}
after using \cite[Eq. (15.4.9)]{DLMF}.

Another particular case that made its appearance in the literature corresponds to the case where $\mu=-1/2$ and $\rho=0$, and therefore $\gamma_{n+2}=\widetilde{\gamma}_n=4/27$ and $\gamma_1=4/9$.  In fact, $\{P_n(x;\frac{-1}{2},0)\}_{n\geqslant 0}$  is a particular case of {\it Faber polynomials}, named after the author of \cite{Faber}, and this polynomial sequence with almost constant coefficients has been studied in \cite{DM97 II} and \cite{HeSaff}. The latter paper was essentially devoted to the study of the zeros, while the former was mainly dedicated to the algebraic properties as well as a representation of the measure on an interval on the real line. Here, we combine the two approaches in a more general setting. So,  $\{P_n(x;\frac{-1}{2},0)\}_{n\geqslant 0}$  
is $2$-orthogonal with respect to $(u_0,u_1)$ which admit the integral representation  \eqref{int rep for uk classical} with $b=1$ and  
\begin{align*}
&	\mathcal{U}_0(x;-\tfrac{1}{2},0) 
	=\frac{3 \sqrt{3} \left( \left(1-\sqrt{1-x^3}\right)^{1/3}+\left(1+\sqrt{1-x^3}\right)^{1/3}\right)}{4 \pi  \sqrt{1-x^3}} , 
	\\
& 	\mathcal{U}_1(x; -\tfrac{1}{2},0) 
		  = \frac{9 \sqrt{3} \left(\left(1-\sqrt{1-x^3}\right)^{2/3}+\left(1+\sqrt{1-x^3}\right)^{2/3}\right)}{8 \pi  \sqrt{1-x^3}}. 
\end{align*}

The penultimate and the latter particular cases are related to each other: if we recall \eqref{C: D of Pn} we readily see that 
$$ 
	\frac{1}{n+1} P_{n+1}'(x;-\tfrac{1}{2},0) = P_{n}(x;1,\tfrac{3}{2}),\qquad n\geqslant 0. 
$$ 
It turns out that, the sequence $\{R_n(x;-\frac{1}{2},0)\}_{n\geqslant 0}$ of the associated polynomials of  $\{P_n(x;-\frac{1}{2},0)\}_{n\geqslant 0}$, defined by 
$$ 	
	R_n(x;-\tfrac{1}{2},0) \coloneqq \left\langle u_0, \frac{P_n(x;-\tfrac{1}{2},0)-P_n(t;-\tfrac{1}{2},0)}{x-t} \right\rangle 
$$ 
actually coincides with $\{P_n(x;1,\tfrac{3}{2})\}_{n\geqslant 0}$, as discussed in  \cite{DM97 II}.

\section*{Acknowledgments}

We thank the anonymous referees for their careful reading and for the suggestions and corrections, which led to an improvement of the paper.  The research of WVA was supported by FWO research grant G.0864.16N and EOS project PRIMA 30889451. 

%\appendix

%\section*{References}

\newpage

\section*{Appendix}

As discussed in Lemma \ref{lem: 2sym components rec rel}, the polynomial sequences  $\{P_n^{[j]}(x)\}_{n\geqslant 0}$, with $j=0,1,2$, arising from the cubic decomposition of threefold-symmetric $2$-Hahn classical polynomials $P_n(x)$ are $2$-orthogonal polynomials. The expressions for the corresponding recurrence coefficients can be obtained from the expressions of the $\gamma$-coefficients. So, in the light of Lemma \ref{lem: 2sym components rec rel}, we computed the expressions of the recurrence coefficients for each of the cubic components of the polynomial sequences discussed in Cases $\text{B}_1$ and C, and this is detailed below. We have included them here for a matter of completion, as these polynomial sequences $\{P_n^{[j]}(x)\}_{n\geqslant 0}$ were already described in \cite{DM92} up to a linear change of variable. 

{\small  

\bigskip 

\noindent{\bf Case $\text{B}_1$}\label{Coeffs B1}\\

\noindent\fbox{\parbox{\textwidth}{\begin{align*}
& \beta_{2n}^{[0]} (\mu) = \dfrac{2 (\mu +3 n (3 \mu +9 n (2 \mu +14 n+3)-2)-1)}{3 (\mu +9 n-1) (\mu +9 n+2)} \\
& 	\beta_{2n+1}^{[0]} (\mu) =  \frac{2 (19 \mu +3 n (21 \mu +9 n (2 \mu +10 n+17)+77)+32)}{3 (\mu +9 n+2) (\mu +9 n+8)} \\
&		\alpha_{2n}^{[0]} (\mu) = \frac{4 n (3 n-1) (6 n-1) \left(-3 \mu ^2-5 \mu +702 n^3+27 (8 \mu -9) n^2+3 (6 (\mu -3) \mu -17)
   n+8\right)}{3 (\mu +9 n-4) (\mu +9 n-1)^2 (\mu +9 n+2)} \\
& 	\alpha_{2n+1}^{[0]} (\mu) =  \frac{4 (2 n+1) (3 n+1) (6 n+1) \left(2 \mu +3 \left(\mu ^2+117 n^3+9 (4 \mu +9) n^2+(3 \mu  (\mu
   +6)+5) n\right)-5\right)}{3 (\mu +9 n-1) (\mu +9 n+2)^2 (\mu +9 n+5)} \\
&		\gamma_{2n}^{[0]} (\mu) =  \frac{8 n (2 n+1) (3 n-1) (3 n+1) (6 n-1) (6 n+1) (\mu +3 n-1) (\mu +3 n) (\mu +3 n+1)}{3 (\mu +9
   n-4) (\mu +9 n-1)^2 (\mu +9 n+2)^2 (\mu +9 n+5)} \\
& 	\gamma_{2n+1}^{[0]} (\mu) = \frac{8 (n+1) (2 n+1) (3 n+1) (3 n+2) (6 n+1) (6 n+5)}{3 (\mu +9 n+2) (\mu +9 n+5) (\mu +9 n+8)}
\end{align*}}}

% The polynomials $P_n^{[0]}(x;\mu)$ 

\noindent\fbox{\parbox{\textwidth}{\begin{align*}
& 	\beta_{2n}^{[1]} (\mu) = \dfrac{8 (\mu -1)+6 n (9 \mu +9 n (2 \mu +10 n+7)+5)}{3 (\mu +9 n-1) (\mu +9 n+5)} \\
& 	\beta_{2n+1}^{[1]} (\mu) =  \frac{2 (31 \mu +3 n (27 \mu +9 n (2 \mu +14 n+31)+202)+143)}{3 (\mu +9 n+5) (\mu +9 n+8)} \\
&	\alpha_{2n}^{[1]} (\mu) =  \frac{4 n \left(36 n^2-1\right) \left(-4 \mu +3 n \left(3 (\mu -2) \mu +117 n^2+36 (\mu -1)
   n-10\right)+4\right)}{3 (\mu +9 n-4) (\mu +9 n-1)^2 (\mu +9 n+2)} \\
&	\begin{multlined}
\alpha_{2n+1}^{[1]} (\mu) = \frac{4 (2 n+1) (3 n+1) (3 n+2)}{3 (\mu +9 n+2) (\mu +9 n+5)^2 (\mu +9 n+8)} \\
 \left((\mu +2) (9 \mu +37)+702 n^3+27 (8 \mu +43) n^2+3 (6 \mu 
   (\mu +13)+187) n\right)
   \end{multlined}
\\[0.2cm]
&	\gamma_{2n}^{[1]} (\mu) =   \frac{8 n (2 n+1) (3 n+1) (3 n+2) (6 n-1) (6 n+1)}{3 (\mu +9 n-1) (\mu +9 n+2) (\mu +9 n+5)} \\
& 	\gamma_{2n+1}^{[1]} (\mu) = \frac{8 (n+1) (2 n+1) (3 n+1) (3 n+2) (6 n+5) (6 n+7) (\mu +3 n+1) (\mu +3 n+2) (\mu +3 n+3)}{3
   (\mu +9 n+2) (\mu +9 n+5)^2 (\mu +9 n+8)^2 (\mu +9 n+11)} 
\end{align*}}}

\noindent\fbox{\parbox{\textwidth}{\begin{align*}
&	\beta_{2n}^{[2]} (\mu) = \dfrac{20 (\mu +2)+6 n (15 \mu +9 n (2 \mu +14 n+17)+58)}{3 (\mu +9 n+2) (\mu +9 n+5)} 
\\
&	\beta_{2n+1}^{[2]} (\mu) = \frac{2 (46 \mu +3 n (33 \mu +9 n (2 \mu +10 n+27)+209)+170)}{3 (\mu +9 n+5) (\mu +9 n+11)}
\\
&	\alpha_{2n}^{[2]} (\mu) =  \frac{4 n (3 n+1) (6 n+1) \left(\mu +3 \left(\mu ^2+234 n^3+9 (8 \mu +17) n^2+(6 \mu  (\mu +5)+7)
   n\right)-10\right)}{3 (\mu +9 n-1) (\mu +9 n+2)^2 (\mu +9 n+5)} 
\\
&	\begin{multlined}\alpha_{2n+1}^{[2]} (\mu) =  \frac{4 (2 n+1) (3 n+2) (6 n+5) }{3 (\mu +9 n+2) (\mu +9 n+5)^2 (\mu +9 n+8)} 
	\\
	\left(2 (\mu +2) (3 \mu +10)+351 n^3+54 (2 \mu +11) n^2+3 (3 \mu 
   (\mu +14)+98) n\right)
   \end{multlined}
\\[0.3cm]
&	\gamma_{2n}^{[2]} (\mu) = \frac{8 n (2 n+1) (3 n+1) (3 n+2) (6 n+1) (6 n+5) (\mu +3 n) (\mu +3 n+1) (\mu +3 n+2)}{3 (\mu +9
   n-1) (\mu +9 n+2)^2 (\mu +9 n+5)^2 (\mu +9 n+8)}
\\
&	\gamma_{2n+1}^{[2]} (\mu) =  \frac{8 (n+1) (2 n+1) (3 n+2) (3 n+4) (6 n+5) (6 n+7)}{3 (\mu +9 n+5) (\mu +9 n+8) (\mu +9 n+11)}
\end{align*}}}

\bigskip

\noindent{\bf Case C}\label{app: Case C}\\

\noindent\fbox{\parbox{\textwidth}{\begin{align*}
& \beta _{2n}^{[0]}(\mu,\rho)= \frac{2 (2 n+1) (3 n+1) (6 n+1)}{(9 n+\mu +2) (9
   n+\rho +3)}-\frac{4 n (3 n-1) (6 n-1)}{(9 n+\mu -1) (9 n+\rho -3)} \\
&  \beta _{2n+1}^{[0]}(\mu,\rho)= \frac{4 (n+1) (3 n+2) (6 n+5)}{(9 n+\mu +8) (9
   n+\rho +6)}-\frac{2 (2 n+1) (3 n+1) (6 n+1)}{(9 n+\mu +2) (9 n+\rho +3)} \\
& \begin{multlined}
\alpha _{2n}^{[0]}(\mu,\rho)=  \frac{6 n (6 n-2) (6 n-1) }{(9 n+\mu -4) (9 n+\mu -1)^2 (9 n+\rho -3)^2 (9 n+\rho )} 
 \Bigg((6 n-1) (3 n+\mu -1)
   (3 n+\rho -1) \\ 
   	+\frac{(6 n-3) (9 n+\mu -1) (3 n+\rho -2) (3 n+\rho -1)}{9 n+\rho -6}
	+\frac{(6   n+1) (3 n+\mu -1) (3 n+\mu ) (9 n+\rho -3)}
	{9 n+\mu +2}\Bigg) 
\end{multlined} \\
& \begin{multlined}
 \alpha _{2n+1}^{[0]}(\mu,\rho)
 	= \frac{6 (2 n+1) (6 n+1) }{(9 n+\mu
   +2)^2 (9 n+\mu +5) (9 n+\rho ) (9 n+\rho +3)^2}
   \Bigg(2 (3 n+\mu +1) (3
   n+\rho ) (3 n+1)^2 \\
   +\frac{6 n (3 n+\mu ) (3 n+\mu +1) (9 n+\rho +3) (3 n+1)}{9 n+\mu
   -1}
   +\frac{(3 n+2) (6 n+2) (9 n+\mu +2) (3 n+\rho ) (3 n+\rho +1)}{9 n+\rho +6}\Bigg) 
\end{multlined}\\
&  \gamma _{2n}^{[0]}(\mu,\rho)= \frac{6 n (6 n-2) (6 n-1) (6 n+1) (6 n+2) (6 n+3)
   (\mu +3 n-1) (\mu +3 n) (\mu +3 n+1)}{(\mu +9 n-4) (\mu +9 n-1)^2 (\mu +9 n+2)^2 (\mu +9 n+5)
   (9 n+\rho -3) (9 n+\rho ) (9 n+\rho +3)} \\
& \gamma _{2n+1}^{[0]}(\mu,\rho)= \frac{6 n (6 n-2) (6 n-1) (6 n+1) (6 n+2) (6 n+3)
   (\mu +3 n-1) (\mu +3 n) (\mu +3 n+1)}{(\mu +9 n-4) (\mu +9 n-1)^2 (\mu +9 n+2)^2 (\mu +9 n+5)
   (9 n+\rho -3) (9 n+\rho ) (9 n+\rho +3)} 
\end{align*}}}

\noindent\fbox{\parbox{\textwidth}{\begin{align*}
& \beta _{2n}^{[1]}(\mu,\rho)= \frac{4 (2 n+1) (3 n+1) (3 n+2)}{(9 n+\mu +5) (9
   n+\rho +3)}+\frac{2 n-72 n^3}{(9 n+\mu -1) (9 n+\rho )}  \\
&  \beta _{2n+1}^{[1]}(\mu,\rho)= \frac{2 (n+1) (6 n+5) (6 n+7)}{(9 n+\mu +8) (9
   n+\rho +9)}-\frac{4 (2 n+1) (3 n+1) (3 n+2)}{(9 n+\mu +5) (9 n+\rho +3)} \\
&  \begin{multlined}
 \alpha _{2n}^{[1]}(\mu,\rho)= \frac{6 n (6 n-1) (6 n+1)}{(9 n+\mu -1)^2 (9 n+\mu +2) (9 n+\rho -3)
   (9 n+\rho )^2}
   \Bigg(6 n (3 n+\mu ) (3
   n+\rho -1)
   \\ 
   +\frac{(6 n+2) (9 n+\mu -1) (3 n+\rho ) (3 n+\rho -1)}{9 n+\rho +3}+\frac{(6 n-2) (3
   n+\mu -1) (3 n+\mu ) (9 n+\rho )}{9 n+\mu -4}\Bigg) \\
\end{multlined} \\
& \begin{multlined}
 \alpha _{2n+1}^{[1] }(\mu,\rho)
 	=\frac{6 (2 n+1) }{(9 n+\mu +2) (9 n+\mu +5)^2 (9 n+\rho +3)^2 (9 n+\rho +6)}\\
	\times
	 \Bigg(6 (2 n+1) (3 n+1) (3 n+2) (3 n+\mu +1) (3 n+\rho +1)
   \\ 
   +\frac{(3 n+2) (6 n+1) (6 n+2) (9 n+\mu +5) (3 n+\rho ) (3 n+\rho +1)}{9
   n+\rho } \\
   +\frac{(3 n+1) (6 n+4) (6 n+5) (3 n+\mu +1) (3 n+\mu +2) (9 n+\rho +3)}{9 n+\mu
   +8} \Bigg)
\end{multlined} \\
& \gamma _{2n}^{[1]}(\mu,\rho)= \frac{6 n (6 n-1) (6 n+1) (6 n+2) (6 n+3) (6 n+4) (3
   n+\rho -1) (3 n+\rho ) (3 n+\rho +1)}{(\mu +9 n-1) (\mu +9 n+2) (\mu +9 n+5) (9 n+\rho -3) (9
   n+\rho )^2 (9 n+\rho +3)^2 (9 n+\rho +6)} \\
 &  \gamma _{2n+1}^{[1]}(\mu,\rho)= \frac{6 n (6 n-1) (6 n+1) (6 n+2) (6 n+3) (6 n+4)
   (3 n+\rho -1) (3 n+\rho ) (3 n+\rho +1)}{(\mu +9 n-1) (\mu +9 n+2) (\mu +9 n+5) (9 n+\rho -3)
   (9 n+\rho )^2 (9 n+\rho +3)^2 (9 n+\rho +6)} 
\end{align*}}}

\noindent\fbox{\parbox{\textwidth}{\begin{align*}
&   \beta _{2n}^{[2]}(\mu,\rho)= \frac{2 (2 n+1) (3 n+2) (6 n+5)}{(9 n+\mu +5) (9
   n+\rho +6)}-\frac{4 n (3 n+1) (6 n+1)}{(9 n+\mu +2) (9 n+\rho )} \\
&  \beta _{2n+1}^{[2]}(\mu,\rho)= \frac{4 (n+1) (3 n+4) (6 n+7)}{(9 n+\mu +11) (9
   n+\rho +9)}-\frac{2 (2 n+1) (3 n+2) (6 n+5)}{(9 n+\mu +5) (9 n+\rho +6)} \\
&  \begin{multlined}
 \alpha _{2n}^{[2]}(\mu,\rho)= \frac{6 n (6 n+1) (6 n+2)}{(9 n+\mu -1) (9 n+\mu +2)^2 (9 n+\rho )^2
   (9 n+\rho +3)} \\
   \Bigg((6 n+1) (3 n+\mu )
   (3 n+\rho )\\
   +\frac{(6 n-1) (9 n+\mu +2) (3 n+\rho -1) (3 n+\rho )}{9 n+\rho -3}
   \\
   +\frac{(6 n+3) (3
   n+\mu ) (3 n+\mu +1) (9 n+\rho )}{9 n+\mu +5} \Bigg)
\end{multlined} \\
& \begin{multlined}
 \alpha _{2n+1}^{[2]}(\mu,\rho)= \frac{6 (2 n+1) (6 n+5)}{(9 n+\mu
   +5)^2 (9 n+\mu +8) (9 n+\rho +3) (9 n+\rho +6)^2}\\
    \Bigg(2 (3 n+\mu +2) (3
   n+\rho +1) (3 n+2)^2
   \\
   +\frac{6 (n+1) (9 n+\mu +5) (3 n+\rho +1) (3 n+\rho +2) (3 n+2)}{9 n+\rho
   +9} \\
   +\frac{(3 n+1) (6 n+4) (3 n+\mu +1) (3 n+\mu +2) (9 n+\rho +6)}{9 n+\mu +2} \Bigg)
\end{multlined} \\
&  \gamma _{2n}^{[2]}(\mu,\rho)= \frac{6 n (6 n+1) (6 n+2) (6 n+3) (6 n+4) (6 n+5)
   (\mu +3 n) (\mu +3 n+1) (\mu +3 n+2)}{(\mu +9 n-1) (\mu +9 n+2)^2 (\mu +9 n+5)^2 (\mu +9 n+8)
   (9 n+\rho ) (9 n+\rho +3) (9 n+\rho +6)} \\
& \gamma _{2n+1}^{[2]}(\mu,\rho)=\frac{6 n (6 n+1) (6 n+2) (6 n+3) (6 n+4) (6 n+5)
   (\mu +3 n) (\mu +3 n+1) (\mu +3 n+2)}{(\mu +9 n-1) (\mu +9 n+2)^2 (\mu +9 n+5)^2 (\mu +9 n+8)
   (9 n+\rho ) (9 n+\rho +3) (9 n+\rho +6)} 
\end{align*}}}

}

\end{document}